\documentclass[11pt,a4paper,reqno]{amsart}
\usepackage[applemac]{inputenc}
\usepackage[T1]{fontenc}
\usepackage{amsmath}
\usepackage{amsthm}
\usepackage{amsfonts}
\usepackage{amssymb}
\usepackage{graphicx}
\usepackage{amsbsy}
\usepackage{mathrsfs}
\usepackage{bbm}
\usepackage{extarrows}
\usepackage{color}
\newtheorem{theorem}{Theorem}[section]
\newtheorem{lemma}[theorem]{Lemma}
\newtheorem{proposition}[theorem]{Proposition}

\newtheorem{definition}[theorem]{Definition}
\theoremstyle{remark}
\newtheorem{remark}{Remark}[section]

\numberwithin{equation}{section}
\catcode`@=11
\def\section{\@startsection{section}{1}%
  \z@{1.5\linespacing\@plus\linespacing}{.5\linespacing}%
  {\normalfont\bfseries\large\centering}}
\catcode`@=12

\title[]{Quantized slow blow-up dynamics for the energy-critical corotational wave maps problem}
\author[U. Jeong]{Uihyeon Jeong}
\address{Department of Mathematical Sciences, Korea Advanced Institute of Science and Technology}
\email{juih26@kaist.ac.kr}

\begin{document}

\begin{abstract}
  We study the blow-up dynamics for the energy-critical 1-corotational wave maps problem with 2-sphere target. In \cite{RaphaelRodnianski2012IHES}, Rapha\"el and Rodnianski exhibited a stable finite time blow-up dynamics arising from smooth initial data. In this paper, we exhibit a sequence of new finite-time blow-up rates (quantized rates), which can still arise from well-localized smooth initial data. We closely follow the strategy of the paper \cite{RaphaelSchweyer2014Anal.PDE} by Rapha\"el and Schweyer, who exhibited a similar construction of the quantized blow-up rates for the harmonic map heat flow. The main difficulty in our wave maps setting stems from the lack of dissipation and its critical nature, which we overcome by a systematic identification of correction terms in higher-order energy estimates.
\end{abstract}


\maketitle

%%%%%%%%%%%%%%%%%%%%%%%%%%%%%%%%%%%%%%%%%%%%%%%%%%%%%%%%%%
%%%%%%%%%%%%%%%%%%%%%%%%%%%%%%%%%%%%%%%%%%

\section{Introduction}
\subsection{Wave map problem} For a map $\Phi: \mathbb{R}^{n+1} \to \mathbb{S}^n$, the wave maps problem is given by
\begin{equation}\label{eq:wm2}
  \partial_{tt} \Phi - \Delta \Phi = \Phi (|\nabla \Phi|^2 - |\partial_t \Phi|^2),\quad \vec{\Phi}(t):=(\Phi,\partial_t \Phi)(t) \in \mathbb{S}^n \times T_{\Phi}\mathbb{S}^n.
\end{equation}
\eqref{eq:wm2} has an intrinsic derivation from the following Lagrangian action
\begin{equation}\label{action}
  \frac{1}{2} \int_{\mathbb{R}^{n+1}} (|\nabla \Phi(x,t)|^2 - |\partial_t \Phi (x,t)|^2) dx dt,
\end{equation}
which yields the energy conservation
\begin{equation}\label{energy}
  E(\vec{\Phi}(t))=\frac{1}{2} \int_{\mathbb{R}^n} |\nabla \Phi|^2 + |\partial_t \Phi|^2 dx = E(\vec{\Phi}(0)).
\end{equation}
In particular for the case $n=2$, \eqref{eq:wm2} is called  \emph{energy-critical} since the conserved energy is invariant under the scaling symmetry: if $\vec{\Phi}(t,x)$ is a solution to \eqref{eq:wm2}, then $\vec{\Phi}_{\lambda}(t,x)$ is also a solution to \eqref{eq:wm2} where
\[\vec{\Phi}_{\lambda}(t,x):=\left(\Phi\left(\frac{t}{\lambda}, \frac{x}{\lambda}\right),\frac{1}{\lambda}\partial_t\Phi\left(\frac{t}{\lambda}, \frac{x}{\lambda}\right)\right)\]
and satisfies $E(\vec{\Phi}_{\lambda})=E(\vec{\Phi})$.

When observing a complicated model, it makes sense from a physics perspective to extract the essential dynamics of the problem by reducing the degrees of freedom. Especially for field theories such as \eqref{eq:wm2}, the \emph{geodesic approximation}, that is, a method of approximating the dynamics of the full problem as a geodesic motion over a space of static solutions, is prevalent (see \cite{manton_sutcliffe_2004}). 

To talk about static solutions in more detail, we focus on the solutions that have finite energy. This assumption extends the spatial domain of $\Phi$ to $\mathbb{S}^2$ and allows the \emph{topological degree} of $\Phi$ to be well-defined:
\[k= \frac{1}{|\mathbb{S}^2|} \int_{\mathbb{R}^2} \Phi^* (dw) = \frac{1}{4\pi} \int_{\mathbb{R}^2} \Phi \cdot (\partial_x\Phi \times \partial_y\Phi) dxdy.\] 
Here, $dw$ is the area form on $\mathbb{S}^2$ and $k$ is given only as an integer. We also remark that $k$ is conserved over time. 

We now consider static solutions to \eqref{eq:wm2}:
\begin{equation}\label{eq:static eq}
  \Delta \Phi + \Phi |\nabla \Phi|^2 =0,
\end{equation}
so-called \emph{harmonic maps}. Recall our Lagrangian action \eqref{action}, harmonic maps are characterized as the (local) minimizer of the Dirichlet energy:
\[\frac{1}{2}\int_{\mathbb{R}^2} |\nabla \Phi|^2 dxdy.\] 
Assume the topological degree of a harmonic map $\Phi$ is $k\in \mathbb{Z}$. Then we have the following inequality:
    \begin{align*}
        \frac{1}{2}\int_{\mathbb{R}^2} |\nabla \Phi|^2 dxdy & =\frac{1}{2} \int_{\mathbb{R}^2} | \partial_x\Phi |^2 + |\partial_y\Phi |^2  dxdy \\
        &= \frac{1}{2}\int_{\mathbb{R}^2} | \partial_x\Phi \pm \Phi \times \partial_y \Phi |^2 dxdy \mp  \int_{\mathbb{R}^2} \partial_x \Phi \cdot (\Phi \times \partial_y \Phi) dx dy\\
        & \ge  \pm \int_{\mathbb{R}^2}  \Phi \cdot (\partial_x\Phi \times \partial_y \Phi) dx dy =  4\pi |k|.
    \end{align*}
    Hence in a given topological sector $k$, $\Phi$ satisfies the \emph{Bogomol'ny\u{\i} equation} \cite{Bogomolnyi1976SJNP}
    \begin{equation}\label{eq:Bogomolnyi}
      \partial_x \Phi \pm \Phi \times \partial_y \Phi =0 \quad \textnormal{ for }\pm k \ge 0.
    \end{equation}
    That is, the field equation \eqref{eq:static eq} can be reduced from a second order PDE to a first order PDE. From the stereographic projection, we can see that the equation \eqref{eq:Bogomolnyi} is equivalent to the Cauchy-Riemann equation\footnote{If $k$ is negative, we adopt the conjugate Cauchy-Riemann equation instead of the Cauchy-Riemann equation. Thence, harmonic maps can be represented as rational maps with $\bar{z}$ as a complex variable.}, which clearly identifies the space of harmonic maps as the space of rational maps of degree $k$.

    Under the $L^2$ metric induced naturally from the kinetic energy formula, it is well known that the space of static solutions is \emph{geodesically incomplete}, which leads us to expect a blow-up scenario of low energy problem.
\subsection{Corotational symmetry}
 We consider an ansatz of solutions to \eqref{eq:wm2} with $k$-corotational symmetry: 
\begin{equation}\label{eq:corotational sym}
  \Phi(t,r,\theta) = \begin{pmatrix}
     \sin (u(t,r)) \cos k\theta \\
     \sin (u(t,r)) \sin k\theta \\
     \cos (u(t,r))
  \end{pmatrix} 
\end{equation}
where $(r,\theta)$ are polar coordinates on $\mathbb{R}^2$.

Under $k$-corotational symmetry assumption, $u(t,r)$ satisfies 
\begin{equation}\label{eq:wm1}
  \begin{cases}
    \partial_{tt} u - \partial_{rr}u - \frac{1}{r} \partial_r u   + k^2\frac{f(u)}{r^2} =0, \\
    u_{| t=0} = u_0,\quad \partial_t u_{|t=0}=\dot{u}_0, 
  \end{cases}
  \quad f(u) = \frac{\sin 2u}{2}.
\end{equation}
It is known that the flow \eqref{eq:wm2} preserves such corotational symmetry \eqref{eq:corotational sym} with smooth initial data at least local-in-time, see \cite{RaphaelRodnianski2012IHES}.  

Also, the energy functional \eqref{energy} can be rewritten as
\begin{equation}\label{def:energy functional}
  E(u,\dot{u}) := \pi \int_{0}^{\infty} \left( |\dot{u}|^2 + |\partial_r u|^2 + k^2\frac{\sin^2 u}{r^2} \right) rdr = E(u_0,\dot{u}_0)
\end{equation}
From the above expression, we can observe that a solution to \eqref{eq:wm1} with finite energy must satisfy the following boundary conditions:
\begin{equation}\label{eq:bd condi}
  \lim_{r\to 0} u(r) = m\pi \textnormal{ and } \lim_{r\to \infty} u(r) = n\pi, \quad m,n \in \mathbb{Z}. 
\end{equation}
We have additional symmetries from the geometry of target domain $\mathbb{S}^2$,
\begin{equation}
  -u(t,r),\quad u(t,r) + \pi 
\end{equation}
are also solutions to \eqref{eq:wm1}. Thus, we restrict our solution space to a set of functions $(u,v)$ that have finite energy and satisfy the boundary conditions \eqref{eq:bd condi} with $m=0$ and $n=1$, which provides the local well-posedness of \eqref{eq:wm1} (see also \cite{KlainermanMachedon1993CPAM, KlainermanMachedon1995Duke, Krieger2004CMP, Tao2001CMP,Tataru2005Amer.J.Math.}).
\subsection{Harmonic map}
In this restriction, the harmonic map is uniquely determined (up to scaling) and can be written explicitly as 
\begin{equation}\label{def:Q}
  Q(r)=2 \tan^{-1} r^k.
\end{equation}
Based on the geodesic approximation, it can be said that observing the vicinity of $Q$ under the corotational symmetry assumption facilitates the analysis of blow-up dynamics. This has been proven as a rigorous statement in several past global regularity works (see \cite{ChristodoulouTahvildar-Zadeh1993Comm.PureAppl.Math., ShatahTahvildar-Zadeh1992CPAM, ShatahTahvildar-Zadeh1994Comm.PureAppl.Math., Struwe2003Comm.PureAppl.Math.}).

The above results proved that if a wave map blows up in finite time, such singularity should be created by bubbling off of a non-trivial harmonic map (strictly) inside the backward light cone. 

This statement has inspired other researches studying global behaviors of solutions, and many of the results have been developed based on the existence of nontrivial harmonic map.

Firstly, there is global existence, which is a consequence of the preceding blow-up criteria. If the initial data cannot form a nontrivial harmonic map, that is, if the energy is less than the ground state energy, it can be naturally predicted that the solution exists global in time, and mathematical proof is also contained in the previously mentioned global regularity results.

This study also allows us to consider the problems of energy threshold (see \cite{CoteKenigMerle2008CMP} for the symmetric case and \cite{Tao2009arxiv, SterbenzTataru2010CMP1, SterbenzTataru2010CMP2, KriegerSchlag2012EMS} for the general case). In this case, it is also important to set an appropriate threshold value and the ground state energy is suitable for our problem setting. However for other boundary conditions or other topological degrees, it is often given as an integer multiple of $E(Q,0)$. The heuristic reason is that the degree condition cannot be satisfied with just one bubble (see \cite{CoteKenigLawrieSchlag2015AJM1, LawrieOh2016CMP}). This goes beyond suggesting the existence of a multi-bubbles solution \cite{JendrejLawire2018invent, JendrejLawire2022Anal.PDE, JendrejLawire2023CPAM, JendrejLawire2022MRL, Rodriguez2021Anal.PDE} and serves as an opportunity to soliton resolution conjecture \cite{JendrejLawrie2022, DuyckaertsKenigMartelMerle2022CMP} (see also \cite{Cote2015CPAM, CoteKenigLawrieSchlag2015AJM1, CoteKenigLawrieSchlag2015AJM2, JiaKenig2017AJM}). 

The most recent soliton resolution result \cite{JendrejLawrie2022} fully characterizes the profile decomposition of the solution in all equivariant classes. Thus, our interest is to observe how the scale of the profile given by the harmonic map changes over time within the lifespan of the solution. In particular for the case of low energy, that is, when the energy is slightly greater than the ground state energy, the geodesic approximation discussed earlier leads us to focus on the situation of having only one harmonic map as the blow-up profile.
\subsection{Blow-up near $Q$}
From a methodological perspective, studies investigating the blow-up of a single bubble can be broadly divided into the backward construction starting from Krieger--Schlag--Tataru \cite{KriegerSchlagTataru2008Invent.Math.} and the forward construction inspired from Rodnianski--Sterbenz \cite{RodnianskiSterbenz2010Ann.Math.} and Rapha\"el--Rodnianski \cite{RaphaelRodnianski2012IHES}.

The former work obtained a continuum of blow-up rates for the case $k=1$ via the iteration method and inspired other extended results such as stability under regular perturbations \cite{KriegerMiao2020Duke, KriegerMiaoSchlag2020arxiv} and the construction of more exotic solutions \cite{Pillai2023MAMS, Pillai2023CMP}. Beyond direct extensions of this approach, there is a classification result \cite{JendrejLawrieRodriguez2022Ann.Sci.Ec.Norm.Super.} via configuring radiations appropriately at the blow-up time. These constructions inevitably involves some constraints on regularity and degeneracy of the initial data. 

The latter case adopts a method that accurately describes the initial data set that drives blow-up. Although it is difficult (probably ruled out) to form a family of blow-up rates as in the previous results, the emphasis is on being able to observe the construction of blow-up solutions with smooth initial data. Especially in \cite{RaphaelRodnianski2012IHES}, the authors explicitly describe an initial data set that is open under $H^2$ topology around $Q$ and prove the so-called stable blow-up, in which the solutions starting from that set blow up with a universal rate that slightly misses the self-similar one for all $k\ge 1$.

We note that the initial data set in the above result does not imply a universal blow-up of all well-localized smooth data.
Our main theorem says that there exist other smooth solutions that blow up in finite-time with quantized rates corresponding to the excited regime.

%%%%%%%%%%%%%%%%%%%%%%%%%%%%%%%%%%%%%%%%%%%%
%%%%%%%%%%%%%%%%%%%%%%%%%%%%%%%%%%%%%%%%%%%%

%%%%%%%%%%%%%%%%%%%%%%%%%%%%%%%%%%%%%%%%%%%%%%%%%%

\subsection{Main theorem}
We focus on the solution to \eqref{eq:wm1} with 1-corotational initial data, i.e. $k=1$. Let us restate the stable blow-up result.
\begin{theorem}[Stable blow-up for 1-corotational wave maps \cite{RaphaelRodnianski2012IHES, Kim2023CMP}]\label{thm:stable results}
  There exists a constant $\varepsilon_0>0$ such that for all 1-corotational initial data $(u_0,\dot{u}_0)$ with
  \begin{equation}
    {\left\lVert u_0-Q, \dot{u}_0 \right\rVert}_{\mathcal{H}^2}  < \varepsilon_0,
  \end{equation}
  the corresponding solutions to \eqref{eq:wm1} blow up in finite time $0<T=T(u_0,\dot{u}_0)<\infty$ as follows: for some $(u^*,\dot{u}^*) \in \mathcal{H}$,
  \begin{equation}\label{eq:strong convergence}
    {\left\lVert u(t,r)-Q\left(\frac{r}{\lambda(t)} \right)-u^*, \partial_t u(t,r) - \dot{u}^* \right\rVert}_{\mathcal{H}} \to 0 \quad \textnormal{as } t\to T
  \end{equation}
  with the universal blow up speed:
  \begin{equation}\label{eq:stable speed}
    \lambda(t)=2e^{-1}(1+o_{t\to T}(1))(T-t) e^{-\sqrt{|\log(T-t)|}}.
  \end{equation}
  Here, $\mathcal{H}$, $\mathcal{H}^2$ are given by \eqref{def:H sobolev}, \eqref{def:H2 sobolev}.
\end{theorem}
\begin{remark}[$1$-corotational symmetry]
  In \cite{RaphaelRodnianski2012IHES}, the authors mentioned that the nature of the harmonic map, which varies depending on whether $k$ equals to $1$ or not, leads to distinctive blow-up rates. As a result of the logarithmic calculation that occurs additionally only when $k=1$, the universality of the blow-up rate in this case was unclear. The sharp constant $2e^{-1}$ in \eqref{eq:stable speed} was later obtained by Kim \cite{Kim2023CMP} using a refined modulational analysis.

  Nevertheless, the slow decaying nature of the harmonic map is rather an advantage in our analysis, which allows us to exhibit the following smooth blow-up with the quantized blow-up rates corresponding to the excited regime.
\end{remark}

\begin{theorem}[Quantized blow-up for 1-corotational wave maps]\label{thm:main} For a natural number $\ell \ge 2$ and an arbitrarily small constant $\varepsilon_0 >0$, there exists a smooth 1-corotational initial data $(u_0,\dot{u}_0)$ with 
  \begin{equation}
    {\left\lVert u_0-Q, \dot{u}_0 \right\rVert}_{\mathcal{H}}  < \varepsilon_0
  \end{equation}
  such that the corresponding solution to \eqref{eq:wm1} blows up in finite time $0<T=T(u_0,\dot{u}_0)<\infty$ and satisfies \eqref{eq:strong convergence} with the quantized blow up speed:
\begin{equation}\label{eq:lambda}
  \lambda(t)=c(u_0,\dot{u}_0)(1+o_{t\to T}(1)) \frac{(T-t)^{\ell}}{|\log (T-t)|^{\ell/(\ell-1)}}, \quad c(u_0,\dot{u}_0)>0.
\end{equation}
%Moreover, $(u^*,\dot{u}^*)$ has an additional $\mathcal{H}\cap (\dot{H}^{\ell}\times \dot{H}^{\ell-1})$ regularity.
\end{theorem}
\begin{remark}[Further regularity of asymptotic profile]\label{rmk:further regularity}
  The asymptotic profile $(u^*,\dot{u}^*)$ also has $\dot{H}^{\ell}\times \dot{H}^{\ell-1}$ regularity in the sense that certain $\ell$-fold (resp., $\ell-1$-fold) derivatives of $u^*$ (resp., $\dot{u}^*$) belong to $L^2$. This is a consequence of the fact that the $\ell$-th order energy of the radiative part of the solution satisfies the scaling invariance bound ($\mathcal{E}_{\ell} \le C \lambda^{2(\ell-1)}$; see \eqref{eq:scale invariance bound}) similarly as in \cite{RaphaelSchweyer2014Anal.PDE}.
\end{remark}
\begin{remark}[Quantized blow-up]
  The existence of (type-II) blow-up solutions with quantized blow-up rates has also been well studied in parabolic equations, especially for nonlinear heat equations. Starting with the discovery of formal mechanisms \cite{HerreroVelazquez1994CRASPSI,FilippasHerreroVelazquez2000RSLPSAPES}, there are classification works \cite{Mizoguchi2004AdvDE, Mizoguchi2007Math.Ann.} in the energy-supercritical regime.  The proofs in this literature are based on maximum principle (cf. \cite{MatanoMerle2009JFA, MatanoMerle2004CPAM}).

  Through modulational analysis, not relying on maximum principle, there have been some (type-II) quantized rate constructions in the critical parabolic equations such as \cite{RaphaelSchweyer2014Anal.PDE, HadzicRaphael2019JEMS} for the energy-critical case and \cite{CollotGhoulMasmoudiNguyen2022CPAM} for the mass-critical case. See also the works \cite{DelpinoMussoWeiZhangZhang2020arxiv,Junichi2020Ann} relying on the inner-outer gluing method. In \cite{RaphaelSchweyer2014Anal.PDE}, the authors expected that their modulation technique is robust enough to be propagated to dispersive models including the wave maps problem, and the quantized rate constructions have been established in the energy-supercritical dispersive equations \cite{MerleRaphaelRodnianski2015Camb.J.Math., Collot2018MEM.AMS, GhoulIbrahimNguyen2018JDE}. Up to our knowledge, Theorem \ref{thm:main} provides the first rigorous quantized rate constructions for energy-critical dispersive equations. We expect that our analysis can also be extended to other energy-critical dispersive equations such as the nonlinear wave equation. 
\end{remark}

%\begin{itemize}
 % \item by KST and JLR, no quantized blow-up rate for (possibly) non-smooth data
  %\item parabolic : instantaneous smoothing property for parabolic case, expect quantized blow-up 
%\end{itemize}

\begin{remark}[Instability of blow-up]
  In contrast to Theorem \ref{thm:stable results}, our initial data set is of codimension $\ell-1$, similar to \cite{RaphaelSchweyer2014Anal.PDE}, due to unstable directions inherent in the ODE system driving the blow-up dynamics. This similarity follows from the fact that the wave map problem and the harmonic map heat flow share the same ground states and linearized Hamiltonian under the 1-corotational symmetry. We also expect the stability formulated by constructing a smooth manifold of the initial data set.
  %\begin{itemize}
   % \item k$\ge $2 : no finite time blow-up (H dot 1 neighbor hood near Q), k$\ge 3$ : asymptotic stability
   % \item wave maps: stable blow-up results for all k
  %\end{itemize}
\end{remark}

%%%%%%%%%%%%%%%%%%%%%%%%%%%%%%%%%%%%%%%%%%%%%%%%%%

\subsection{Notation} We introduce some notation needed for the proof before going into the strategy of the proof.
 We first use the bold notation for vectors in $\mathbb{R}^2$:
\begin{equation}
  \boldsymbol{u}:=\begin{pmatrix}
    u \\
    \dot{u}
    \end{pmatrix}, \quad  \boldsymbol{u}(r):=\begin{pmatrix}
      u(r) \\
      \dot{u}(r)
      \end{pmatrix}.
\end{equation}
For $\lambda >0$, the $\dot{H}^1 \times L^2$ scaling is defined by:
  \begin{equation}
  \boldsymbol{u}_{\lambda}(r)=\begin{pmatrix}
    u_{\lambda}(r) \\
    \lambda^{-1} \dot{u}_{\lambda}(r )
    \end{pmatrix}  :=\begin{pmatrix}
  u(y) \\
  \lambda^{-1} \dot{u}(y )
  \end{pmatrix} , \quad y:=\frac{r}{\lambda}
  \end{equation}
  and the corresponding generator is denoted by
  \begin{equation}
    \boldsymbol{\Lambda} \boldsymbol{u} := \begin{pmatrix}
      \Lambda u \\
      \Lambda_0 \dot{u}
      \end{pmatrix}
      := -\frac{d \boldsymbol{u}_{\lambda}(r)}{d\lambda}\Big{\vert}_{\lambda=1}  = \begin{pmatrix}
      r\partial_r u(r) \\
      \left( 1 +r\partial_r \right)\dot{u}(r)
      \end{pmatrix} .
  \end{equation}
In general, we employ the $\dot{H}^{k}$ scaling generator
\begin{equation}
  \Lambda_{k} u:= - \frac{d}{d\lambda} \left( \lambda^{k-1} u_{\lambda}(r) \right)\Big{\vert}_{\lambda=1}  = (-k+1+r\partial_r)u(r).
\end{equation}
 We now reformulate \eqref{eq:wm1} using the vector-valued function $\boldsymbol{F}:\mathbb{R}^2 \to \mathbb{R}^2$:
\begin{equation}\label{eq:WM}
  \begin{cases}
    \partial_{t} \boldsymbol{u} =\boldsymbol{F}(\boldsymbol{u}), \\  \boldsymbol{u}_{|t=0}=\boldsymbol{u}_0, 
  \end{cases}\quad \boldsymbol{u}=\boldsymbol{u}(t,r) ,\quad \boldsymbol{F}(\boldsymbol{u}):=\begin{pmatrix}
    \dot{u} \\ \Delta u-\frac{1}{r^2}f(u)
  \end{pmatrix}
  \end{equation}
  where $\Delta = \partial_{rr} + \frac{1}{r}\partial_r$.

We use two subsets of the real line
\[\mathbb{R}_+=\{r \in \mathbb{R}: x\ge 0\},\quad \mathbb{R}_+^*=\{r \in \mathbb{R}: x >0\}.\]
  We denote by $\chi$ a $C^{\infty}$ radial cut-off function on $\mathbb{R}_+$: 
  \[\chi(r)=\begin{cases}
    1 & \textnormal{for } r\le 1 \\
    0 & \textnormal{for } r\ge 2
  \end{cases}.\]
   We let $\chi_B(r):=\chi(r/B)$ for $B>0$. Similarly, we denote by $\mathbf{1}_A(y)$ as the indicator function on the set $A$. In particular, $\mathbf{1}_{B\le y \le 2B}$ will be rewritten by $\mathbf{1}_{y\sim B}$, or simply $\mathbf{1}_B$ abusively. The cut-off boundary $B$ will often be chosen as the constant multiples of   
  \begin{equation}\label{def:B_0 B_1}
  B_0:=\frac{1}{b_1}, \quad B_1:=\frac{|\log b_1|^{\gamma}}{b_1}, \quad b_1>0.
  \end{equation}
  Later, we will choose $\gamma=1+\overline{\ell}$ where $\ell$ appeared from Theorem \ref{thm:main}. Here, we denote the remainder of dividing $i$ by $2$ as $\overline{i}$ i.e. $\overline{i}=i \mod 2$ for an integer $i$. We also denote $L=\ell+\overline{\ell+1}$ i.e. $L$ is the smallest odd integer greater than or equal to $\ell$. We also abuse the indicator notation $\mathbf{1}_{\{l \ge m \}}$ as
  \[\mathbf{1}_{\{l \ge m \}}=\begin{cases}
    1 & \textnormal{if }l\ge m \\
    0 & \textnormal{if }l<m
  \end{cases} ,\quad l,m\in \mathbb{Z}.\]

  We adopt the following $L^2(\mathbb{R}^2)$ inner product for radial functions $u,v$:
 \begin{equation*}
  \langle u, v \rangle := \int_{0}^{\infty} u(r)v(r) rdr
 \end{equation*}
 and $L^2\times L^2$ inner product for vector-valued functions $\boldsymbol{u},\boldsymbol{v}$:
 \begin{equation}\label{def:vector inner product}
  \langle \boldsymbol{u}, \boldsymbol{v} \rangle := \langle u, v \rangle + \langle \dot{u}, \dot{v} \rangle
 \end{equation}
 We introduce two Sobolev spaces $\mathcal{H}$ and $\mathcal{H}^2$ with the following norms:
\begin{align}
  {\lVert u,\dot{u} \rVert}_{\mathcal{H}}^2  &:= \int |\partial_y u|^2 + \frac{|u|^2}{y^2} + |\dot{u}|^2, \label{def:H sobolev}\\
  {\lVert u,\dot{u} \rVert}_{\mathcal{H}^2}^2  &:= {\lVert u,\dot{u} \rVert}_{\mathcal{H}}^2 + \int   |\partial_y^2 u|^2 + |\partial_y\dot{u}|^2 + \frac{|\dot{u}|^2}{y^2} + \int_{|y|\le 1} \frac{1}{y^2} \left( \partial_y u - \frac{u}{y} \right)^2 \label{def:H2 sobolev}
\end{align}
where the above shorthand for integrals is given by $\int = \int_{\mathbb{R}^2}$.

  For any $x:=(x_1,\dots,x_n)\in \mathbb{R}^n$, we set $|x|^2 = x_1^2+ \cdots + x_n^2$ and 
  \[ \mathcal{B}^n:=\left\{x\in \mathbb{R}^n,   |x|\leq 1 \right\}, \quad \mathcal{S}^n:=\partial \mathcal{B}^n=\left\{x\in \mathbb{R}^n,   |x|= 1 \right\}.\]
  We use the Kronecker delta notation: $\delta_{ij}=1$ for $i=j$ and $\delta_{ij}=0$ for $i\neq j$. 
%%%%%%%%%%%%%%%%%%%%%%%%%%%%%%%%%%%%%%%%%%%%
%%%%%%%%%%%%%%%%%%%%%%%%%%%%%%%%%%%%%%%%%%%%

\subsection{Strategy of the proof}
Our proof is based on the general modulational analysis scheme developed by Rapha\"el--Rodnianski \cite{RaphaelRodnianski2012IHES}, Merle--Rapha\"el--Rodnianski \cite{MerleRaphaelRodnianski2013Invent.Math.} and Rapha\"el--Schweyer \cite{RaphaelSchweyer2014Anal.PDE}, which also have difficulties arising from energy-critical nature and the small equivariance index, including logarithmic computations. We closely follow the main strategy of \cite{RaphaelSchweyer2014Anal.PDE}. However, notable differences stem from the lack of dissipation in the higher-order ($H^{L+1}$, $L \gg 1$) energy estimates due to the dispersive nature of our problem. We overcome this difficulty by carefully correcting the higher-order energy functional to uncover the repulsive property (to identify terms with good sign), generalizing the computation in the $H^2$ energy estimates of \cite{RaphaelRodnianski2012IHES}.

Given an odd integer $L\ge 3$, we first construct the blow-up profile $\boldsymbol{Q}_b$ of the form 
\begin{equation}
     \boldsymbol{Q}_b := \boldsymbol{Q} + \boldsymbol{\alpha}_b:= \begin{pmatrix}
    Q \\ 0    
  \end{pmatrix} + \sum_{i=1}^L b_i \boldsymbol{T}_i + \sum_{i=2}^{L+2} \boldsymbol{S}_i 
\end{equation}
where $b=(b_1,\dots,b_L)$ is a set of modulation parameters and $\boldsymbol{T}_i$, $\boldsymbol{S}_i$ are deformation directions so that $(\boldsymbol{Q}_{b(t)})_{\lambda(t)}$ solves \eqref{eq:WM} approximately. Equivalently, $\boldsymbol{Q}_b$ satisfies
\begin{equation}\label{eq:approx eqn}
  \partial_s \boldsymbol{Q}_b-\boldsymbol{F}(\boldsymbol{Q}_b) -\frac{\lambda_s}{\lambda}\boldsymbol{\Lambda Q}_b \approx 0 ,\quad   \frac{ds}{dt} = \frac{1}{\lambda(t)}.
\end{equation}
From the imposed relations \eqref{eq:approx eqn}, the blow-up dynamics is determined by the evolution of the modulation parameters $b=(b_1,\dots,b_L)$. The leading dynamics of $b$ and $\boldsymbol{T}_i$ are determined by considering the linearized flow of \eqref{eq:approx eqn} near $\boldsymbol{Q}$:
\begin{align}\label{eq:approx linearized eqn}
  0\approx \partial_s \boldsymbol{Q}_b-\boldsymbol{F}(\boldsymbol{Q}_b) -\frac{\lambda_s}{\lambda} \boldsymbol{\Lambda Q}_b &= \partial_s (\boldsymbol{Q}_b-\boldsymbol{Q})-\boldsymbol{F}(\boldsymbol{Q}_b) + \boldsymbol{F}(\boldsymbol{Q}) -\frac{\lambda_s}{\lambda} \boldsymbol{\Lambda Q}_b  \nonumber \\
  &\approx \partial_s \boldsymbol{\alpha}_b + \boldsymbol{H}\boldsymbol{\alpha}_b -\frac{\lambda_s}{\lambda} \boldsymbol{\Lambda }(\boldsymbol{Q} + \boldsymbol{\alpha}_b)
\end{align}
where $\boldsymbol{H}$ denotes the linearized Hamiltonian 
\begin{equation}
  \boldsymbol{H}:= \begin{pmatrix}
    0 & -1 \\
    H & 0
\end{pmatrix}, \quad H= -\Delta + \frac{f'(Q)}{y^2}.
\end{equation}
After defining $\boldsymbol{T}_i$ inductively
\begin{equation}
  \boldsymbol{H}\boldsymbol{T}_{i+1}=- \boldsymbol{T}_i, \quad \boldsymbol{T}_0:=\boldsymbol{\Lambda Q},
\end{equation}
\eqref{eq:approx linearized eqn} and asymptotics $\boldsymbol{\Lambda T}_i \sim (i-1)\boldsymbol{T}_i$ yield the leading dynamics of $b$:
\begin{equation}\label{eq:mod leading ode}
 -\frac{\lambda_s}{\lambda}=b_1,\quad (b_k)_s=b_{k+1}-\left(k-1\right)b_1b_k,\quad b_{L+1}:=0,\quad 1\le k \le L.
\end{equation}
%Ignoring $\boldsymbol{S}_i$, 
%\begin{align}
 % (b_1)_s \boldsymbol{T}_1 + \boldsymbol{H} (b_2 \boldsymbol{T}_2) + b_1^2 \boldsymbol{\Lambda T}_1  =0 &  \Rightarrow (b_1)_s -b_2=0 \\
  %\vdots & \nonumber \\
  %(b_i)_s \boldsymbol{T}_i + \boldsymbol{H} (b_{i+1} \boldsymbol{T}_{i+1}) + b_1b_{i} \boldsymbol{\Lambda T}_i  =0 & \Rightarrow (b_i)_s - b_{i+1} + (i-1)b_1b_i=0 \\
  %\vdots & \nonumber \\
  %(b_L)_s \boldsymbol{T}_i  + b_1b_{L} \boldsymbol{\Lambda T}_L  =0 & \Rightarrow (b_L)_s  + (L-1)b_1b_L=0
%\end{align}
$\boldsymbol{S}_i$ appears to correct \eqref{eq:approx linearized eqn} to \eqref{eq:approx eqn} containing some radiative terms from the difference $\boldsymbol{\Lambda T}_i -(i-1) \boldsymbol{T}_i$ and the nonlinear effect from $\boldsymbol{F}(\boldsymbol{Q}_b) - \boldsymbol{F}(\boldsymbol{Q}) + \boldsymbol{H}\boldsymbol{\alpha}_b$. Then $b$ drives the following ODE system 
\begin{equation}\label{eq:mod ode}
  (b_k)_s=b_{k+1}-\left(k-1+\frac{1}{(1+\delta_{1k})\log s}\right)b_1b_k,\quad b_{L+1}:=0,\quad 1\le k \le L.
\end{equation}
We then choose a special solution of \eqref{eq:mod ode}:
\begin{equation}
  b_1(s) \sim \frac{\ell}{\ell-1} \left( \frac{1}{s} - \frac{(\ell-1)^{-1}}{s\log s} \right),
\end{equation}
which leads to \eqref{eq:lambda} from the relations $-\lambda_t=b_1$ and $\frac{ds}{dt}=\frac{1}{\lambda}$. Since the special solution we choose is formally codimension $\ell-1$ stable, we control the unstable directions in the vicinity of these special solutions to ODE system \eqref{eq:mod ode} by Brouwer's fixed point theorem.

Now, we decompose the solution $\boldsymbol{u}=\boldsymbol{u}(t,r)$ to \eqref{eq:WM} as follows
\begin{equation}\label{eq:decomposition and orthogonality}
  \boldsymbol{u}= ( \boldsymbol{Q}_{b(t)} + \boldsymbol{\varepsilon})_{\lambda(t)}= (\boldsymbol{Q}_{b(t)})_{\lambda(t)} + \boldsymbol{w},\quad \langle \boldsymbol{H}^{i}\boldsymbol{\varepsilon},  \boldsymbol{\Phi}_M \rangle = 0, \quad 0\le i \le L 
\end{equation}
where $\boldsymbol{\Phi}_M$ is defined in \eqref{def:PhiM}.
The orthogonality conditions in \eqref{eq:decomposition and orthogonality} uniquely determine the decomposition by the implicit function theorem. Then we derive the evolution equation of $\boldsymbol{\varepsilon}$ from \eqref{eq:WM}, which contains the formal modulation ODE \eqref{eq:mod ode} with some errors in terms of $\boldsymbol{\varepsilon}$. 

To justify the formal modulation ODE \eqref{eq:mod ode}, we need sufficient smallness of $\boldsymbol{\varepsilon}$ and we need to propagate it. For this purpose, we consider the higher-order energy associated to the linearized Hamiltonian $H$:
\begin{equation}\label{eq:higher energy}
  \mathcal{E}_{L+1} = \langle H^{\frac{L+1}{2}}\varepsilon, H^{\frac{L+1}{2}}\varepsilon\rangle + \langle H H^{\frac{L-1}{2}}\dot{\varepsilon}, H^{\frac{L-1}{2}}\dot{\varepsilon} \rangle    .
\end{equation} This energy is coercive thanks to the orthogonality conditions in \eqref{eq:decomposition and orthogonality}. 

Thus, our analysis boils down to estimating the time derivative of $\mathcal{E}_{L+1}$. Unlike in \cite{RaphaelSchweyer2014Anal.PDE}, we cannot employ dissipation to control the time derivative of $\mathcal{E}_{L+1}$ due to the dispersive nature of our problem. Instead, we use the repulsive property of the (super-symmetric) conjugated Hamiltonian $\widetilde{H}$ of $H$ observed in \cite{RodnianskiSterbenz2010Ann.Math.} and \cite{RaphaelRodnianski2012IHES}. To illuminate the repulsive property in the energy estimate, we consider the linearized flow in terms of $w$ from $\boldsymbol{w}=(w,\dot{w})$ and the well-known factorization:
\[{w}_{tt} + H_{\lambda} w=0,\quad  H_{\lambda} =A_{\lambda}^* A_{\lambda},\quad A_{\lambda}=-\partial_r + \frac{\sin Q_{\lambda}}{r}.\]
Defining the higher-order derivatives adapted to $H_{\lambda}$ and its corresponding operator
\[w_k:=\mathcal{A}_{\lambda}^k w,\quad \mathcal{A}_{\lambda} = A_{\lambda},\; \mathcal{A}_{\lambda}^2 = A_{\lambda}^*A_{\lambda}, \; \cdots,  \; \mathcal{A}_{\lambda}^k = \underbrace{\cdots A_{\lambda}^*A_{\lambda} A_{\lambda}^* A_{\lambda}}_{k \textnormal{ times}},\]
 the higher-order energy \eqref{eq:higher energy} can essentially be written as follows:
 \begin{align*}
  {\mathcal{E}_{L+1}} &\approx \lambda^{2L}(\langle w_{L+1},w_{L+1} \rangle + \langle \partial_t w_L, \partial_t w_L \rangle)\\
  &=\lambda^{2L}(\langle \widetilde{H}_{\lambda} w_{L},w_{L} \rangle + \langle \partial_t w_L, \partial_t w_L \rangle)
 \end{align*}
%\[{\mathcal{E}_{L+1}} \approx \lambda^{2L}(\langle w_{L+1},w_{L+1} \rangle + \langle \partial_t w_L, \partial_t w_L \rangle)=\lambda^{2L}(\langle \widetilde{H}_{\lambda} w_{L},w_{L} \rangle + \langle \partial_t w_L, \partial_t w_L \rangle)\]
where $\widetilde{H}_{\lambda}=A_{\lambda} A_{\lambda}^*$ is the conjugated Hamiltonian of $H_{\lambda}$. As an advantage of the adoption of the Leibniz rule notation between an operator and a function
\[\partial_t (P f) = \partial_t(P)f + P f_t,\quad \partial_t(P):=[\partial_t, P],\]
we can express the energy estimate for $\mathcal{E}_{L+1}$ succinctly:
\begin{align*}
  \frac{d}{dt} \left\{ \frac{\mathcal{E}_{L+1}}{2\lambda^{2L}} \right\} & \approx \frac{1}{2}\langle \partial_t(\widetilde{H}_{\lambda}) w_{L},w_{L} \rangle + \langle \widetilde{H}_{\lambda} w_{L},\partial_t w_{L} \rangle + \langle \partial_{tt} w_L, \partial_t w_L \rangle \\
  &\approx  \frac{1}{2}\langle \partial_t(\widetilde{H}_{\lambda}) w_{L},w_{L} \rangle + 2 \langle \partial_{t} w_L, \partial_t (\mathcal{A}_{\lambda}^L)w_t \rangle.
\end{align*}
Integrating by parts in time the second term, we get
\begin{align*}
  \frac{d}{dt} \left\{ \frac{\mathcal{E}_{L+1}}{2\lambda^{2L}} - 2  \langle  w_L, \partial_t (\mathcal{A}_{\lambda}^L)w_t \rangle\right\}
  &\approx  \frac{1}{2}\langle \partial_t(\widetilde{H}_{\lambda}) w_{L},w_{L} \rangle + 2 \langle  w_L, \partial_t (\mathcal{A}_{\lambda}^L)w_{2} \rangle.
\end{align*}
%In \cite{RaphaelRodnianski2012IHES}, the fact that $L=1$ made the repulsive property directly observable:
%the authors directly obtained the repulsive property with the advantage of $L=1$:
In \cite{RaphaelRodnianski2012IHES}, the authors exhibited the repulsive property by directly calculating the following identity with the advantage of $L=1$: 
\[\langle  w_1, \partial_t (\mathcal{A}_{\lambda})w_{2} \rangle=\frac{1}{2}\langle \partial_t(\widetilde{H}_{\lambda}) w_{1},w_{1} \rangle \le 0\]
However, this computation does not seem to be directly extended to our case $L\ge 3$. We overcome this problem by first writing $\mathcal{A}_{\lambda}^L = \widetilde{H}_{\lambda}\mathcal{A}_{\lambda}^{L-2}$ and pulling out the repulsive term using Leibniz rule
\begin{align*}
  \langle  w_L, \partial_t (\mathcal{A}_{\lambda}^L)w_{2} \rangle &= \langle  w_L, \partial_t (\widetilde{H}_{\lambda})w_{L} \rangle + \langle \widetilde{H}_{\lambda} w_L, \partial_t (\mathcal{A}_{\lambda}^{L-2})w_{2} \rangle\\
  & \approx  \langle  w_L, \partial_t (\widetilde{H}_{\lambda})w_{L} \rangle - \langle \partial_{tt} w_L, \partial_t (\mathcal{A}_{\lambda}^{L-2})w_{2} \rangle.
\end{align*}
Again integrating by parts in time, we obtain
\begin{align*}
  &\frac{d}{dt} \left\{ \frac{\mathcal{E}_{L+1}}{2\lambda^{2L}} - 2 ( \langle  w_L, \partial_t (\mathcal{A}_{\lambda}^L)w_t \rangle -  \langle \partial_{t} w_L, \partial_t (\mathcal{A}_{\lambda}^{L-2})w_{2} \rangle + \langle  w_L, \partial_t (\mathcal{A}_{\lambda}^{L-2})\partial_t w_{2} \rangle) \right\}\\ 
  &\approx  \frac{5}{2}\langle \partial_t(\widetilde{H}_{\lambda}) w_{L},w_{L} \rangle + 2 \langle  w_L, \partial_t (\mathcal{A}_{\lambda}^{L-2})w_{4} \rangle.
\end{align*}
Repeating the above correction procedure, we arrive at the term with good sign:
\begin{align*}
  \frac{d}{dt} \left\{ \frac{\mathcal{E}_{L+1}}{2\lambda^{2L}} + \textnormal{corrections} \right\}
  &\approx  \frac{2L-1}{2}\langle \partial_t(\widetilde{H}_{\lambda}) w_{L},w_{L} \rangle + 2 \langle  w_L, \partial_t (\mathcal{A}_{\lambda})w_{L+1} \rangle \\
  &\approx  \frac{2L+1}{2}\langle \partial_t(\widetilde{H}_{\lambda}) w_{L},w_{L} \rangle \le 0.
\end{align*}

 In the actual energy estimate, there are also error terms such as the profile equation error and nonlinear terms in $\boldsymbol{\varepsilon}$. For these nonlinear terms, we also estimate the intermediate energies $\mathcal{E}_k$, which can be defined similarly to $\mathcal{E}_{L+1}$. %Especially for $\mathcal{E}_{\ell}$, we detect subtle corrections arising from a different criticality than $\mathcal{E}_{L+1}$.
% We point out that these replaced derivative operators depend on time, which requires more systematic calculations in our case $L\ge 3$ where $H_{\lambda}$ is taken multiple times.  
%%%%%%%%%%%%%%%%%%%%%%%%%%%%%%%%%%%%%%%%%%%%
%%%%%%%%%%%%%%%%%%%%%%%%%%%%%%%%%%%%%%%%%%%%

\subsection*{Organization of the paper} 
In section 2, we construct the approximate blow-up profile with the description of the ODE dynamics of the modulation equations. Section 3 is devoted to the decomposition of the solution into the blow-up profile constructed in the previous section and the remaining error. We also introduce the bootstrap setting to control the error and establish a Lyapunov-type monotonicity for the higher-order energy with respect to such error. Section 4 provides the proof of Theorem \ref{thm:main} by closing the bootstrap with some standard topological arguments. 

\subsection*{Acknowledgements}
The author appreciates Kihyun Kim and Soonsik Kwon for helpful discussions and suggestions for this work. The author is partially supported by the National Research Foundation of Korea (NRF) grant funded by the Korea government (MSIT) (NRF-2019R1A5A1028324 and NRF-2022R1A2C109149912).

\section{Construction of the approximate solution}
In this section, we construct the approximate blow-up profile $\boldsymbol{Q}_b$, represented by a deformation of the harmonic map $\boldsymbol{Q}$ through modulation parameters $b=(b_1,\dots,b_L)$. We also derive formal dynamical laws of $b$, which leads to our desired blow-up rate. 

%%%%%%%%%%%%%%%%%%%%%%%%%%%%%%%%%%%%%%%%%%%%
%%%%%%%%%%%%%%%%%%%%%%%%%%%%%%%%%%%%%%%%%%%%

\subsection{The linearized dynamics}
It is natural to look into the linearized dynamics of our system near the stationary solution $Q$. Let $\boldsymbol{u}=\boldsymbol{Q}+\boldsymbol{\varepsilon}$ where $\boldsymbol{Q}=(Q,0)^t$ and $\boldsymbol{u}$ is the solution to \eqref{eq:WM}. Then $\boldsymbol{\varepsilon}$ satisfies
\begin{align*}
  \partial_t \boldsymbol{\varepsilon} &= \boldsymbol{F}(\boldsymbol{Q}+\boldsymbol{\varepsilon})-\boldsymbol{F}(\boldsymbol{Q}) \\
  & =\begin{pmatrix}
    \dot{\varepsilon} \\
    \Delta \varepsilon - \frac{1}{r^2} (f(Q+\varepsilon) - f(Q))
  \end{pmatrix} \\
  & =\begin{pmatrix}
    \dot{\varepsilon} \\ \Delta \varepsilon - r^{-2}f'(Q)\varepsilon\end{pmatrix} - \frac{1}{r^2} \begin{pmatrix}
      0 \\ f(Q+\varepsilon)-f(Q)-f'(Q)\varepsilon
      \end{pmatrix}.
\end{align*}
Ignoring higher-order terms for $\boldsymbol{\varepsilon}$ and setting $\lambda =1$ (i.e. $r=y$), we roughly obtain the linearized system:
\begin{equation}\label{def:big H}
  \partial_t \boldsymbol{\varepsilon} + \boldsymbol{H}\boldsymbol{\varepsilon} =0,\quad \boldsymbol{H}\boldsymbol{\varepsilon}=\begin{pmatrix}
    0 & -1 \\
    H & 0
    \end{pmatrix} \begin{pmatrix}
      \varepsilon \\
      \dot{\varepsilon} 
      \end{pmatrix} 
\end{equation} 
where $H$ is the Schr\"odinger operator with explicitly computable potential $f'(Q)$ from \eqref{eq:wm1} and \eqref{def:Q} 
\begin{equation}\label{def:H}
  H:=-\Delta +\frac{V}{y^2},\quad V=f'(Q)=\frac{y^4-6y^2+1}{(y^2+1)^2}.
  \end{equation}
Due to the scaling invariance, we have $H\Lambda Q=0$ where 
\begin{equation}\label{def:Lambda Q}
  \Lambda Q=\frac{2y}{1+y^2}.
\end{equation}
However, $\Lambda Q$ slightly fails to belong to $L^2(\mathbb{R}^2)$, so we call $\Lambda Q$ the \emph{resonance} of $H$. The positivity of $\Lambda Q$ on $\mathbb{R}_+^*$ allows us to factorize $H$:
\begin{equation}\label{def:A}
  H=A^*A,\quad A= -\partial_y + \frac{Z}{y}, \quad A^*= \partial_y + \frac{1+Z}{y},\quad Z(y)=\sin Q=\frac{1-y^2}{1+y^2}.
\end{equation}
The above factorization facilitates examining the formal kernel of $H$ on $\mathbb{R}_+^*$, denoted by $\mathrm{Ker}(H)$. More precisely, the following equivalent form
\begin{align}
  Au&= -\partial_y u + \partial_y(\log \Lambda Q)u =-\Lambda Q{\partial_y}\left(\frac{u}{\Lambda Q}\right)\label{eq:equi form of A} \\
  A^*u&= \frac{1}{y} \partial_y(yu) + \partial_y(\log \Lambda Q)u = \frac{1}{y\Lambda Q}{\partial_y}\left(uy\Lambda Q\right)\label{eq:equi form of A*}
\end{align}
yields for $y>0$, $\mathrm{Ker}(H)=\textnormal{Span}(\Lambda Q,\Gamma)$ where
\begin{equation}\label{def:ker H}
   \Gamma(y)=\Lambda Q\int_1^y\frac{dx}{x(\Lambda Q(x))^2}=\begin{cases}
    O\left(\frac{1}{y}\right) & \textnormal{as } y\to 0 \\
    \frac{y}{4} + O\left(\frac{\log y}{y}\right) & \textnormal{as } y\to \infty.
  \end{cases}
\end{equation} 
From variation of parameters, we obtain the formal inverse of $H$: 
\begin{equation}\label{eq:inverse of H}
  H^{-1}f= \Lambda Q\int_{0}^y f\Gamma xdx-\Gamma\int_0^y f\Lambda Qxdx,
\end{equation}
so the inverse of $\boldsymbol{H}$ is given by
\begin{equation*}
\boldsymbol{H}^{-1}:=\begin{pmatrix}
0& H^{-1} \\
-1 & 0
\end{pmatrix}.
\end{equation*}
We remark that the inverse formula \eqref{eq:inverse of H} is uniquely determined by the boundary condition at the origin: for any smooth function $f$ with $f=O(1)$, $H^{-1}f=O(y^2)$ near the origin.

On the other hand, the super-symmetric conjugate operator $\widetilde{H}$ is given by
  \begin{equation}
    \widetilde{H}:=AA^*=-\Delta +\frac{\widetilde{V}}{y^2},\quad \widetilde{V}(y)=(1+Z)^2-\Lambda Z = \frac{4}{y^2+1}.
  \end{equation}
  We note that $\widetilde{H}$ has a repulsive property represented by its potential
  \begin{equation}\label{eq:repulsive property}
   \widetilde{V} = \frac{4}{y^2+1} >0,\quad \Lambda \widetilde{V} = -\frac{8y^2}{(y^2+1)^2} \le 0.
  \end{equation}
  Based on the following commutation relation
  \begin{equation*}
    A H = \widetilde{H} A,
  \end{equation*}
  we can naturally define higher-order derivatives adapted to the linearized Hamiltonian $H$ inductively:
  \begin{equation}\label{def:adaptedderivative}
    f_0:=f \quad f_{k+1}:=
    \begin{cases}
      A f_k & \textnormal{for} \ k \ \textnormal{even} ,\\
      A^* f_k & \textnormal{for} \ k \ \textnormal{odd} .
    \end{cases}
    \end{equation}
For the sake of simplicity, we denote the corresponding operator as follows:
\begin{equation}\label{def:A^k}
  \mathcal{A}:=A,\quad \mathcal{A}^2:=A^* A  ,\quad \mathcal{A}^3:= AA^*A,\quad \cdots \quad \mathcal{A}^k:=\underbrace{\cdots A^*A A^* A}_{k \textnormal{ times}}.
\end{equation}
%\begin{align}
 % \mathcal{A}^1&:=A, \nonumber\\
  %\mathcal{A}^2&:=A^* A  \nonumber \\
   %\mathcal{A}^3&:= AA^*A \nonumber \\
   %& \;\; \,\vdots \nonumber \\
   % \mathcal{A}^k&:=\overbrace{\cdots A^*A A^* A}^{k \textnormal{ times}}.
%\end{align}
%\begin{equation}
 %\mathcal{A}^{k}:= \begin{cases}
  %  H^{k/2} & \textnormal{for} \ k \ \textnormal{even} ,\\
   % AH^{(k-1)/2}=\widetilde{H}^{(k-1)/2}A & \textnormal{for} \ k \ \textnormal{odd} .
  %\end{cases}.
%\end{equation}
    We observe that $f$ need an odd parity condition near the origin to define $f_k$. More precisely for any smooth function $f$, \eqref{eq:equi form of A} implies 
    \begin{equation}\label{eq:A near origin}
      f_1=Af\sim -y \partial_y(y^{-1}f)
    \end{equation}
    near $y=0$. Thus, $f$ must degenerate near the origin as $f=cy+O(y^2)$ and so $Af=c'y + O(y^2)$. Here, the leading term $c'y$ comes from a cancellation 
    \begin{equation}\label{eq:cancellation A}
      Ay=O(y^2),
    \end{equation}
    which is a direct consequence of \eqref{eq:A near origin}. However, $f_2$ does not degenerate near the origin like $f$ since $A^*$ does not have any cancellation like \eqref{eq:cancellation A}. Hence, $f$ should be more degenerate near the origin as $f=cy+c'y^3 + O(y^4)$. Furthermore, if $f_k$ is to be well-defined for all $k\in \mathbb{N}$, $f$ must satisfy the following condition: for all $p\in \mathbb{N}$, $f$ has a Taylor expansion near the origin as
    \begin{equation}\label{eq:odd near origin}
      f(y) = \sum_{k=0}^p c_k y^{2k+1} + O(y^{2p+3}).
    \end{equation} 
    In Appendix A of \cite{RaphaelSchweyer2014Anal.PDE}, it is proved that for a well-localized smooth 1-corotational map $\Phi(r,\theta)$, the corresponding $u$ be a smooth function that satisfies \eqref{eq:odd near origin}. 
%%%%%%%%%%%%%%%%%%%%%%%%%%%
%%%%%%%%%%%%%%%%%%%%%%%%%%%

\subsection{Admissible functions}\label{sec:admissible}
As mentioned earlier, the leading dynamics of the blow-up are determined by the leading growth of tails from the blow-up profile. %In order to construct appropriate tails, several conditions related to suitable derivatives \eqref{def:adaptedderivative} except for the origin condition \eqref{eq:odd near origin}.
In the same way as in \cite{RaphaelSchweyer2014Anal.PDE} and \cite{Collot2018MEM.AMS}, we first define an "admissible" vector-valued function characterized by three different indices, which represent a certain behavior near the origin and infinity, and the position of nonzero coordinate. 
\begin{definition}[Admissible functions]
  We say that a smooth vector-valued function $\boldsymbol{f}:\mathbb{R}_+ \to \mathbb{R}^2$ is admissible of degree $(p_1,p_2,\iota) \in \mathbb{N}\times \mathbb{Z} \times \{0,1\}$ if
  \begin{itemize}
  \item[(i)] $\boldsymbol{f}$ is situated on the $\iota + 1$-th coordinate, i.e.
  \begin{equation}
  \boldsymbol{f}=\begin{pmatrix} f \\ 0 \end{pmatrix}   \; \textnormal{if} \; \iota=0 \; \textnormal{and} \; \boldsymbol{f}=\begin{pmatrix} 0 \\ f \end{pmatrix} \; \textnormal{if} \; \iota=1 .
  \end{equation}
   As for such case, we use $f$ and $\boldsymbol{f}$ interchangeably.
  \item[(ii)] We can expand ${f}$ near $y=0$: for all $2p\ge p_1$,
  \begin{equation}\label{eq:admissible near origin condition}
   f(y)=\sum_{k=p_1-\iota, k  \textnormal{ is even}}^{2p} c_ky^{k+1}+ O(y^{2p+3})
  \end{equation}
  and similar expansions hold after taking derivatives.
  \item[(iii)] The adapted derivatives ${f}_k$ have the following bounds: for all $k \ge 0$ and $y\ge 1$,
  \begin{equation}\label{eq:admissible near infinity condition}
  |f_k(y)|\lesssim y^{p_2-1-\iota-k}(1+|\log y|\mathbf{1}_{p_2-k-\iota \ge 1})
  \end{equation}
  \end{itemize}
  \end{definition}
  \begin{remark}
    The logarithmic term in \eqref{eq:admissible near infinity condition} comes from integrating $y^{-1}$. 
  \end{remark}
  From \eqref{def:Lambda Q}, we can easily check that $\boldsymbol{\Lambda Q}=(\Lambda Q,0)^t$ is admissible of degree $(0,0,0)$. The next lemma says that admissible functions are designed to be compatible with the linearized operator $\boldsymbol{H}$.
  \begin{lemma}[Action of $\boldsymbol{H}$ and $\boldsymbol{H}^{-1}$ on admissible functions]\label{lem:action of H on admissible function} Let $\boldsymbol{f}$ be an admissible function of degree $(p_1,p_2,\iota)$. Recall $\overline{i}=i \mod 2$. Then 
    \begin{itemize}
    \item[(i)] For all $k \in \mathbb{N}$, $\boldsymbol{H}^k \boldsymbol{f}$ is admissible of degree
    \begin{equation}
      (\max(p_1-k,\iota),p_2-k,\overline{\iota+k}).
    \end{equation}
    \item[(ii)] For all $k \in \mathbb{N}$ and $p_2\ge \iota$, $\boldsymbol{H}^{-k}\boldsymbol{f}$ is admissible of degree
    \begin{equation}
      (p_1+k,p_2+k,\overline{\iota+k}).
    \end{equation}
    \end{itemize}
    \end{lemma}
    \begin{proof}
    (i) This claim directly comes from the facts
      \[\boldsymbol{H}=\begin{pmatrix}
          0 & -1 \\
          H & 0
      \end{pmatrix},\quad  \boldsymbol{H}^2=\begin{pmatrix}
          -H & 0 \\
          0 & -H
        \end{pmatrix}.
          \]
        More precisely, the maximum choice $\max(p_1-k,\iota)$ appears from the cancellation \eqref{eq:cancellation A} near the origin. Near the infinity, the degree condition $p_2-k$ is a consequence of the simple relation $Hf=f_{2}$.  

   (ii) It suffices to calculate the case $k=1$ by induction. For $\iota=0$, 
      \[\boldsymbol{H}^{-1} \boldsymbol{f}=\begin{pmatrix}
          0& H^{-1} \\
          -1 & 0
        \end{pmatrix} \begin{pmatrix}
          f \\
           0
        \end{pmatrix}=\begin{pmatrix}
              0 \\ -f
          \end{pmatrix}, \]
     $\boldsymbol{H}^{-1}\boldsymbol{f}$ is admissible of degree $(p_1+1,p_2+1,1)$. For $\iota=1$, we have
       \[\boldsymbol{H}^{-1} \boldsymbol{f}=\begin{pmatrix}
          0& H^{-1} \\
          -1 & 0
       \end{pmatrix} \begin{pmatrix}
            0 \\
            f
         \end{pmatrix}=\begin{pmatrix}
           H^{-1}f \\
          0
        \end{pmatrix} \]  
        Instead of using the formal inverse formula \eqref{eq:inverse of H} directly, we utilize the relation \eqref{eq:equi form of A*} as
        \begin{equation}\label{eq:AH inverse}
          AH^{-1}f=\frac{1}{y\Lambda Q} \int_0^y f \Lambda Q x dx,
        \end{equation}
        and the relation \eqref{eq:equi form of A} as
        \begin{equation}\label{eq:H inverse}
          H^{-1}f=-\Lambda Q \int_0^y \frac{AH^{-1}f}{\Lambda Q}dx.
        \end{equation}
        Near the origin, \eqref{eq:AH inverse} gives the expansion for $AH^{-1}f$:
        \begin{equation}\label{eq:AH inverse near zero}
          AH^{-1}f = \sum_{k=p_1-1,\textnormal{even}}^{2p} \tilde{c}_k y^{k+2} + O(y^{2p+4}),
        \end{equation}
        thus $H^{-1}f$ satisfies the Taylor expansion
        \begin{equation}\label{eq:H inverse near zero}
          H^{-1}f = \sum_{k=p_1-1,\textnormal{even}}^{2p} \tilde{c}_k y^{k+3} + O(y^{2p+5})=\sum_{k=p_1+1-0,\textnormal{even}}^{2p} \tilde{c}_k y^{k+1} + O(y^{2p+3}).
        \end{equation}
        For $y\ge 1$, \eqref{eq:AH inverse} and \eqref{eq:H inverse} imply
        \begin{align}
          |A H^{-1}f|   &\lesssim  \int_0^y |f| dx \label{eq:AH inverse bound} \\
          &\lesssim \int_1^y x^{p_2-2}(1+|\log x|\mathbf{1}_{p_2 \ge 2}) dx \nonumber\\
          & \lesssim y^{(p_2+1)-1-0 -1}(1+|\log y|\mathbf{1}_{p_2 \ge 1}), \nonumber  \\
          |H^{-1}f| & \lesssim \frac{1}{y} \int_0^y |x AH^{-1}f| dx  \label{eq:H inverse bound}\\
          &\lesssim\frac{1}{y} \int_1^y x^{p_2}(1+|\log x|\mathbf{1}_{p_2 \ge 1}) dx \nonumber \\
          &\lesssim y^{(p_2+1)-0-1}(1+|\log y|\mathbf{1}_{p_2 \ge 0}),  \nonumber
        \end{align}
        we obtain \eqref{eq:admissible near infinity condition} for $f$ and $f_1$. The higher derivatives results come from $H(H^{-1}f)=f$. Hence, $\boldsymbol{H}^{-1}\boldsymbol{f}$ is admissible of degree $(p_1+1,p_2+1,0)$. 
    \end{proof}
    Lemma \ref{lem:action of H on admissible function} yields the presence of the admissible functions which generates the generalized null space of $\boldsymbol{H}$ formally: 
    \begin{definition}[\emph{Generalized kernel of $\boldsymbol{H}$}]\label{def:generalized ker} For each $i \ge 0$, we define an admissible function $\boldsymbol{T}_i$ of degree $(i,i,\overline{i})$ as follows:
      \begin{equation}\label{eq:def of T_i}
     \boldsymbol{T}_{i}:=(-\boldsymbol{H})^{-i}\boldsymbol{\Lambda Q}.
      \end{equation}
      \end{definition}
      \begin{remark}
        By the definition of the admissible functions, we will use the notation $T_i$ as a scalar function. 
      \end{remark}
  \subsection{$b_1$-admissible functions}
      We will keep track of the logarithmic weight $|\log b_1|$ from the blow-up profiles to be constructed later. In the sense, the logarithmic loss of $\boldsymbol{T}_i$ hinders our analysis, so we settle this problem via introducing a new class of functions.    
    \begin{definition}[$b_1$-admissible functions]\label{def:b-adm}
      We say that a smooth vector-valued function $\boldsymbol{f}:\mathbb{R}_+^*\times \mathbb{R}_+ \to \mathbb{R}^2$ is $b_1$-admissible of degree $(p_1,p_2,\iota)\in \mathbb{ N }\times \mathbb{Z}\times \{0,1\}$ if
      \begin{enumerate}
        \item[(i)] $\boldsymbol{f}$ is situated on the $\iota + 1$-th coordinate (so we use $f$ and $\boldsymbol{f}$ interchangeably).

        \item[(ii)] $f=f(b_1,y)$ can be expressed as a finite sum of the smooth functions of the form $h(b_1)\tilde{f}(y)$, where $\tilde{f}(y)$ has a Taylor expansion \eqref{eq:admissible near origin condition} and $h(b_1)$ satisfies 
        \begin{equation}
          \forall l\geq 0, \ \ \left|\frac{\partial^l h_j}{\partial b_1^l}\right|\lesssim \frac{1}{b_1^l},\quad b_1 >0.
        \end{equation}
        \item[(iii)] $f$ and its adapted derivatives $f_k$ given by \eqref{def:adaptedderivative} have the following bounds: there exists a constant $c_{p_2}>0$ such that for all $k \ge 0$ and $y\geq 1$,  
        \begin{equation}\label{eq:b admissible f_k}
          |f_k(b_1,y)| \lesssim   y^{p_2-k-1-\iota}\left(g_{p_2-k-\iota}(b_1,y)+\frac{|\log y|^{c_{p_2}}}{y^2} + \dfrac{\mathbf{1}_{\{p_2 \ge k+3+\iota, y\ge 3B_0\}}}{y^2 b_1^2|\log b_1|}\right),
        \end{equation}
         and for all $l \ge 1$
        \begin{equation}\label{eq:b admissible dbf_k}
          \left|\frac{\partial^{l}}{\partial b_1^l}f_k(b_1,y)\right|\lesssim \frac{ y^{p_2-k-1-\iota}}{b^l_1|\log b_1|}\left( \tilde{g}_{p_2-k-\iota}(b_1,y)+\frac{|\log y|^{c_{p_2}}}{y^2}+ \dfrac{\mathbf{1}_{\{p_2 \ge k+3+\iota,y\geq 3B_0\}}}{y^2b_1^2}  \right).
        \end{equation} 
         where $B_0$ is given by \eqref{def:B_0 B_1} and $g_l$, $\tilde{g}_l$ are defined as
        \begin{equation}\label{eq:b admissible g}
          g_l(b_1,y)=\dfrac{1+|\log ({b_1}y)|\mathbf{1}_{\{l\ge 1\}}}{|\log b_1|}{\mathbf{1}}_{y\leq 3B_0},\quad
          \tilde{g}_l(b_1,y)= \dfrac{1+|\log y|\mathbf{1}_{\{l\ge 1\}}}{|\log b_1|}{\mathbf{1}}_{y\leq 3B_0}.
        \end{equation}
      \end{enumerate}
    \end{definition}
  \begin{remark}
   One may think that the asymptotics \eqref{eq:b admissible f_k} and \eqref{eq:b admissible dbf_k} are quite artificial, the functions $g_{\ell}(b_1,y)$ and $\tilde{g}_{\ell}(b_1,y)$ will appear in the construction of the radiation, Lemma \ref{lem:theta_i}. Then the indicator part $\mathbf{1}_{p_2 \ge k+3+\iota,y\ge 3B_0}$ comes from integrating $g_{\ell}$ in the region $1\le y \le 3B_0$ to take $\boldsymbol{H}^{-1}$, which can be seen in more detail in the proof of the following lemma. 
  \end{remark}
    \begin{lemma}[Action of $\boldsymbol{H}$ and $\boldsymbol{H}^{-1}$ on $b_1$-admissible functions]\label{lem:action of H on b admissible} Let $\boldsymbol{f}$ be a $b_1$-admissible function of degree $(p_1,p_2,\iota)$. Then 
      \begin{itemize}
      \item[(i)] for all $k \in \mathbb{N}$, $\boldsymbol{H}^k \boldsymbol{f}$ is $b_1$-admissible of degree 
      \begin{equation}
        (\max(p_1-k,\iota),p_2-k,\overline{\iota+k}).
      \end{equation}
      \item[(ii)] for all $k \in \mathbb{N}$ and $p_2 \ge \iota$, $\boldsymbol{H}^{-k}\boldsymbol{f}$ is $b_1$-admissible of degree 
      \begin{equation}
        (p_1+k,p_2+k,\overline{\iota+k}).
      \end{equation}
      \item[(iii)] The operators $\boldsymbol{\Lambda} : \boldsymbol{f}\mapsto \boldsymbol{\Lambda}\boldsymbol{f}$ and $b_1 \frac{\partial}{\partial b_1} : \boldsymbol{f}\mapsto b_1 \frac{\partial \boldsymbol{f}}{\partial b_1}$ preserve the degree.
      \end{itemize}
      \end{lemma}
      \begin{proof}
      (i) We can borrow the proof of Lemma \ref{lem:action of H on admissible function} since $b_1$ is independent of $H$.
      
      (ii) Similar to the proof of Lemma \ref{lem:action of H on admissible function}, it suffices to consider the case $\iota=1$ and $k=1$. Near the origin, we still use \eqref{eq:AH inverse near zero} and \eqref{eq:H inverse near zero} for $\tilde{f}$ from $h(b_1)\tilde{f}(y)$ in the Definition \ref{def:b-adm}.
      
      However for $y\ge 1$, we need a subtle calculation to integrate the terms containing $g_l$ and $\tilde{g}_l$, defined in \eqref{eq:b admissible g}. More precisely, \eqref{eq:AH inverse bound} implies for $1\le y\le 3B_0$, 
      \begin{align}
        |AH^{-1}f| &\lesssim \int_1^y  x^{p_2-2} g_{p_2-1}(b_1,x)+x^{p_2-4}{|\log x|^{c_{p_2}}} dx \nonumber \\
        &\lesssim \int_1^y x^{p_2-2}\dfrac{1+|\log ({b_1}x)|\mathbf{1}_{\{p_2\ge 2\}}}{|\log b_1|} dx + y^{p_2-3}{|\log y|^{1+c_{p_2}}}\nonumber\\
        & \lesssim \frac{1}{b_1^{p_2-1}|\log b_1|} \int_{0}^{b_1 y} x^{p_2-2}(1+|\log x|\mathbf{1}_{\{p_2\ge 2\}}) dx + y^{p_2-3}{|\log y|^{1+c_{p_2}}} \nonumber\\
        &\lesssim y^{p_2-1} \frac{1+|\log (b_1 y)|\mathbf{1}_{\{p_2\ge 1\}}}{|\log b_1|} + y^{p_2-3}{|\log y|^{1+c_{p_2}}}\nonumber \\
        &= y^{(p_2+1)-1-1-0} \left( g_{(p_2+1)-1}(b_1,y) + \frac{|\log y|^{1+c_{p_2}}}{y^2} \right),\label{eq:AH inverse b admissible small y}
      \end{align}
      and for $y\ge 3B_0$, 
        \begin{align}
          |AH^{-1}f| &\lesssim \int_1^y  x^{p_2-2}g_{p_2-1}(b_1,x)+x^{p_2-4}{|\log x|^{c_{p_2}}} + \dfrac{x^{p_2-4}\mathbf{1}_{\{p_2 \ge 4 ,x\ge 3B_0 \}}}{b_1^2|\log b_1|} dx \nonumber \\
          &\lesssim  \frac{1}{b_1^{p_2-1}|\log b_1|} +  \dfrac{y^{p_2-3}\mathbf{1}_{\{p_2 \ge 4 \}}}{b_1^2|\log b_1|}  +y^{p_2-3}{|\log y|^{1+c_{p_2}}} \nonumber \\
          &\lesssim y^{(p_2+1)-1-1-0} \left( \dfrac{\mathbf{1}_{\{p_2 \ge 1+3, y\ge 3B_0 \}}}{y^2b_1^2|\log b_1|}+ \frac{|\log y|^{1+c_{p_2}}}{y^2} \right).\label{eq:AH inverse b admissible large y}
        \end{align}
        Once again, \eqref{eq:H inverse bound} and \eqref{eq:AH inverse b admissible small y} yield for $1\le y\le 3B_0$,
        \begin{align*}
          |H^{-1}f| &\lesssim \frac{1}{y}\int_1^y  x^{p_2} g_{p_2}(b_1,x)+x^{p_2-3}{|\log x|^{1+c_{p_2}}} dx  \\
          &= y^{(p_2+1)-1-0} \left( g_{p_2+1}(b_1,y) + \frac{|\log y|^{2+c_{p_2}}}{y^2} \right),
        \end{align*}
        and \eqref{eq:AH inverse b admissible large y} implies for $y\ge 3B_0$,
        \begin{align}
          |H^{-1}f| &\lesssim \frac{1}{y}\int_1^y x^{p_2-2}{|\log x|^{1+c_{p_2}}} + \dfrac{x^{p_2-2}\mathbf{1}_{\{p_2 \ge 4 ,x\ge 3B_0 \}}}{b_1^2|\log b_1|} dx \nonumber \\
          &\lesssim y^{(p_2+1)-1-0} \left( \dfrac{\mathbf{1}_{\{p_2 \ge 3, y\ge 3B_0 \}}}{y^2b_1^2|\log b_1|}+ \frac{|\log y|^{2+c_{p_2}}}{y^2} \right),\nonumber
        \end{align}
        we obtain \eqref{eq:b admissible f_k} for $f$ and $f_1$. The higher derivatives results come from $H(H^{-1}f)=f$. We can easily prove \eqref{eq:b admissible dbf_k} by replacing $g_{l}$ to $\tilde{g}_l$ and dividing $b_1^l |\log b_1|$. Hence, $\boldsymbol{H}^{-1}\boldsymbol{f}$ is $b_1$-admissible of degree $(p_1+1,p_2+1,0)$. 

        (iii) Note that
        \[ \boldsymbol{\Lambda f}=
        \begin{cases}
          (\Lambda f, 0)^t & \textnormal{if } \iota=0, \\
          (0,\Lambda_0 f)^t & \textnormal{if } \iota=1,
        \end{cases} \]
        and $\Lambda_0 f = f+ \Lambda f$, we get the desired result since ${\Lambda}$ preserve the parity of ${f}$ and its adapted derivative satisfies the bound
        \[|(\Lambda f)_k| \lesssim |y f_{k+1}| + |f_k| + y^{p_2-k-3-\iota} ,\quad y\ge 1,\]
        which established in \cite{RaphaelSchweyer2014Anal.PDE}.

        Near the origin, the property of the operator $b_1 \frac{\partial}{\partial b_1}$ comes from the fact that $b_1 \frac{\partial}{\partial b_1}$ preserves the parity of $f$. For $y\ge 1$, \eqref{eq:b admissible dbf_k} multiplied by $b_1$ with $l=1$ is bounded to \eqref{eq:b admissible f_k} from the following bound
        \[\frac{\tilde{g}_l(b_1,y)}{|\log b_1|} \lesssim {g}_l(b_1,y). \qedhere\]
      \end{proof}
%%%%%%%%%%%%%%%%%%%%%%%%%%%
\subsection{Control of the extra growth} The elements of the null space of $\boldsymbol{H}$, which was defined in \eqref{eq:def of T_i}, serves as a kind of tails in our blow-up profile. Since we basically plan a bubbling off blow-up by scaling, the situation where the scaling generator $\boldsymbol{\Lambda}$ is taken by the tails $\boldsymbol{T}_i$ naturally emerges. Especially for $i\ge 2$, the leading asymptotics of $\boldsymbol{\Lambda T}_i$ matches that of $(i-1)\boldsymbol{T}_i$ and determines the leading dynamical laws. However, the extra growth of $\boldsymbol{\Lambda T}_i - (i-1)\boldsymbol{T}_i$ is inadequate to close our analysis, we will eliminate it by adding some radiations, which were first introduced in \cite{MerleRaphaelRodnianski2013Invent.Math.}.

We now define the radiation situated on the first coordinate as follows: for small $b_1>0$,
\begin{equation}
\boldsymbol{\Sigma}_{b_1}=\begin{pmatrix} \Sigma_{b_1} \\ 0 \end{pmatrix},\quad \Sigma_{b_1} = H^{-1} \{-{c}_{b_1} \chi_{B_0/4} \Lambda Q + d_{b_1} H[(1-\chi_{B_0})\Lambda Q]\}
\end{equation}
where
\begin{align}
  c_{b_1}&= \frac{4}{\int \chi_{B_0/4} (\Lambda Q)^2}=\frac{1}{|\log b_1|}+O\left( \frac{1}{|\log b_1|^2}\right),\label{def:c_b}\\
  d_{b_1}&=c_{b_1} \int_0^{B_0} \chi_{B_0/4} \Lambda Q \Gamma y dy = O\left(\frac{1}{b_1^2 |\log b_1|}\right).\label{def:d_b}
\end{align}
From the inverse formula \eqref{eq:inverse of H}, we obtain the asymptotics near origin and infinity:
\begin{align}\label{eq:sigma asymptotics}
  \Sigma_{b_1}= \begin{cases}
    c_{b_1}T_2 & \textnormal{for} \ \ y \leq \frac{B_0}{4} \\
4\Gamma & \textnormal{for} \ \ y \geq 3B_0.
  \end{cases}
\end{align}
To deal with $\boldsymbol{T}_1$, which is radiative itself, we further define
\begin{equation}\label{def:tilde c_b}
  \tilde{c}_{b_1}:=\frac{\langle \Lambda_0 \Lambda Q, \Lambda Q \rangle}{\langle \chi_{B_0/4} \Lambda Q, \Lambda Q \rangle}=\frac{1}{2|\log b_1|} + O\left( \frac{1}{|\log b_1|^2} \right).
\end{equation}
\begin{lemma}[Cancellation by the radiation]\label{lem:theta_i}
  For $i\ge 1$, let $\boldsymbol{\Theta}_i$ be 
  \begin{align}
    \boldsymbol{\Theta}_1&:=\boldsymbol{\Lambda}\boldsymbol{T}_1 - \tilde{c}_{b_1} \chi_{B_0/4} \boldsymbol{T}_1 \\
    \textnormal{for }i\ge 2, \quad  \boldsymbol{\Theta}_i&:=\boldsymbol{\Lambda}\boldsymbol{T}_i - (i-1)\boldsymbol{T}_i -(-\boldsymbol{H})^{-i+2} \boldsymbol{\Sigma}_{b_1} 
  \end{align}
  where $\boldsymbol{T}_i$ is given by \eqref{eq:def of T_i}. Then $\boldsymbol{\Theta}_i$ is $b_1$-admissible of degree $(i,i,\overline{i})$.
\end{lemma}
\begin{remark}
  As mentioned earlier, our radiation $\Sigma_{b_1}$ cancels the extra growth of $\Lambda T_2 - T_2  \sim y$ from the asymptotics
  \[T_2 = y\log y + cy + O\left( \frac{|\log y|^2}{y} \right),\quad \Lambda T_2 = y\log y + (c+1)y + O\left( \frac{|\log y|^2}{y} \right)\]
  by $4\Gamma$ in \eqref{eq:sigma asymptotics}. Since $T_2$ and $\Gamma$ are elements of the generalized null space of $H$, the above cancellation holds for all $\boldsymbol{\Theta}_i$, $i \ge 2$. 
\end{remark}
\begin{proof}
  \textbf{Step 1:} $i=1$.
  Note that $\boldsymbol{\Theta}_1=(0,\Theta_1)^t$ and \[\Theta_1=\Lambda_0\Lambda Q - \tilde{c}_{b_1} \Lambda Q \chi_{B_0/4},\]
  %\begin{equation}
   % \boldsymbol{\Theta}_1=\begin{pmatrix} 0\\ \Theta_1 \end{pmatrix}= \begin{pmatrix} 0\\ \Lambda_0\Lambda Q - \tilde{c}_{b_1} \Lambda Q \chi_{B_0/4} \end{pmatrix}.
  %\end{equation}
  $\boldsymbol{\Theta}_1$ is $b_1$-admissible of degree $(1,1,1)$ from the explicit formulae
  \[\Lambda Q(y) =\frac{2y}{1+y^2},\quad \Lambda_0 \Lambda Q (y) =  4y/(1+y^2)^2\]
   and the bounds for $l\ge 1$,
  \begin{equation}\label{eq:b derivatives estimates}
    \left\lvert \frac{\partial^l {c}_{b_1}}{\partial b_1^l} \right\rvert+ \left\lvert \frac{\partial^l \tilde{c}_{b_1}}{\partial b_1^l} \right\rvert \lesssim \frac{1}{b_1^l |\log b_1|^2},\quad  \left\lvert \frac{\partial^l {d}_{b_1}}{\partial b_1^l}\right\rvert\lesssim \frac{1}{b_1^{l+2}|\log b_1|} ,\quad \left\lvert \frac{\partial^l \chi_{B_0}}{\partial b_1^l} \right\rvert \lesssim \frac{\mathbf{1}_{y\sim B_0}}{b_1^l}.
  \end{equation}
  \textbf{Step 2:} $i=2$. Now, we use induction on $i\ge 2$. For $i=2$, \eqref{eq:sigma asymptotics} and the admissibility of $\boldsymbol{T}_2$ imply that $\Theta_2$ satisfies the desired condition near zero \eqref{eq:admissible near origin condition} since 
\begin{equation}
  \boldsymbol{\Theta}_2=\begin{pmatrix} \Theta_2\\ 0 \end{pmatrix}= \begin{pmatrix} \Lambda T_2 - T_2 -\Sigma_{b_1}\\ 0 \end{pmatrix}.
\end{equation}
 To exhibit the behavior near infinity, we deal with the case $1\le y \le 3B_0$ and $y\ge 3B_0$ separately. The inverse formula \eqref{eq:inverse of H} yields for $1\le y \le 3B_0$, 
\begin{align}
  \Sigma_{b_1}(y)&=\Gamma\int_0^y {c}_{b_1} \chi_{B_0/4} (\Lambda Q)^2 xdx - \Lambda Q\int_{0}^y {c}_{b_1} \chi_{B_0/4} \Lambda Q \Gamma xdx + d_{b_1} (1-\chi_{B_0}) \Lambda Q \nonumber \\
  & =  y\frac{\int_0^y\chi_{\frac{B_0}{4}}(\Lambda Q)^2 x}{\int \chi_{\frac{B_0}{4}}(\Lambda Q)^2x}+O\left(\frac{1+y}{|\log b_1|}\right),\\
  \Theta_2(y)& = y+ O\left( \frac{|\log y|^2}{y} \right)  - y\frac{\int_0^y\chi_{\frac{B_0}{4}}(\Lambda Q)^2 x}{\int \chi_{\frac{B_0}{4}}(\Lambda Q)^2}+O\left(\frac{1+y}{|\log b_1|}\right)\nonumber\\
  &=y\frac{\int_{y}^{B_0}\chi_{B_0/4}(\Lambda Q)^2 x}{\int \chi_{\frac{B_0}{4}}(\Lambda Q)^2}+O\left(\frac{1+y}{|\log b_1|}\right)+O\left(\frac{|\log y|^2}{y}\right) \nonumber \\
  &= O \left( \frac{1+y}{|\log b_1|}(1+|\log (b_1 y) |) \right).
\end{align}
For $y\ge 3B_0$, \eqref{def:ker H} implies
\begin{equation}
  \Sigma_{b_1}(y) = \Gamma\int_0^y {c}_{b_1} \chi_{B_0/4} (\Lambda Q)^2 xdx = y+ O\left( \frac{\log y}{y} \right).
\end{equation}
Hence, for $y\ge 1$, $\Theta_2$ satisfies \eqref{eq:b admissible f_k} for the case $k=0$ as
\begin{equation}
|\Theta_2(y)| \lesssim y^{2-0-1-0}g_2(b_1,y) + y^{2-0-3-0}(\log y)^2.
\end{equation}
The higher derivatives, namely $f_k$ and $\partial^l f_k/\partial b_1^l$ can also be estimated by using \eqref{eq:AH inverse}, the bounds of the coefficients \eqref{def:c_b}, \eqref{def:d_b}, \eqref{eq:b derivatives estimates} and the commutator relation
\[A(\Lambda f) = Af + \Lambda Af -\frac{\Lambda Z}{y} f,\quad H(\Lambda f)=2Hf + \Lambda Hf - \frac{\Lambda V}{y^2}f\]
where $Z$ and $V$ are given by \eqref{def:H} and \eqref{def:A}. Here, we can easily check that $\Lambda Z /y$ is an odd function and $\Lambda V /y^2$ is an even function. Furthermore for $y\ge 1$,
\begin{equation}\label{eq:bound of potentials}
 \left\lvert \frac{\partial^k}{\partial y^k} \left( \frac{\Lambda Z}{y} \right) \right\rvert \lesssim \frac{1}{1+y^{k+3}},\quad \left\lvert \frac{\partial^k}{\partial y^k} \left( \frac{\Lambda V}{y} \right) \right\rvert \lesssim \frac{1}{1+y^{k+4}}.
\end{equation}
Therefore, $\boldsymbol{\Theta}_2$ is $b_1$-admissible of degree $(2,2,0)$.

\textbf{Step 3}: \emph{Induction on} $i$. Suppose that $\boldsymbol{\Theta}_i$ is $b_1$-admissible of degree $(i,i,\overline{i})$. For even $i$, $\boldsymbol{\Theta}_{i+1}$ is $b_1$-admissible of degree $(i+1,i+1,\overline{i+1})$ since
\begin{align*}
  \boldsymbol{\Theta}_{i+1}&= \begin{pmatrix} 0 \\ \Lambda_0 T_{i+1} - iT_{i+1} - (-H)^{-i/2+1}\Sigma_{b_1}  \end{pmatrix}\\
  & = \begin{pmatrix} 0 \\ \Lambda T_{i} - (i-1)T_{i} - (-H)^{-i/2+1}\Sigma_{b_1}  \end{pmatrix} = \begin{pmatrix} 0 \\ \Theta_i \end{pmatrix}.
\end{align*}
For odd $i$, we have
\begin{align*}
  \boldsymbol{H \Theta}_{i+1}&=\begin{pmatrix} 0 & 1 \\ H & 0 \end{pmatrix} \begin{pmatrix} \Theta_{i+1} \\ 0\end{pmatrix}
\\
& =\begin{pmatrix} 0 \\ H \Lambda T_{i+1} - i HT_{i+1} -H(-H)^{-(i+1)/2+1} \Sigma_{b_1}\end{pmatrix} \\
&= \begin{pmatrix} 0 \\ \Lambda HT_{i+1} - (i-2) HT_{i+1} -y^{-2}\Lambda V T_{i+1} + (-H)^{-(i-1)/2+1} \Sigma_{b_1}\end{pmatrix}\\
&=-\begin{pmatrix} 0 \\ \Lambda T_{i} - (i-2) T_{i}  - (-H)^{-(i-1)/2+1} \Sigma_{b_1}+y^{-2}\Lambda V T_{i+1}\end{pmatrix} \\
&= -\begin{pmatrix} 0 \\ \Lambda_0 T_{i} - (i-1) T_{i}  - (-H)^{-(i-1)/2+1} \Sigma_{b_1}\end{pmatrix} + \begin{pmatrix}
  0 \\ y^{-2}\Lambda V T_{i+1}
\end{pmatrix}\\
&= -\boldsymbol{\Theta}_i +\begin{pmatrix}
  0 \\ y^{-2}\Lambda V T_{i+1}
\end{pmatrix}.
\end{align*}
 The Taylor expansion condition \eqref{eq:admissible near origin condition} of $(0, y^{-2} \Lambda V T_{i+1})^t$ comes from the definition of $\boldsymbol{T}_i$ and the cancellation $\Lambda V = O(y^2)$ near $y=0$.

For $y \ge 1$, \eqref{eq:bound of potentials} implies
\[\mathcal{A}^k \left( \frac{\Lambda V}{y^2} T_{i+1} \right) \lesssim \sum_{j=0}^k \frac{1}{y^{j+4}} y^{i-(k-j)}|\log y|^{c_i}\lesssim y^{i-3-k-1}|\log y|^{c_i}.\]
Hence, $(0, y^{-2} \Lambda V T_{i+1})^t$ is $b_1$-admissible of degree $(i,i,1)$, the desired result comes from Lemma \ref{lem:action of H on b admissible}.
\end{proof}
%%%%%%%%%%%%%%%%%%%%%%%%%%%

\subsection{Adapted norms of $b_1$ admissible functions}
The next lemma yields some suitable norms corresponding to the adapted derivatives of $b_1$-admissible functions.
\begin{lemma}[Adapted norms of $b_1$-admissible function]\label{lem:admissible bounds}
For $i\ge 1$, a $b_1$-admissible function $\boldsymbol{f}$ of degree $(i,i,\overline{i})$ has the following bounds:
\begin{enumerate}
  \item[(i)] Global bounds:
  \begin{equation}\label{eq:sobolev global}
    {\lVert {f}_{k-\overline{i}} \rVert}_{L^2(|y|\le 2B_1)}  \lesssim 
    \begin{cases}
     {b_1^{k-i}}|\log b_1|^{\gamma(i-k-2)-1} & \textnormal{if } k\le i-3 \\
     \dfrac{b_1^{k-i}}{|\log b_1|} & \textnormal{if } k=i-2,i-1 \\
     1  & \textnormal{if } k\ge i
    \end{cases} 
  \end{equation}
  \item[(ii)] Logarithmic weighted bounds:  
\begin{equation}\label{eq:sobolev log}
  \sum_{k=0}^m {\left\lVert \frac{1+|\log y|}{1+y^{m-k}}f_{k-\overline{i}}\right\rVert}_{L^2(|y|\le 2B_1)}\lesssim \begin{cases}
    {b_1^{m-i}} |\log b_1|^C & \ \ \ \textnormal{ for } m\le i-1\\
    |\log b_1|^C & \ \ \ \textnormal{ for } m\ge i
   \end{cases} 
\end{equation}
  \item[(iii)] Improved global bounds:
  \begin{equation}\label{eq:sobolev improved}
    \sum_{j=0}^{k-\overline{i}} {\left\lVert y^{-(k-\overline{i}-j)} f_{j}\right\rVert}_{L^2(y \sim B_1)}\lesssim b_1^{k-i} |\log b_1|^{\gamma(i-k-2)-1}.
  \end{equation}
\end{enumerate}
Here, $B_1=\frac{|\log b_1|^{\gamma}}{b_1}$ and $\gamma = 1 + \overline{\ell}$.
\end{lemma}
\begin{remark}
  Due to the growth in \eqref{eq:b admissible f_k}, it is indispensable to restrict the integration domain taking $L^2$ norm. Later, we will attach a cutoff function $\chi_{B_1}$ to the profile modifications. Considering Leibniz's rule, the adapted derivative $\mathcal{A}^k$ can be taken on such modifications or the cutoff function. Then the global bounds \eqref{eq:sobolev global} yield some estimates for the former case and \eqref{eq:sobolev improved} give those for the latter case. The choice of cutoff region $B_1$ will be determined by the localization of our blow-up profile, which can be seen in more detail in Proposition \ref{prop:local approx}.
\end{remark}
\begin{proof}
  (i) From \eqref{eq:b admissible f_k}, ${f}_{k- \overline{i}}$ satisfies the following estimate for $y\ge 2$:
   \[|{f}_{k-\overline{i}}|  \lesssim  
   y^{i-k-1}\left(g_{i-k}(b_1,y)+\frac{|\log y|^{c_{p_2}}}{y^2} + \dfrac{\mathbf{1}_{\{i \ge k+3, y\ge 3B_0\}}}{y^2 b_1^2|\log b_1|}\right).\]
   Therefore, we obtain \eqref{eq:sobolev global} for $i\ge k+1$,
   \begin{align*}
   {\lVert {f}_{k-\overline{i}} \rVert}_{L^2(|y|\le 2B_1)}  &\lesssim {\lVert \mathbf{1}_{|y|\le 2} \rVert}_{L^2} +  {\left\lVert  y^{i-k-1}\frac{1+|\log (b_1y)|}{|\log b_1|} \right\rVert}_{L^2 (2\le |y|\le 3B_0)} \\
   & + {\lVert  y^{i-k-3}|\log y|^{c_{i}}\rVert}_{L^2 (2\le |y|\le 2B_1)}+{\left\lVert  \frac{y^{i-k-3}\mathbf{1}_{\{i\ge k+3\}}}{b_1^2|\log b_1|}\right\rVert}_{L^2 (3B_0\le |y|\le 2B_1)} \\
    &\lesssim 1 + \frac{b_1^{k-i}}{|\log b_1|}+b_1^{(k-i+2)\mathbf{1}_{\{i\ge k+2\}}}|\log b_1|^C + \frac{B_1^{i-k-2} }{b_1^2 |\log b_1|}\mathbf{1}_{\{i\ge k+3\}}\\
    &\lesssim \dfrac{b_1^{k-i}}{|\log b_1|}|\log b_1|^{\gamma(i-k-2)\mathbf{1}_{\{i\ge k+3\}}},
   \end{align*}
   and the case $i\le k$ also holds similarly.
   
   (ii) The logarithmic weighted bounds \eqref{eq:sobolev log} are nothing but \eqref{eq:sobolev global} multiplied by the logarithmic loss $|\log b_1|^C$ with the fact $|\log y|/|\log b_1| \lesssim 1$ on $2\le |y| \le 3B_0$.

   (iii) We can prove \eqref{eq:sobolev improved} from pointwise estimate in the region $y\sim B_1$:
   \begin{equation}\label{eq:pointwise homogeneous}
    |y^{-(k-\overline{i} -j)} f_j | \lesssim y^{i-k-3} \left( |\log y|^C + \frac{\mathbf{1}_{\{i\ge \overline{i} + j +3\}}}{b_1^2 |\log b_1|} \right) \lesssim \frac{y^{i-k-1}}{|\log b_1|^{2\gamma + 1}}. \qedhere
   \end{equation}
\end{proof}
%%%%%%%%%%%%%%%%%%%%%%%%%%

\subsection{Approximate blow-up profiles}
From now on, we fix
\[\ell \ge 2 \quad \textnormal{and} \quad L=\ell+\overline{\ell+1}.\]
We construct the blow-up profiles based on the generalized kernels $\boldsymbol{T}_i$. To be more specific, our blow-up scenario is done by bubbling off $\boldsymbol{Q}$ via scaling and adding $b_i \boldsymbol{T}_i$, the evolution of $\lambda$ is determined by the system of dynamical laws for $b=(b_1,\dots,b_L)$. Here, we are faced with unnecessary growth made by linear and nonlinear terms. To minimize this growth, we define the homogeneous functions, which do not affect the evolution of $b$ (i.e. $b_i T_i$). We note that this kind of construction was introduced in \cite{RaphaelSchweyer2014Anal.PDE}.
\begin{definition}[Homogeneous functions]
  Denote $J=(J_1,\dots,J_L)$ and $|J|_2=\sum_{k=1}^L k J_k$. We say that a smooth vector-valued function $\boldsymbol{S}(b,y)=\boldsymbol{S}(b_1,\dots,b_L,y)$ is homogeneous of degree $(p_1,p_2,\iota,p_3) \in \mathbb{N} \times \mathbb{Z}\times \{0,1\} \times\mathbb{N}$ if it can be expressed as a finite sum of smooth functions of the form $(\prod_{i=1}^L b_i^{J_i}) \boldsymbol{S}_J(y)$, where $\boldsymbol{S}_J(y)$ is a $b_1$-admissible function of degree $(p_1,p_2,\iota)$ with $|J|_2=p_3$.
\end{definition}

%%%%%%%%%%%%%%%%%%%%%%%%%%%%%%%%%%%%%%

%%%%%%%%%%%%%%%%%%%%%%%%%%%%%%%%%%%%%%

\begin{proposition}[Construction of the approximate profile]\label{prop:approx profile}
Given a large constant $M>0$, there exists a small constant $0<b^*(M) \ll 1$ such that a $C^1$ map 
\[b : s \mapsto (b_1(s),\dots,b_L(s))\in \mathbb{R}^{*}_+ \times \mathbb{R}^{L-1}\]
verifies the existence of a slowly modulated profile $\boldsymbol{Q}_b$ given by
\begin{equation}
  \boldsymbol{Q}_{b}:=\boldsymbol{Q}+\boldsymbol{\alpha}_{b}, \ \boldsymbol{\alpha}_{b}:=\sum_{i=1}^L b_i\boldsymbol{T}_i+\sum_{i=2}^{L+2} \boldsymbol{S}_i,
  \end{equation} 
which drives the following equation
\begin{equation}\label{eq:modpsi}
  \partial_s \boldsymbol{Q}_{b}-\boldsymbol{F}(\boldsymbol{Q}_b) + b_1 \boldsymbol{\Lambda Q}_b = \mathbf{Mod}(t) +\boldsymbol{\psi}_b. 
\end{equation}
where $\mathbf{Mod}(t)$ establishes the dynamical law of $b$:
\begin{equation}\label{def:mod}
  \mathbf{Mod}(t) = \sum_{i=1}^L ((b_i)_s +(i-1 + c_{b_1,i})b_1b_i - b_{i+1}) \left(\boldsymbol{T}_i + \sum_{j=i+1}^{L+2} \frac{\partial \boldsymbol{S}_j}{\partial b_i}\right),
\end{equation}
where we set $b_{L+1}=0$ for convenience and $c_{b_1,i}$ is defined by
\begin{equation}
  c_{b_1,i} = \begin{cases}
    \tilde{c}_{b_1}=\frac{\langle \Lambda_0 \Lambda Q, \Lambda Q \rangle}{\langle \chi_{B_0/4} \Lambda Q, \Lambda Q \rangle} & \textnormal{for } i=1 \\
    c_{b_1}=\frac{4}{\int \chi_{B_0/4} (\Lambda Q)^2} & \textnormal{for } i\neq 1
  \end{cases}
\end{equation}
Here, $\boldsymbol{T}_i$ is given by \eqref{eq:def of T_i} and $\boldsymbol{S}_i$ is a homogeneous function of degree $(i,i,\overline{i},i)$ satisfies
  \begin{equation}\label{linearized:eq:degree Si}
    \boldsymbol{S}_1=0,\quad
    \frac{\partial \boldsymbol{S}_i}{\partial b_j}=0 \ \ \textnormal{for} \ \ 2\leq i\leq j \leq L.
  \end{equation}
Moreover, the restriction $|b_k|\lesssim b_1^k$ and $0<b_1 < b^*(M)$ yield the estimates below for $\boldsymbol{\psi}_b=(\psi_b,\dot{\psi}_b)^t$, 
\begin{enumerate}
  \item[(i)] Global bound: for $2\le k \le L-1$,
  \begin{align}
     {\lVert \mathcal{A}^k{\psi}_b \rVert}_{L^2(|y|\le 2B_1)} + {\lVert \mathcal{A}^{k-1}{\dot{\psi}}_b \rVert}_{L^2(|y|\le 2B_1)} &\lesssim b_1^{k+1}|\log b_1|^C \label{eq:psi global k},\\
    \quad {\lVert \mathcal{A}^L{\psi}_b \rVert}_{L^2(|y|\le 2B_1)} + {\lVert \mathcal{A}^{L-1}{\dot{\psi}}_b \rVert}_{L^2(|y|\le 2B_1)} &\lesssim \frac{b_1^{L+1}}{|\log b_1|^{1/2}} \label{eq:psi global L}\\
    {\lVert \mathcal{A}^{L+1}{\psi}_b \rVert}_{L^2(|y|\le 2B_1)} + {\lVert \mathcal{A}^{L}{\dot{\psi}}_b \rVert}_{L^2(|y|\le 2B_1)} &\lesssim \frac{b_1^{L+2}}{|\log b_1|}.\label{eq:psi global L+1}%\\
   %{\lVert \mathcal{A}^L \psi_b \rVert}_{L^2(|y|\le 2B_1)} &\lesssim  \frac{b_1^{L+1}}{|\log b_1|} \label{eq:psi global L}
  \end{align}
  \item[(ii)] Logarithmic weighted bound: for $ m \ge 1 $ and $0\le k \le m$,
  \begin{align}
   {\left\lVert \frac{1+|\log y| }{1+y^{m-k}} \mathcal{A}^k \psi_b \right\rVert}_{L^2(|y|\le 2B_1)} &\lesssim b_1^{m+1}|\log b_1|^C ,\quad m \le L+1 \label{eq:psi log}\\
   {\left\lVert \frac{1+|\log y| }{1+y^{m-k}} \mathcal{A}^k \dot{\psi}_b \right\rVert}_{L^2(|y|\le 2B_1)} &\lesssim b_1^{m+2}|\log b_1|^C,\quad m \le L.\label{eq:psi dot log}
  \end{align}
  \item[(iii)] Improved local bound: 
  \begin{equation}\label{eq:psi local}
    {}^{\forall} 2\le k \le L+1,\quad {\lVert \mathcal{A}^k{\psi}_b \rVert}_{L^2(|y|\le 2M)} + {\lVert \mathcal{A}^{k-1}{\dot{\psi}}_b \rVert}_{L^2(|y|\le 2M)} \lesssim  C(M) b_1^{L+3}.
  \end{equation}
\end{enumerate}
Here, $B_0 = \frac{1}{b_1}$ and $B_1 = \frac{|\log b_1|^{\gamma}}{b_1}$.
\end{proposition}
\begin{remark}
  As can be seen in the following proof, the homogeneous profile $\boldsymbol{S}_i$ is eventually derived from the $b_1$-admissible function $\boldsymbol{\Theta}_{i-1}$ with some nonlinear effects.
\end{remark}
\begin{proof}
  \textbf{Step 1: }Linearization. We pull out the modulation law of $b$ from linearizing the renormalized equation. Recall 
  \[\boldsymbol{F}(\boldsymbol{u}):=\begin{pmatrix}
    \dot{u} \\ \Delta u-\frac{1}{r^2}f(u)
  \end{pmatrix}.\]
  Since $\boldsymbol{F}(\boldsymbol{Q})=0$, we have
  \begin{align*}
    \partial_s \boldsymbol{Q}_b+b_1\boldsymbol{\Lambda} \boldsymbol{Q}_b -\boldsymbol{F}(\boldsymbol{Q}_b) &= \partial_s \boldsymbol{\alpha}_b + b_1 \boldsymbol{\Lambda} (\boldsymbol{Q} + \boldsymbol{\alpha}_b ) - (\boldsymbol{F}(\boldsymbol{Q}+ \boldsymbol{\alpha}_b) -\boldsymbol{F}(\boldsymbol{Q})) \\
    &=: b_1 \boldsymbol{\Lambda}\boldsymbol{Q} + (\partial_s + b_1 \boldsymbol{\Lambda})\boldsymbol{\alpha}_b + \boldsymbol{H}\boldsymbol{\alpha}_b + \boldsymbol{N}(\boldsymbol{\alpha}_b) 
  \end{align*}
  where $\boldsymbol{N}$ denotes the higher-order terms:
  \begin{equation}
    \boldsymbol{N}(\boldsymbol{\alpha}_b) := \frac{1}{y^2}\begin{pmatrix}
      0 \\ f(Q+\alpha_b)-f(Q)-f'(Q)\alpha_b 
      \end{pmatrix},\quad \boldsymbol{\alpha}_b = \begin{pmatrix}
        \alpha_b \\ \dot{\alpha}_b 
        \end{pmatrix}.
  \end{equation}
  Note that 
  \begin{align*}
    \partial_s \boldsymbol{\alpha}_b&=\sum_{i=1}^{L} \left[ (b_i)_s \boldsymbol{T}_i + \sum_{j=i+1}^{L+2}  (b_i)_s \frac{\partial \boldsymbol{S}_j}{\partial b_i}\right] \\
    & =\sum_{i=1}^{L} \left[ (b_i)_s \boldsymbol{T}_i + \sum_{j=1}^{i-1}  (b_j)_s \frac{\partial \boldsymbol{S}_i}{\partial b_j} \right] + \sum_{i=1}^L (b_i)_s \frac{\partial \boldsymbol{S}_{L+1}}{\partial b_i} + \sum_{i=1}^L (b_i)_s \frac{\partial \boldsymbol{S}_{L+2}}{\partial b_i}.
  \end{align*}
Rearranging the linear terms to the degree with respect to $b_1$ using the fact $\boldsymbol{H}\boldsymbol{T}_{i+1}=-\boldsymbol{T}_i$ for $1\le i \le L-1$,
  \begin{align}\label{eq:alpha linear1}
    b_1 \boldsymbol{\Lambda}\boldsymbol{Q} + (\partial_s + b_1 \boldsymbol{\Lambda})\boldsymbol{\alpha}_b + \boldsymbol{H}\boldsymbol{\alpha}_b &=  \sum_{i=1}^{L}[(b_i)_s \boldsymbol{T}_i + b_1b_i\boldsymbol{\Lambda} \boldsymbol{T}_i-b_{i+1}\boldsymbol{T}_{i}] \nonumber \\
    &+ \sum_{i=1}^L \left[ \boldsymbol{H}\boldsymbol{S}_{i+1}+b_1\boldsymbol{\Lambda} \boldsymbol{S}_{i}+\sum_{j=1}^{i-1}  (b_j)_s \frac{\partial \boldsymbol{S}_i}{\partial b_j}  \right] \nonumber \\
    & +  b_1 \boldsymbol{\Lambda} \boldsymbol{S}_{L+1} +\boldsymbol{H}\boldsymbol{S}_{L+2} + \sum_{i=1}^L (b_i)_s \frac{\partial \boldsymbol{S}_{L+1}}{\partial b_i} \nonumber  \\
  &   +b_1\boldsymbol{\Lambda} \boldsymbol{S}_{L+2} + \sum_{i=1}^L (b_i)_s \frac{\partial \boldsymbol{S}_{L+2}}{\partial b_i}. 
  \end{align}
  From Lemma \ref{lem:theta_i}, 
  \begin{equation*}
     (b_1)_s \boldsymbol{T}_1 + b_1^2  \boldsymbol{\Lambda} \boldsymbol{T}_1 - b_2 \boldsymbol{T}_1 = ((b_1)_s + b_1^2  \tilde{c}_{b_1}-b_2)\boldsymbol{T}_1 - b_1^2\tilde{c}_{b_1}(1-\chi_{B_0/4}) \boldsymbol{T}_1 + b_1^2\boldsymbol{\Theta}_1
  \end{equation*}
  and for $2\le i\le L$,
  \begin{align}
    (b_i)_s \boldsymbol{T}_i + b_1b_i\boldsymbol{\Lambda} \boldsymbol{T}_i - b_{i+1}\boldsymbol{T}_i &= ((b_i)_s + (i-1 + c_{b_1})b_1b_i -b_{i+1})\boldsymbol{T}_i \nonumber  \\
    & + b_1b_i(-\boldsymbol{H})^{-i+2} (\boldsymbol{\Sigma}_{b_1}-c_{b_1} \boldsymbol{T}_2) + b_1b_i\boldsymbol{\Theta}_i.\label{eq:b_i identity}
  \end{align}
Hence, we can separate $\mathbf{Mod}(t)$ from the RHS of \eqref{eq:alpha linear1}:
\begin{align}\label{eq:alpha linear2}
  \mathbf{Mod}(t) &- b_1^2\tilde{c}_{b_1}(1-\chi_{B_0/4}) \boldsymbol{T}_1 + \sum_{i=2}^L b_1b_i(-\boldsymbol{H})^{-i+2} (\boldsymbol{\Sigma}_{b_1}-c_{b_1} \boldsymbol{T}_2) \\
& + \sum_{i=1}^L \left[ \boldsymbol{H}\boldsymbol{S}_{i+1} +b_1b_i \boldsymbol{\Theta}_i+b_1\boldsymbol{\Lambda} \boldsymbol{S}_{i}-\sum_{j=1}^{i-1}((j-1+c_{b_1,j})b_1b_j-b_{j+1})\frac{\partial \boldsymbol{S}_i}{\partial b_j}  \right] \nonumber \\
& + \boldsymbol{H}\boldsymbol{S}_{L+2}+b_1 \boldsymbol{\Lambda} \boldsymbol{S}_{L+1}-\sum_{i=1}^L ((i-1+c_{b_1,i})b_1b_i-b_{i+1})\frac{\partial \boldsymbol{S}_{L+1}}{\partial b_i} \nonumber \\
& + b_1 \boldsymbol{\Lambda} \boldsymbol{S}_{L+2}-\sum_{i=1}^L ((i-1+c_{b_1,i})b_1b_i-b_{i+1})\frac{\partial \boldsymbol{S}_{L+2}}{\partial b_i}. \nonumber
\end{align}
\textbf{Step 2: }Construction of $\boldsymbol{S}_i$. One can observe that the second and third lines of \eqref{eq:alpha linear2} provide the definition of homogeneous profiles $\boldsymbol{S}_i$ inductively. We need to pull out the additional homogeneous functions from $\boldsymbol{N}(\boldsymbol{\alpha}_b)=(0,N({\alpha}_b))^t$ via Taylor theorem: 
\begin{equation*}
N(\alpha_b) = \frac{1}{y^2} \left\{ \sum_{j=2}^{\frac{L+1}{2}} \frac{f^{(j)}(Q)}{j!}\alpha_b^j + N_0(\alpha_b) \alpha_b^{\frac{L+3}{2}} \right\}
\end{equation*}
where $N_0(\alpha_b)$ is the coefficient of the remainder term
\begin{equation*}
  N_0(\alpha_b) =  \frac{1}{((L+1)/2)!}\int_0^1 (1-\tau)^{\frac{L+1}{2}}f^{\left(\frac{L+3}{2}\right)}(Q+\tau \alpha_b) d\tau.
  \end{equation*}
  Roughly, $N_0(\alpha_b)=O(b_1^{L+3})$. We also rewrite the Taylor polynomial part of $N(\alpha_b)$ in terms of the degree of $b_1$: for the $L$-tuple $J:=(J_2,J_4,\dots,J_{L-1},\tilde{J}_2,\tilde{J}_4,\dots,\tilde{J}_{L+1})$,  
\begin{align*}
  \sum_{j=2}^{\frac{L+1}{2}} \frac{f^{(j)}(Q)}{j!}\alpha_b^j &= \sum_{i=1}^{\frac{L+1}{2}} P_{2i} + R'
\end{align*} 
where 
\begin{align*}
  P_i &:= \sum_{j=2}^{\frac{L+1}{2}} \sum_{|J|_1=j}^{|J|_2=i} c_{j,J} \prod_{k=1}^{\frac{L-1}{2}} (b_{2k}T_{2k})^{J_{2k}} \prod_{k=1}^{\frac{L+1}{2}} S_{2k}^{\tilde{J}_{2k}}, \\   
R' &:= \sum_{j=2}^{\frac{L+1}{2}}\sum_{|J|_1=j }^{|J|_2 \ge L+3}   c_{j,J} \prod_{k=1}^{\frac{L-1}{2}} (b_{2k}T_{2k})^{J_{2k}} \prod_{k=1}^{\frac{L+1}{2}} S_{2k}^{\tilde{J}_{2k}},\quad  c_{j,J}= \frac{f^{(j)}(Q)}{ \prod_{k=1}^{\frac{L-1}{2}} J_{2k}! \prod_{k=1}^{\frac{L+1}{2}} \tilde{J}_{2k}!}
\end{align*}
with two distinct counting notations
\begin{align*}
   |J|_1 := \sum_{k=1}^{\frac{L-1}{2}} J_{2k} + \sum_{k=1}^{\frac{L+1}{2}} \tilde{J}_{2k},\quad  |J|_2 := \sum_{k=1}^{\frac{L-1}{2}} 2kJ_{2k} + \sum_{k=1}^{\frac{L+1}{2}} 2k\tilde{J}_{2k}.
\end{align*}
In short, $P_{2i} = O(b_1^{2i})$ and $R'=O(b_1^{L+3})$. We collect all $O(b_1^{L+3})$ terms
\begin{equation} \label{def:remainder of N}
  R:= N_0(\alpha_b)\alpha_b^{\frac{L+3}{2}} + R'
\end{equation}
We claim that $\boldsymbol{P}_{2i}/y^2 = (0,P_{2i}/y^2)$ is homogeneous of degree $(2i-1,2i-1,1,2i)$ for $1\le i \le \frac{L+1}{2}$. The case $i=1$ is trivial since $P_2=0$. For $2\le i \le \frac{L+1}{2}$, we recall that ${P}_{2i}/y^2$ is a linear combination of the following monomials: for $ |J|_1=j$, $|J|_2=2i$ and $2\le j \le i$,
\[\frac{f^{(j)}(Q)}{y^2}  \prod_{k=1}^{i} (b_{2k}T_{2k})^{J_{2k}} \prod_{k=1}^{i} S_{2k}^{\tilde{J}_{2k}}.\]
Near the origin, we observe that $T_{2k}$, $S_{2k}$ are odd functions and the parity of a function $f^{(j)}(Q)$ is determined by the parity of $j$, each monomial is either an odd or even function. Hence, it suffices to calculate the leading power of the Taylor expansion of each function constituting the monomial: $T_{2k} \sim y^{2k+1}$, $S_{2k} \sim O(b_1^{2k}) y^{2k+1}$ and $f^{(j)}(Q) \sim y^{\overline{j+1}}$, the leading power of each monomial is given by 
\begin{equation}\label{eq:b order monomial}
  b_1^{\sum_{k=1}^i 2k {J}_{2k} }\cdot b_1^{\sum_{k=1}^i 2k \tilde{J}_{2k} } = b_1^{2i},
\end{equation}
\[y^{-2} y^{\overline{j+1}}  y^{\sum_{k=1}^i (2k+1)J_{2k} }  y^{\sum_{k=1}^i (2k+1)\tilde{J}_{2k} }=  y^{2i+j-1-\overline{j} }.\]
Therefore, the Taylor expansion condition \eqref{eq:admissible near origin condition} comes from $j-1-\overline{j} \ge 1 $ is an odd number since $j\ge 2$. 

Similarly for $y\ge 1$, $|T_{2k}| \lesssim y^{2k-1}\log y$, $|S_{2k}|\lesssim b_1^{2k} y^{2k-1}$ and $|f^{(j)}(Q)|\lesssim y^{-1 + \overline{j}}$ imply 
\begin{align}
  \left| \frac{ f^{(j)}(Q)}{y^2} \prod_{k=1}^i b_{2k}^{J_{2k}}T_{2k}^{J_{2k}} \prod_{k=1}^i S_{2k}^{\tilde{J}_{2k}} \right| &\lesssim b_1^{2i} |y^{-3 + \overline{j}}| \prod_{k=1}^i |y^{2k-1}\log y|^{J_{2k}} \prod_{k=1}^i {|y^{2k-1}|}^{\tilde{J}_{2k}}\nonumber \\
  &\lesssim b_1^{2i} y^{2i-j-3 + \overline{j}} |\log y|^C \lesssim b_1^{2i} y^{2i-5} |\log y|^C \label{eq:far away monomial}
\end{align}
with the fact $j-\overline{j} \ge 2$. We can easily estimate the higher derivatives of each monomial.

Under the setting $\boldsymbol{P}_{2k+1} := (0,0)^t$ for $k \in \mathbb{N}$, we obtain the final definition of $\boldsymbol{S}_i$: $\boldsymbol{S}_1:=0$ and for $i=1,\dots,L+1$,
\begin{equation}
    \boldsymbol{S}_{i+1}:= (-\boldsymbol{H})^{-1} \left(b_1b_i \boldsymbol{\Theta}_i+b_1\boldsymbol{\Lambda} \boldsymbol{S}_{i} + \frac{\boldsymbol{P}_{i+1}}{y^2}-\sum_{j=1}^{i-1}((j-1+c_{b_1,j})b_1b_j-b_{j+1})\frac{\partial \boldsymbol{S}_i}{\partial b_j}  \right).
\end{equation}
From the homogeneity of $\boldsymbol{P}_i/y^2$ established above and Lemma \ref{lem:action of H on b admissible}, Lemma \ref{lem:theta_i}, we can prove $\boldsymbol{S}_i$ is homogeneous of degree $(i,i,\overline{i}, i)$ for $1\le i \le L+2$ with \eqref{linearized:eq:degree Si} via induction. To sum up, we get \eqref{eq:modpsi} by collecting remaining errors into $\boldsymbol{\psi}_b$:
\begin{align}
  \boldsymbol{\psi}_b :=&-b_1^2\tilde{c}_{b_1}(1-\chi_{B_0/4})\boldsymbol{T}_1 +\sum_{i=2}^{L}b_1b_i (-\boldsymbol{H})^{-i+2}\widetilde{\boldsymbol{\Sigma}}_{b_1}\label{eq:psi1}   \\
  & + b_1 \boldsymbol{\Lambda} \boldsymbol{S}_{L+2}-\sum_{i=1}^L ((i-1+c_{b_1,i})b_1b_i-b_{i+1})\frac{\partial \boldsymbol{S}_{L+2}}{\partial b_i} +\frac{\boldsymbol{R}}{y^2}\label{eq:psi2}
\end{align}
where $\widetilde{\boldsymbol{\Sigma}}_{b_1}:=\boldsymbol{\Sigma}_{b_1}-c_{b_1}\boldsymbol{T}_2$ and $\boldsymbol{R}=(0,R)^t$ from \eqref{def:remainder of N}. 

\textbf{Step 3: }Error bounds. Now, it remains to prove the Sobolev bounds: \eqref{eq:psi global k} to \eqref{eq:psi local}. We can treat the errors involving $\boldsymbol{S}_{L+2}$ in \eqref{eq:psi2} easily. Since $\boldsymbol{S}_{L+2}$ is homogeneous of degree $(L+2,L+2,1,L+2)$, Lemma \ref{lem:action of H on b admissible} ensures that the functions containing $\boldsymbol{S}_{L+2}$ are homogeneous of degree $(L+2,L+2,1,L+3)$ and thus the desired bounds come from Lemma \ref{lem:admissible bounds}. 

The other errors require separate integration to conclude. We first visit the RHS of \eqref{eq:psi1}. Note that $\boldsymbol{T}_1=(0,T_1)^t$ and $\Lambda Q \sim 1/y$ on $y\ge 1$, we have for $k\ge 0$,
\begin{equation}\label{eq:bound k t1}
  |\mathcal{A}^k(1-\chi_{B_0/4}) T_1|  \lesssim   y^{-(k+1)}\mathbf{1}_{y \ge B_0/4},
\end{equation}
which imply \eqref{eq:psi global k}, \eqref{eq:psi global L} and \eqref{eq:psi global L+1}: for $2\le k \le L+1$,
\begin{align*}
   {\lVert b_1^2 \tilde{c}_{b_1} \mathcal{A}^{k-1} (1-\chi_{B_0/4}) {T}_1 \rVert}_{L^2(|y|\le 2B_1)}&\lesssim \frac{b_1^2}{|\log b_1|}  {\lVert  y^{-k} \rVert}_{L^2(B_0/4 \le |y|\le 2B_1)} \lesssim \frac{b_1^{k+1}}{|\log b_1|}.
\end{align*}
For $2\le i \le L$, we rewrite 
\begin{equation}
  (-\boldsymbol{H})^{i+2} \widetilde{\boldsymbol{\Sigma}}_{b_1} = \begin{cases}
    ((-H)^{-\frac{i}{2} + 1}\widetilde{\Sigma}_{b_1},0 )^t  & \textnormal{for even } i \\
    (0,-(-H)^{-\frac{i-1}{2} + 1}\widetilde{\Sigma}_{b_1} )^t  & \textnormal{for odd } i
  \end{cases} 
\end{equation}
from the fact $\boldsymbol{H}^{-2}=-H^{-1}$. Moreover, 
$\mathrm{supp}(\widetilde{\Sigma}_{b_1}) \subset \{|y| \ge B_0/4\}$ and for $k\ge 0$, we have the crude bound: for $B_0/4 \le y\le 2B_1 $,
\begin{equation}\label{eq:bound k sigma}
  |\mathcal{A}^{k-\overline{i}} H^{-\frac{i - \overline{i}}{2}+1} \widetilde{\Sigma}_{b_1}| \lesssim  y^{i-k-1}\frac{|\log y|}{|\log b_1|} \lesssim y^{i-k-1}.
\end{equation}
Hence for $1\le k < i \le L$, we obtain \eqref{eq:psi global k} from the following estimation
\begin{align}
  {\lVert b_1 b_i \mathcal{A}^{k-\overline{i}} H^{-\frac{i - \overline{i}}{2}+1} \widetilde{\Sigma}_{b_1}\rVert}_{L^2(|y|\le 2B_1)} &\lesssim  b_1^{i+1} {\lVert  y^{i-k-1} \rVert}_{L^2(B_0/4 \le |y|\le 2B_1)} \nonumber \\
  & \lesssim b_1^{k+1} |\log b_1|^{\gamma(i-k)}.
\end{align}
We also observe for $k \ge i$,
\begin{equation}
  \mathcal{A}^{k-\overline{i}} H^{-\frac{i - \overline{i}}{2}+1} \widetilde{\Sigma}_{b_1}= \mathcal{A}^{k-i} H \widetilde{\Sigma}_{b_1}, 
\end{equation}
the sharp bounds
\begin{equation}\label{eq:sharp pointwise sigma}
  |H \widetilde{\Sigma}_{b_1}| \lesssim \frac{\mathbf{1}_{y\ge B_0/4}}{|\log b_1|} \frac{1}{y} ,\quad  |\mathcal{A}^j H \widetilde{\Sigma}_{b_1}| \lesssim \frac{\mathbf{1}_{y\sim B_0}}{B_0^{j+1} |\log b_1|} ,\quad j\ge 1
\end{equation}
imply \eqref{eq:psi global k}, \eqref{eq:psi global L} and \eqref{eq:psi global L+1}:
\[{\lVert b_1 b_i \mathcal{A}^{k-i} H \widetilde{\Sigma}_{b_1}\rVert}_{L^2(|y|\le 2B_1)} \lesssim  \frac{b_1^{i+1}}{|\log b_1|}   {\lVert  y^{i-k-1} \rVert}_{L^2(B_0/4 \le |y|\le 2B_1)}   \lesssim \frac{b_1^{k+1}}{|\log b_1|^{\frac{1}{2}}},\]
\[{\lVert b_1 b_i \mathcal{A}^{L+1-i} H \widetilde{\Sigma}_{b_1}\rVert}_{L^2(|y|\le 2B_1)} \lesssim  \frac{b_1^{i+1}}{B_0^{L+1-i} |\log b_1|}    \lesssim \frac{b_1^{L+2}}{|\log b_1|}.\]
%and \eqref{eq:psi global L}: for even $i$,
%\[{\lVert b_1 b_i \mathcal{A}^{L-i} H \widetilde{\Sigma}_{b_1}\rVert}_{L^2(|y|\le 2B_1)} \lesssim  \frac{b_1^{i+1}}{B_0^{L-i} |\log b_1|}    \lesssim \frac{b_1^{L+1}}{|\log b_1|}. \]
The logarithmic weighted bounds \eqref{eq:psi log}, \eqref{eq:psi dot log} come from the above estimation with the trivial bound $|\log y /\log b_1| \lesssim 1$ on $B_0/4\le y\le 2B_1$ and the fact that the errors in the RHS of \eqref{eq:psi1} are supported in $y\ge B_0/4$. This support property also yields the improved local bound \eqref{eq:psi local} by choosing $b^*(M)$ small enough.

Now, we move to the last error: $\boldsymbol{R}/y^2$. Recall \eqref{def:remainder of N}, we observe that $\boldsymbol{R}/y^2=(0,R/y^2)$ has two parts: sum of monomials like ${P}_{2i}/y^2$ and nonlinear terms 
\[\frac{1}{y^2}N_0(\alpha_b) \alpha_b^{\frac{L+3}{2}}.\]

For the monomial part, we can borrow the calculation of $P_{2i}/y^2$: \eqref{eq:b order monomial} and \eqref{eq:far away monomial}. Under the range $|J|_1=j$, $|J|_2 \ge L+3$, $2\le j \le \frac{L+1}{2}$, those $k$-th suitable derivatives (i.e. $\mathcal{A}^k$) have the pointwise bounds
\begin{equation}\label{eq:pointwise monomial}
  \begin{cases}
    b_1^{L+3} & \textnormal{for }y\le 1, \\
    b_1^{|J|_2} y^{|J|_2 -k-5}|\log y|^C & \textnormal{for }1\le y\le 2B_1,
  \end{cases}
\end{equation}
we simply obtain from \eqref{eq:psi global k} to \eqref{eq:psi local} via integrating the above bound. It remains to estimate the nonlinear term. For $y\le 1$, we utilize the parity of $f^{(\frac{L+3}{2})}(Q)$ and $\alpha_b$. We already know that $\alpha_b$ is an odd function with the leading term $O(b_1^2) y^3$ and the parity of $f^{(\frac{L+3}{2})}(Q)$ is opposite of that of $\frac{L+3}{2}$, $N_0(\alpha_b)\alpha_b^{\frac{L+3}{2}}/y^2$ is an odd function with the leading term $O(b_1^{L+3}) y^{3\frac{L+3}{2}-1 -\overline{\frac{L+3}{2}}}$. Hence for $1\le k \le L$,
\[{\left\lVert \mathcal{A}^k\left(\frac{N_0(\alpha_b) }{y^2} \alpha_b^{\frac{L+3}{2}} \right) \right\rVert}_{L^{\infty}(y\le 1)} \lesssim b_1^{L+3}. \]
For $1\le y\le 2B_1$, the simple bound
\[|\partial_y^k (Q+ \tau \alpha_b)| \lesssim \frac{|\log b_1|^C}{y^{k+1}}, \quad k\ge 1\] 
implies 
\[|N_0(\alpha_b)|\lesssim 1,\quad  |\partial_y^k N_0(\alpha_b)| \lesssim \frac{|\log b_1|^C}{y^{k+1}} \quad \textnormal{for } k\ge 1.\]
 From the Leibniz rule and the crude bound $|\partial_y^k \alpha_b| \lesssim b_1^2 |\log b_1| y^{1-k}$, we have
 \begin{equation}\label{eq:pointwise non}
  \left\lvert \mathcal{A}^k\left(\frac{N_0(\alpha_b) }{y^2} \alpha_b^{\frac{L+3}{2}} \right) \right\rvert \lesssim \sum_{j=0}^k \frac{|\partial_y^j (N_0(\alpha_b) \alpha_b^{\frac{L+3}{2}}) |}{y^{2+k-j}}  \lesssim b_1^{L+3} |\log b_1|^C y^{\frac{L+3}{2}-2-k}
 \end{equation}
for $0\le k \le L$, the above pointwise bound yields from \eqref{eq:psi global k} to \eqref{eq:psi local} via integration. 
\end{proof}

%%%%%%%%%%%%%%%%%%%%%%%%%%%%%%%%%%%%%%%%%%%%%%%%%%%%%%%%%%%%%%%%%%%%%%%%%%%%%%%%%%%%%%%%%%%%%%%%%%%%%%%%%

\subsection{Localization of the approximate profile}
In the previous construction, we observe that the blow-up profile does not approximate the solution of \eqref{eq:modpsi} on the region $y\ge 2B_1$. Hence, it is necessary to cut off the overgrowth of each tail.
\begin{proposition}[Localization of the approximate profile]\label{prop:local approx}
 Assume the hypotheses of Proposition \ref{prop:approx profile} and assume moreover the a priori bounds
\begin{equation}\label{eq:a priori for local}
  |(b_1)_s| \lesssim b_1^2,\quad |b_L| \lesssim \frac{b_1^L}{|\log b_1|} \textnormal{ when }\ell=L-1.
\end{equation}
Then the localized profile $\tilde{\boldsymbol{Q}}_b $ given by
\begin{equation}
  \tilde{\boldsymbol{Q}}_b = \boldsymbol{Q} + \chi_{B_1} \boldsymbol{\alpha}_b
\end{equation}
drives the following equation:
\begin{equation}
  \partial_s \tilde{\boldsymbol{Q}}_{b}-\boldsymbol{F}(\tilde{\boldsymbol{Q}}_b) + b_1 \boldsymbol{\Lambda} \tilde{\boldsymbol{Q}}_b = \chi_{B_1}\mathbf{Mod}(t) +\tilde{\boldsymbol{\psi}}_b 
\end{equation}
where $\mathbf{Mod}(t)$ was defined in \eqref{def:mod} and $\tilde{\boldsymbol{\psi}}_b=(\tilde{\psi}_b,\dot{\tilde{\psi}}_b)^t$ satisfies the bounds:

(i) Global bound:
\begin{align}
  {}^{\forall} 2\le k \le L-1,\quad {\lVert \mathcal{A}^k{\tilde{\psi}}_b \rVert}_{L^2} + {\lVert \mathcal{A}^{k-1}{\dot{\tilde{\psi}}}_b \rVert}_{L^2} &\lesssim b_1^{k+1}|\log b_1|^C, \label{eq:psit global k} \\
  {\lVert \mathcal{A}^L{\tilde{\psi}}_b \rVert}_{L^2} + {\lVert \mathcal{A}^{L-1}{\dot{\tilde{\psi}}}_b \rVert}_{L^2} &\lesssim b_1^{L+1}|\log b_1|, \label{eq:psit global L} \\
  {\lVert \mathcal{A}^{L+1}{\tilde{\psi}}_b \rVert}_{L^2} + {\lVert \mathcal{A}^{L}{\dot{\tilde{\psi}}}_b \rVert}_{L^2} &\lesssim \frac{b_1^{L+2}}{|\log b_1|}\label{eq:psit global L+1}. %\\
 %{\lVert \mathcal{A}^L \tilde{\psi}_b \rVert}_{L^2} &\lesssim  \frac{b_1^{L+1}}{|\log b_1|} \label{eq:psit global L}
\end{align}

(ii) Logarithmic weighted bound: for $m\ge 1$ and $0\le k \le m$,
\begin{align}
 {\left\lVert \frac{1+|\log y| }{1+y^{m-k}} \mathcal{A}^k \tilde{\psi}_b \right\rVert}_{L^2} &\lesssim b_1^{m+1}|\log b_1|^C ,\quad m \le L+1,\label{eq:psit log}\\
 {\left\lVert \frac{1+|\log y| }{1+y^{m-k}} \mathcal{A}^k \dot{\tilde{\psi}}_b \right\rVert}_{L^2} &\lesssim b_1^{m+2}|\log b_1|^C ,\quad m\le L.\label{eq:psit dot log}
\end{align}

(iii) Improved local bound:
\begin{equation}\label{eq:psit local}
  {}^{\forall} 2\le k \le L+1,\quad {\lVert \mathcal{A}^k{\tilde{\psi}}_b \rVert}_{L^2(|y|\le 2M)} + {\lVert \mathcal{A}^{k-1}{\dot{\tilde{\psi}}}_b \rVert}_{L^2(|y|\le 2M)} \lesssim C(M) b_1^{L+3}.
\end{equation}
\end{proposition}
\begin{remark}
This proposition says that our cutoff function $\chi_{B_1}$ does not affect the estimates from \eqref{eq:psi global k} to \eqref{eq:psi local} in Proposition \ref{prop:approx profile}. Although such bounds came from integrating over the region $|y|\le 2B_1$, there are two main reasons why this is possible. First, we do not need to keep track of logarithmic weight $|\log b_1|$ except for \eqref{eq:psi global L+1} corresponding to the highest order derivative. Second, \eqref{eq:psi global L+1} was derived from the sharp pointwise bound \eqref{eq:sharp pointwise sigma}, which only depends on $B_0$. Thus, $B_1=|\log b_1|^{\gamma}/b_1$ just needs to be large enough to obtain \eqref{eq:psit global L+1} by raising $\gamma$. 
\end{remark}
\begin{proof}
  Note that $\tilde{\boldsymbol{\psi}}_b=\boldsymbol{\psi}_b$ on $|y|\le B_1$, \eqref{eq:psi local} directly implies the local bound \eqref{eq:psit local}. For the other estimates,
we will prove the global bounds \eqref{eq:psit global k}, \eqref{eq:psit global L+1} first, and the less demanding logarithmic weighted bounds \eqref{eq:psit log}, \eqref{eq:psit dot log} later. By a straightforward calculation, $\tilde{\boldsymbol{\psi}}_b$ is given by
\begin{align}
  \tilde{\boldsymbol{\psi}}_b&= \chi_{B_1}\boldsymbol{\psi}_b + (\partial_s(\chi_{B_1})+ b_1(y\chi')_{B_1})\boldsymbol{\alpha}_b +b_1(1-\chi_{B_1})\boldsymbol{\Lambda} \boldsymbol{Q} \label{eq:modified psi}  \\
  &-\begin{pmatrix}
    0 \\ \Delta (\chi_{B_1}\alpha_b)-\chi_{B_1}\Delta(\alpha_b)
    \end{pmatrix} -\frac{1}{y^2} \begin{pmatrix}
      0 \\ f(\tilde{Q}_b)-f(Q)-\chi_{B_1}(f(Q_b)-f(Q))
      \end{pmatrix}. \label{eq:modified psi f}
\end{align}
Before we estimate $\chi_{B_1}\boldsymbol{\psi}_b$ in the RHS of \eqref{eq:modified psi}, we introduce a useful asymptotics of cutoff:
\begin{equation}\label{eq:cutoff asymptotic}
  \mathcal{A}^k(\chi_{B_1} f) = \chi_{B_1} \mathcal{A}^k f + \mathbf{1}_{y\sim B_1}\sum_{j=0}^{k-1} O(y^{-(k-j)}) \mathcal{A}^j f.
\end{equation}
 Applying the above asymptotics to $\chi_{B_1}\boldsymbol{\psi}_b$, we get from Proposition \ref{prop:approx profile} that we only need to estimate the errors localized in $y\sim B_1$. From \eqref{eq:pointwise homogeneous}, \eqref{eq:bound k t1}, \eqref{eq:bound k sigma}, \eqref{eq:pointwise monomial} and \eqref{eq:pointwise non}, we obtain the following pointwise bounds: for $y\sim B_1$ and $0\le j \le k$,
\begin{align*}
  |y^{-(k-j)} \mathcal{A}^{j} \psi_{b_1} | \lesssim \sum_{i=1}^{\frac{L-1}{2}} b_1^{2i+1} y^{2i-k-1}  \lesssim b_1^{k+1}|\log b_1|^{\gamma(L-1-k)} B_1^{-1}
\end{align*}
and
\begin{align*}
   |y^{-(k-1-j)} \mathcal{A}^{j} \dot{\psi}_{b_1} | &\lesssim  \sum_{i=1}^{\frac{L+1}{2}} b_1^{2i} y^{2i-k-2} + \frac{b_1^{L+3}y^{L+1-k}}{|\log b_1|^{2\gamma + 1}}  + (b_1^{k+4} + b_1^{\frac{L+3}{2} + k+1 }) |\log b_1|^C \\
  &\lesssim b_1^{k+1}|\log b_1|^{\gamma(L-k)} B_1^{-1}. 
\end{align*}
These pointwise bounds directly imply the global bounds \eqref{eq:psit global k}, \eqref{eq:psit global L} and \eqref{eq:psit global L+1} if we choose $\gamma \ge 1$. 

For the second term in the RHS of \eqref{eq:modified psi}, we recall
\[\boldsymbol{\alpha}_b = \begin{pmatrix}
  \alpha_b \\
  \dot{\alpha}_b
\end{pmatrix}
  = \begin{pmatrix}
  \sum_{i=1,\textnormal{even}}^{L} b_i  T_i  + \sum_{i=2,\textnormal{even}}^{L+2}  S_i \\
  \sum_{i=1,\textnormal{odd}}^{L} b_i  T_i  + \sum_{i=2,\textnormal{odd}}^{L+2}   S_i
\end{pmatrix}.\]
From the a priori bound $|b_{1,s}|\lesssim b_1^2$, 
\begin{equation}\label{eq:derivative chi}
  |\partial_s(\chi_{B_1})+ b_1(y\chi')_{B_1} |\lesssim \left(\frac{|b_{1,s}|}{b_1} +b_1 \right) |(y \chi')_{B_1}|\lesssim b_1 \mathbf{1}_{y \sim B_1}.
\end{equation}
One can easily check that \eqref{eq:cutoff asymptotic} still holds even if we replace the cutoff function $\chi_{B_1}$ to other cutoff functions supported in $y\sim B_1$. Hence, the cutoff asymptotics \eqref{eq:cutoff asymptotic} and the admissibility of $\boldsymbol{T}_i$ imply for $1\le i \le L$,
\begin{align}
    {\left\lVert b_i \mathcal{A}^{k-\overline{i}} (\partial_s(\chi_{B_1}) + b_1(y\chi')_{B_1} ) T_i  \right\rVert}_{L^2} & \lesssim  \sum_{j=0}^{k-\overline{i}} b_1|b_i| {\left\lVert y^{-(k-j-\overline{i})} \mathcal{A}^j T_i  \right\rVert}_{L^2 (y\sim B_1)} \nonumber \\
  &\lesssim  b_1|b_i| {\left\lVert y^{i-k-1}|\log y|  \right\rVert}_{L^2 (y\sim B_1)} \nonumber \\
  &\lesssim  b_1^{k+1-i}|b_i| |\log b_1|^{\gamma(i-k)+1},\label{eq:bound k T_i}
\end{align}
and for $2\le i \le L+2$, Lemma \ref{lem:admissible bounds} implies
\begin{align}
  {\left\lVert  \mathcal{A}^{k-\overline{i}} (\partial_s(\chi_{B_1}) + b_1(y\chi')_{B_1} )S_i  \right\rVert}_{L^2} & \lesssim  b_1 \sum_{j=0}^{k-\overline{i}}  {\left\lVert y^{-(k-j-\overline{i})} \mathcal{A}^j S_i  \right\rVert}_{L^2 (y\sim B_1)} \nonumber \\
  & \lesssim  b_1^{k+1} |\log b_1|^{\gamma(i-k-2)-1}, \label{eq:bound k S_i}
\end{align}
we obtain the global bounds \eqref{eq:psit global k} and \eqref{eq:psit global L}. In \eqref{eq:bound k T_i}, we cannot cancel $\log y$ from $T_i$, the additional $|\log b_1|$ appears. Thus, we need to choose $\gamma = 1+\overline{\ell}$ for the case $(k,i)=(L+1,L)$, which corresponds to \eqref{eq:psit global L+1}. We note that $\gamma=1$ when $\ell=L-1$ since we have the additional $|\log b_1|$ gain of $b_L$ from \eqref{eq:a priori for local}.  

The third term in \eqref{eq:modified psi} can be estimated
\[{\left\lVert b_1 \mathcal{A}^k (1-\chi_{B_1}) \Lambda Q \right\rVert}_{L^2} \lesssim b_1  {\left\lVert y^{-k-1} \right\rVert}_{L^2(y\ge B_1)} \lesssim \frac{b_1^{k+1}}{|\log b_1|^{\gamma k}}.\]

Finally, we compute \eqref{eq:modified psi f}
\[\Delta (\chi_{B_1}\alpha_b)-\chi_{B_1}\Delta(\alpha_b) = (\Delta \chi_{B_1})\alpha_b + 2 \partial_y(\chi_{B_1})\partial_y(\alpha_b),\]
\[ f(\tilde{Q}_b)-f(Q)-\chi_{B_1}(f(Q_b)-f(Q)) = \chi_{B_1} \alpha_b \int_0^1 [f'(Q+\tau \chi_{B_1}{\alpha}_b) -f'(Q+\tau {\alpha}_b) ] d\tau,\]
each term is localized in $y\sim B_1$. In this region, the rough bounds $|f^{(k)}| \lesssim 1$ and $|\partial_y^k Q| + |\partial_y^{k}\chi_{B_1}|\lesssim y^{-k}$ yield 
\[\left\lvert\frac{\partial^k}{\partial y^k}\left( \Delta (\chi_{B_1}\alpha_b)-\chi_{B_1}\Delta(\alpha_b) + \frac{f(\tilde{Q}_b)-f(Q)-\chi_{B_1}(f(Q_b)-f(Q)) }{y^2}  \right) \right\rvert \lesssim  \frac{|\alpha_b|}{y^{k+2}},\]
we can borrow the estimation of $\partial_s(\chi_{B_1})\alpha_b$, namely \eqref{eq:bound k T_i} and \eqref{eq:bound k S_i}. 

The logarithmic weighted bounds \eqref{eq:psit log}, \eqref{eq:psit dot log} basically come from the fact $|\log y| \sim |\log b_1|$ on $y\sim B_1$, we further use the decay property $|\log y|^C/y \to 0$ as $y \to \infty$ for the third term in the RHS of \eqref{eq:modified psi}. 
\end{proof}
We also introduce another localization that depends on $\ell$ to verify the further regularity in Remark \ref{rmk:further regularity}.
\begin{proposition}[Localization for the case when $\ell=L$]\label{prop:2nd local approx}
  Assume the hypotheses of Proposition \ref{prop:local approx}. Then the localized profile $\hat{\boldsymbol{Q}}_b$ given by
  \begin{equation}
    \hat{\boldsymbol{Q}}_b =  \tilde{\boldsymbol{Q}}_b + \boldsymbol{\zeta}_b :=  \tilde{\boldsymbol{Q}}_b + (\chi_{B_0} - \chi_{B_1} ) b_L \boldsymbol{T}_L 
  \end{equation}
  drives the following equation:
  \begin{equation}
    \partial_s \hat{\boldsymbol{Q}}_{b}-\boldsymbol{F}(\hat{\boldsymbol{Q}}_b) + b_1 \boldsymbol{\Lambda} \hat{\boldsymbol{Q}}_b = \widehat{\mathbf{Mod}}(t) +\hat{\boldsymbol{\psi}}_b 
  \end{equation}
  where $\widehat{\mathbf{Mod}}(t)$ is given by
  \begin{equation}
    \widehat{\mathbf{Mod}}(t) = \chi_{B_1} {\mathbf{Mod}}(t) + (\chi_{B_0}-\chi_{B_1}) \left( (b_L)_s + (L-1 + c_{b,L})b_1b_L\right)\boldsymbol{T}_L
  \end{equation}
   and $\hat{\boldsymbol{\psi}}_b=(\hat{\psi}_b,\dot{\hat{\psi}}_b)^t$ satisfies the bounds: 
   \begin{align}
    {\lVert \mathcal{A}^{L}({\hat{\psi}}_b -(\chi_{B_1}-\chi_{B_0}) b_L T_{L-1} )\rVert}_{L^2} & \lesssim b_1^{L+1} \label{eq:psih1 global L} \\
    {\lVert \mathcal{A}^{L-1}({\dot{\hat{\psi}}}_b -(\partial_s \chi_{B_0} + b_1(y\chi')_{B_0})b_L T_L)\rVert}_{L^2} &\lesssim b_1^{L+1} \label{eq:dpsih1 global L}
   \end{align}
\end{proposition}
\begin{proof}
  Note that $\boldsymbol{F}(\tilde{\boldsymbol{Q}}_b+\boldsymbol{\zeta}_b )-\boldsymbol{F}(\tilde{\boldsymbol{Q}}_b )=(\chi_{B_0}-\chi_{B_1})b_L \boldsymbol{T}_{L-1}$. From \eqref{eq:b_i identity} and \eqref{def:mod}, we have
  \begin{align}
    &\partial_s \hat{\boldsymbol{Q}}_{b}-\boldsymbol{F}(\hat{\boldsymbol{Q}}_b) + b_1 \boldsymbol{\Lambda} \hat{\boldsymbol{Q}}_b  \nonumber \\
    &= \chi_{B_1} \mathbf{Mod}(t) + \tilde{\boldsymbol{\psi}}_b + \partial_s \boldsymbol{\zeta}_b - (\boldsymbol{F}(\tilde{\boldsymbol{Q}}_b+\boldsymbol{\zeta}_b )-\boldsymbol{F}(\tilde{\boldsymbol{Q}}_b ))+ b_1 \boldsymbol{\Lambda} \boldsymbol{\zeta}_b \nonumber \\
    &= \widehat{\mathbf{Mod}}(t)  + b_1b_L(\chi_{B_0}-\chi_{B_1})\{(-\boldsymbol{H})^{L+2}\widetilde{\Sigma}_{b_1} + \boldsymbol{\theta}_L\}  \label{eq:2nd mod local} \\
    &+ \tilde{\boldsymbol{\psi}}_b -(\partial_s(\chi_{B_1})+ b_1(y\chi')_{B_1})b_L \boldsymbol{T}_L \label{eq:2nd psi local} \\
    & + (\partial_s(\chi_{B_0}) + b_1(y\chi')_{B_0})b_L \boldsymbol{T}_L + (\chi_{B_1}-\chi_{B_0})b_L \boldsymbol{T}_{L-1}. \label{eq:2nd psi main}
  \end{align}
From the above identity, we can see that \eqref{eq:2nd psi main} is exactly subtracted from $\hat{\boldsymbol{\psi}}_b$ in \eqref{eq:psih1 global L} and \eqref{eq:dpsih1 global L}. Hence,
 we need to estimate the second term of \eqref{eq:2nd mod local} and \eqref{eq:2nd psi local}. We point out that the logarithm weight $|\log b_1|$ in \eqref{eq:psit global L} comes from the estimate \eqref{eq:bound k T_i} when $i=L$, which is eliminated in \eqref{eq:2nd psi local}. For the second term of \eqref{eq:2nd mod local}, we can borrow the bound \eqref{eq:sharp pointwise sigma} and Lemma \ref{lem:admissible bounds}.
\end{proof}
\begin{proposition}[Localization for the case when $\ell=L-1$]\label{prop:3rd local approx}
  Assume the hypotheses of Proposition \ref{prop:local approx}. Then the localized profile $\hat{\boldsymbol{Q}}_b$ given by
  \begin{equation}
    \hat{\boldsymbol{Q}}_b =  \tilde{\boldsymbol{Q}}_b + \boldsymbol{\zeta}_b :=  \tilde{\boldsymbol{Q}}_b + (\chi_{B_0} - \chi_{B_1} )(b_{L-1}\boldsymbol{T}_{L-1} +  b_L \boldsymbol{T}_L )
  \end{equation}
  drives the following equation:
  \begin{equation}
    \partial_s \hat{\boldsymbol{Q}}_{b}-\boldsymbol{F}(\hat{\boldsymbol{Q}}_b) + b_1 \boldsymbol{\Lambda} \hat{\boldsymbol{Q}}_b = \widehat{\mathbf{Mod}}(t) +\hat{\boldsymbol{\psi}}_b 
  \end{equation}
  where $\widehat{\mathbf{Mod}}(t)$ is given by
  \begin{align*}
    \widehat{\mathbf{Mod}}(t) = \chi_{B_1} {\mathbf{Mod}}(t) &+ (\chi_{B_0}-\chi_{B_1}) \left( (b_{L-1})_s + (L-2 + c_{b,L-1})b_1b_{L-1}\right)\boldsymbol{T}_{L-1} \\
    &+ (\chi_{B_0}-\chi_{B_1}) \left( (b_L)_s + (L-1 + c_{b,L})b_1b_L\right)\boldsymbol{T}_L
  \end{align*}
   and $\hat{\boldsymbol{\psi}}_b=(\hat{\psi}_b,\dot{\hat{\psi}}_b)^t$ satisfies the bounds: 
   \begin{align}
    {\lVert \mathcal{A}^{L-1}({\hat{\psi}}_b -(\partial_s \chi_{B_0} + b_1(y\chi')_{B_0})b_{L-1} T_{L-1} -(\chi_{B_1}-\chi_{B_0}) b_L T_{L-1} )\rVert}_{L^2} & \lesssim b_1^{L} \label{eq:psih2 global L-1} \\
    {\lVert \mathcal{A}^{L-2}({\dot{\hat{\psi}}}_b -(\partial_s \chi_{B_0} + b_1(y\chi')_{B_0})b_L T_L + b_{L-1}H(\chi_{B_1}-\chi_{B_0})  T_{L})\rVert}_{L^2} &\lesssim b_1^{L} \label{eq:dpsih2 global L-1}
   \end{align}
\end{proposition}
\begin{remark}
We point out that Propositions \ref{prop:2nd local approx} and \ref{prop:3rd local approx} provide improved bounds \eqref{eq:psih1 global L}, \eqref{eq:dpsih1 global L}, \eqref{eq:psih2 global L-1} and \eqref{eq:dpsih2 global L-1} compared to \eqref{eq:psit global k} and \eqref{eq:psit global L} in Proposition \ref{prop:local approx}. These improved bounds will be essential to prove the monotonicity formula \eqref{eq:l energy mono} later.
\end{remark}
\begin{proof}
  Note that $\boldsymbol{F}(\tilde{\boldsymbol{Q}}_b+\boldsymbol{\zeta}_b )-\boldsymbol{F}(\tilde{\boldsymbol{Q}}_b )=-\boldsymbol{H} \boldsymbol{\zeta}_b - \boldsymbol{NL}(\boldsymbol{\zeta}_b) - \boldsymbol{L}(\boldsymbol{\zeta}_b) $ where
  \begin{align}
    \boldsymbol{NL}(\boldsymbol{\zeta}_b)= \begin{pmatrix}
      0 \\
      NL(\zeta_b)
    \end{pmatrix} &:= \frac{1}{y^2}\begin{pmatrix}
      0 \\
      f(\tilde{Q}_b+{\zeta}_b) -f(\tilde{Q}_b) - f'(\tilde{Q}_b){\zeta}_b
    \end{pmatrix}, \\
    \boldsymbol{L}(\boldsymbol{\zeta}_b)=\begin{pmatrix}
      0 \\
      L(\zeta_b)
    \end{pmatrix} &:= \frac{1}{y^2}\begin{pmatrix}
      0 \\
      (f'(\tilde{Q}_b) -f'(Q)){\zeta}_b
    \end{pmatrix}.
  \end{align}
  From \eqref{eq:b_i identity} and \eqref{def:mod}, we have
  \begin{align}
    &\partial_s \hat{\boldsymbol{Q}}_{b}-\boldsymbol{F}(\hat{\boldsymbol{Q}}_b) + b_1 \boldsymbol{\Lambda} \hat{\boldsymbol{Q}}_b  \nonumber \\
    &= \chi_{B_1} \mathbf{Mod}(t) + \tilde{\boldsymbol{\psi}}_b + \partial_s \boldsymbol{\zeta}_b - (\boldsymbol{F}(\tilde{\boldsymbol{Q}}_b+\boldsymbol{\zeta}_b )-\boldsymbol{F}(\tilde{\boldsymbol{Q}}_b ))+ b_1 \boldsymbol{\Lambda} \boldsymbol{\zeta}_b \nonumber \\
    &= \widehat{\mathbf{Mod}}(t) + b_1b_{L-1}(\chi_{B_0}-\chi_{B_1})\{(-\boldsymbol{H})^{L+1}\widetilde{\Sigma}_{b_1} + \boldsymbol{\theta}_{L-1}\}  \nonumber \\
    &  + b_1b_L(\chi_{B_0}-\chi_{B_1})\{(-\boldsymbol{H})^{L+2}\widetilde{\Sigma}_{b_1} + \boldsymbol{\theta}_L\} + \boldsymbol{NL}(\boldsymbol{\zeta}_b) + \boldsymbol{L}(\boldsymbol{\zeta}_b) \label{eq:3rd mod2 local} \\
    &+ \tilde{\boldsymbol{\psi}}_b -(\partial_s(\chi_{B_1})+ b_1(y\chi')_{B_1})(b_{L-1}\boldsymbol{T}_{L-1}+b_L \boldsymbol{T}_L)  \nonumber \\
    & + (\partial_s(\chi_{B_0}) + b_1(y\chi')_{B_0})b_L \boldsymbol{T}_L + (\chi_{B_1}-\chi_{B_0})b_L \boldsymbol{T}_{L-1}+ \boldsymbol{H}\boldsymbol{\zeta}_b.\nonumber
  \end{align}
  Based on the proof of the previous proposition, it suffices to show that
  \[{\lVert \mathcal{A}^{L-2}NL(\zeta_b)\rVert}_{L^2}+{\lVert \mathcal{A}^{L-2}L(\zeta_b)\rVert}_{L^2} \lesssim b_1^{L}, \]
  which come from the following crude pointwise bounds in $B_0 \le y \le 2 B_1$: for $k \ge 0$,
  \[|\mathcal{A}^{k} NL(\zeta_b)|\lesssim b_1^{2L-2} y^{2L-6-k}|\log b_1|^C,\quad  |\mathcal{A}^{k} L(\zeta_b)|\lesssim b_1^L y^{L-4-k}|\log b_1|^C.\qedhere\] 
\end{proof}
%%%%%%%%%%%%%%%%%%%%%%%%%%%%%%%%%%%%%%%%%%%%%%%%%%%%%%%%%%%%%%%%%%%%%%%%%%%%%%%%%%%%%%%%%%%%%%%%%%%%%%%%%%%
\subsection{Dynamical laws of $b=(b_1,\dots,b_L)$}\label{subsec:law of b}
As previously mentioned, the blow-up rate is determined by the evolution of the vector $b$, we figure out its dynamical laws from \eqref{def:mod}: for $1\le k \le L$,
\begin{equation}\label{eq:exact mod eqn}
  (b_k)_s=b_{k+1}-\left(k-1+\frac{1}{(1+\delta_{1k})\log s}\right)b_1b_k,\quad b_{L+1}=0.
\end{equation}
One can check that the above system has $L$ independent solutions characterized by the number of nonzero coordinates: for $1\le k \le L$, $b=(b_1,\dots,b_{k},0,\dots,0)$. Here, we adopt two special solutions (recall that there are two $\ell$s that can achieve the same $L$) among them.
\begin{lemma}[Special solutions for the $b$ system]
  For all $\ell \ge 2$, the vector of functions
  \begin{equation}\label{eq:exact b formula}
    b_k^{e}(s)=\frac{c_k}{s^k}+\frac{d_k}{s^k\log s} \textnormal{ for } 1\leq k\leq \ell, \quad b_{k}^{e}\equiv 0 \textnormal{ for } k>\ell
  \end{equation}
solves \eqref{eq:exact mod eqn} approximately: for $1\le k \le L$,
\begin{equation}\label{eq:exact b system}
  (b_k^{e})_s+\left(k-1+\frac{1}{(1+\delta_{1k})\log s}\right)b^e_1b^e_k-b^e_{k+1}=O\left(\frac{1}{s^{k+1}(\log s)^2}\right), \textnormal{ as } s\to +\infty
\end{equation}
where the sequence $(c_k,d_k)_{k=1,\dots,\ell}$ is given by
\begin{equation}\label{eq:c recurrence}
  c_1= \frac{\ell}{\ell-1}, \quad c_{k+1}=-\frac{\ell-k}{\ell-1} c_k,\quad 1\le k \le \ell
\end{equation}
and for $2\le k \le \ell-1$,
\begin{equation}
  d_1= -\frac{\ell}{(\ell-1)^2},\quad d_2=-d_1+\frac{1}{2}c_1^2,\quad d_{k+1}=-\frac{\ell-k}{\ell-1}d_k+\frac{\ell(\ell-k)}{(\ell-1)^2} c_k. \label{eq:d recurrence}
\end{equation}
\end{lemma}
\begin{remark}
  The recurrence relations \eqref{eq:c recurrence} and \eqref{eq:d recurrence} are obtained by substituting \eqref{eq:exact b formula} into \eqref{eq:exact b system} and comparing the coefficients of $s^{-k}$ and $(s^{k}\log s)^{-1}$, which yields the proof.
\end{remark}
For our $b$ system to drive like the special solution $b^e$, we should control the fluctuation
\begin{equation}\label{eq:difference b_k}
  \frac{U_k(s)}{s^k(\log s)^{\beta}}:=b_k(s)-b_k^e(s)\textnormal{ for } 1\le k \le \ell.
\end{equation}
Here, \eqref{eq:exact b formula} and \eqref{eq:exact b system} restrict the range of $\beta$ to $1<\beta<2$, we will choose $\beta=5/4$ later. The next lemma provides the evolution of $U=(U_1,\dots,U_{\ell})$ from \eqref{eq:exact mod eqn}.
\begin{lemma}[Evolution of $U$]
  Let $b_k(s)$ be a solution to \eqref{eq:exact mod eqn} and $U$ be defined by \eqref{eq:difference b_k}. Then $U$ solves 
  \begin{align}\label{eq:linearization of b}
     s(U)_s= A_{\ell}U + O \left( \frac{1}{(\log s)^{2-\beta}} + \frac{|U|+|U|^2}{\log s} \right) ,
  \end{align}
  where the $\ell \times \ell$ matrix $A_{\ell}$ has of the form:
  \begin{equation}\label{def:matrix A}
    A_{\ell}=\begin{pmatrix} 1 & 1 &  & & &   \\ -c_2 & \frac{\ell-2}{\ell-1} & 1 & & (0) &  \\ -2c_3 & & \frac{\ell-3}{\ell-1} & 1 & &  \\ \vdots &  & & \ddots & \ddots &   \\ -(\ell-2)c_{\ell-1} & & (0) & & \frac{1}{\ell-1} & 1   \\  -(\ell-1)c_{\ell} & & & & & 0     \end{pmatrix}. 
  \end{equation}
Moreover, there exists an invertible matrix $P_{\ell}$ such that $A_{\ell}=P_{\ell}^{-1}D_{\ell}P_{\ell}$ with
\begin{equation}\label{eq:diagonalization}
  D_{\ell}=\begin{pmatrix} -1 &  &  & & &   \\  & \frac{2}{\ell-1} &  & & (0) &  \\  & & \frac{3}{\ell-1} &  & &  \\  &  & & \ddots &  &   \\  & (0) &  & & 1 &   \\  & & & & & \frac{\ell}{\ell-1}     \end{pmatrix}. 
\end{equation}
%The last rows and columns of $A_{\ell}$, $D_{\ell}$ are ignored when $\ell=L$. 
\end{lemma}
\begin{proof}
  Observing the relation
  \[(k-1)c_1-k=\frac{(k-1)\ell}{\ell-1}-k=-\frac{\ell-k}{\ell-1},\]
  we obtain \eqref{eq:linearization of b} and \eqref{def:matrix A} since
\begin{align}\label{eq:mod linearization of b}
  &(b_k)_s+\left(k-1+\frac{1}{(1+\delta_{1k})\log s}\right)b_1b_k-b_{k+1} \\
  &=\frac{1}{s^{k+1}(\log s)^{\beta}}\left[ s(U_k)_s - kU_k + O\left(\frac{|U|}{\log s} \right) \right] + O\left(\frac{1}{s^{k+1}(\log s)^2} \right) \nonumber \\
  &+\frac{1}{s^{k+1}(\log s)^{\beta}}\left[ (k-1)c_kU_1 + (k-1)c_1U_k - U_{k+1} + O\left(\frac{|U|+|U|^2}{\log s} \right) \right] \nonumber \\
  &=\frac{1}{s^{k+1}(\log s)^{\beta}}\left[ s(U_k)_s+(k-1)c_kU_1 - \frac{\ell-k}{\ell-1}U_k -U_{k+1}  \right] \nonumber\\
  &+ O\left( \frac{1}{s^{k+1}(\log s)^{2}}+\frac{|U|+|U^2|}{s^{k+1}(\log s)^{1+\beta}} \right). \nonumber
\end{align}
 \eqref{eq:diagonalization} is obtained by substituting $\alpha=1$ for the result of Lemma 2.17 in \cite{Collot2018MEM.AMS}.
\end{proof}
\begin{remark}
  Since the above process can be seen as linearizing \eqref{eq:exact mod eqn} around our special solution $b^e$, the appearance of the matrix $A_{\ell}$ is quite natural. We also note that $\ell-1$ unstable directions corresponding to $\ell-1$ positive eigenvalues yield the (formal) codimension $\ell-1$ restriction of our initial data. 
\end{remark}
%%%%%%%%%%%%%%%%%%%%%%%%%%%%%%%%%%%%%%%%%%%%%%%%%
%%%%%%%%%%%%%%%%%%%%%%%%%%%%%%%%%%%%%%%%%%%%%%%%%
\section{The trapped solutions}
Our goal in this section is to implement the blow-up dynamics constructed in the previous section into the real solution $\boldsymbol{u}$. To do this, we first decompose the solution $\boldsymbol{u}$ as the blow-up profile and the error, i.e. $\boldsymbol{u}=(\tilde{\boldsymbol{Q}}_{b}+\boldsymbol{\varepsilon})_{\lambda}=\tilde{\boldsymbol{Q}}_{b,\lambda} + \boldsymbol{w}$. For the term "error" to be meaningful, we need to control the "direction" and "size" of $\boldsymbol{w}=\boldsymbol{\varepsilon}_{\lambda}$. 

Here, $\boldsymbol{\varepsilon}$ must be orthogonal to the directions that provoke blow-up from $\tilde{\boldsymbol{Q}}_{b,\lambda}$. Such orthogonal conditions determine the modulation equations system of the dynamical parameters $b$ as designed in subsection \ref{subsec:law of b}. 

In this process, $\boldsymbol{\varepsilon}$ appears as an error that is small in some suitable norms. The smallness is required not to change the leading order evolution laws \eqref{eq:exact mod eqn}. We describe the set of initial data and the trapped conditions represented by some bootstrap bounds for such suitable norms i.e, the higher-order energies.

After establishing estimates of modulation parameters, we also establish a Lyapunov type monotonicity of the higher-order energies to close our bootstrap assumptions.
%%%%%%%%%%%%%%%%%%%%%%%%%%%%%%%%%%%%%%%%%%%%
\subsection{Decomposition of the flow} 
We recall the approximate direction $\boldsymbol{\Phi}_M$ which was defined in \cite{Collot2018MEM.AMS}. For a large constant $M>0$, we define
\begin{equation}\label{def:PhiM}
\boldsymbol{\Phi}_M=\sum_{p=0}^Lc_{p,M} \boldsymbol{H}^{*p}(\chi_M \boldsymbol{\Lambda} \boldsymbol{Q}) ,\quad  \boldsymbol{H}^*= \begin{pmatrix} 0 & H \\
  -1 & 0
  \end{pmatrix}
\end{equation}
where $c_{p,M}$ is given by
\begin{equation}
c_{0,M}=1,\quad  c_{k,M}=(-1)^{k+1} \frac{\sum_{p=0}^{k-1} c_{p,M} \langle \boldsymbol{H}^{*p}(\chi_M \boldsymbol{\Lambda} \boldsymbol{Q}),\boldsymbol{T}_k \rangle}{\langle \chi_M \boldsymbol{\Lambda} \boldsymbol{Q}, \boldsymbol{\Lambda} \boldsymbol{Q} \rangle} ,\quad 1\le k \le L.
\end{equation}
One can easily verify (see section 3.1.1 in \cite{Collot2018MEM.AMS}) that $\boldsymbol{H}^*$ is an adjoint operator of $\boldsymbol{H}$ in the sense that
\[\langle \boldsymbol{H} \boldsymbol{u},\boldsymbol{v}\rangle=\langle \boldsymbol{u},\boldsymbol{H}^*\boldsymbol{v} \rangle\]
and
$\boldsymbol{\Phi}_M=(\Phi_M,0)$ satisfies 
\begin{equation} \label{eq:PhiM properties}
     \langle \boldsymbol{\Phi}_M , \boldsymbol{\Lambda Q} \rangle =  \langle  \chi_M \boldsymbol{\Lambda Q} , \boldsymbol{\Lambda Q} \rangle  \sim 4\log M, \quad 
    |c_{p,M}| \lesssim M^p , \quad 
     ||\Phi_M||_{L^2}^2 \sim c\log M.
\end{equation}
We then obtain our desired decomposition by imposing a collection of orthogonal directions, which approximates the generalized kernel defined in Definition \ref{def:generalized ker}.
\begin{lemma}[Decomposition]
  Let $\boldsymbol{u}(t)$ be a solution to \eqref{eq:WM}  starting close enough to $\boldsymbol{Q}$ in $\mathcal{H}$. Then there exist $C^1$ functions $\lambda(t)$ and $b(t)=(b_1,\dots,b_L)$ such that $\boldsymbol{u}$ can be decomposed as
  \begin{equation}\label{eq:decomposition}
    \boldsymbol{u}=(\tilde{\boldsymbol{Q}}_{b(t)}+\boldsymbol{\varepsilon})_{{\lambda(t)}} 
    \end{equation}
    where $\tilde{\boldsymbol{Q}}_b $ is given in Proposition \ref{prop:local approx} and $\boldsymbol{\varepsilon}$ satisfies the orthogonality conditions \begin{equation}\label{eq:orthogonal conditions}
      \langle \boldsymbol{\varepsilon},\boldsymbol{H}^{*i}\boldsymbol{\Phi}_M\rangle=0, \ \textnormal{for} \ 0\leq i \leq  L .
      \end{equation}
       and an orbital stability estimate:
  \begin{equation}\label{eq:orbital stability}
    |b(t)|+{\left\lVert \boldsymbol{\varepsilon} \right\rVert}_{\mathcal{H}}  \ll 1
  \end{equation}  
\end{lemma}
\begin{remark}
   \eqref{eq:PhiM orthogonality} says that $\{\langle \cdot , \boldsymbol{H}^{*i}\boldsymbol{\Phi}_M\rangle\}_{i \ge 0}$ serves as coordinate functions on the space $\textnormal{Span}\{\boldsymbol{T}_i\}_{i \ge 0}$.
\end{remark}
\begin{proof}
  It is clear that $\boldsymbol{H}^i \boldsymbol{T}_j=0$ for $i>j$.  For $0\le i \le j$,
  \begin{align}
    \langle \boldsymbol{\Phi}_M, \boldsymbol{H}^i \boldsymbol{T}_j \rangle &= (-1)^{i}\langle \boldsymbol{\Phi}_M,\boldsymbol{T}_{j-i} \rangle \nonumber \\ 
    &= (-1)^{i}\sum_{p=0}^{j-i-1} c_{p,M}\langle \boldsymbol{H}^{*p}(\chi_M \boldsymbol{\Lambda} \boldsymbol{Q}),\boldsymbol{T}_{j-i} \rangle + (-1)^{j} c_{j-i,M} \langle \chi_M \boldsymbol{\Lambda} \boldsymbol{Q}, \boldsymbol{\Lambda} \boldsymbol{Q} \rangle\nonumber \\
    &=(-1)^j\langle \chi_M \boldsymbol{\Lambda} \boldsymbol{Q}, \boldsymbol{\Lambda} \boldsymbol{Q} \rangle \delta_{i,j} .\label{eq:PhiM orthogonality}
  \end{align}
  Now, we consider $\boldsymbol{\varepsilon}:=\boldsymbol{u}_{1/\lambda} - \tilde{\boldsymbol{Q}}_b$ as a map in the $(\lambda,b,\boldsymbol{u})$ basis. By the implicit function theorem, \eqref{eq:decomposition} is deduced from the non-degeneracy of the following Jacobian
  \begin{align}
    {\left\lvert \left(\frac{\partial  }{\partial(\lambda,b)}  \langle \boldsymbol{\varepsilon}, \boldsymbol{H}^{*i}\boldsymbol{\Phi}_M \rangle\right)_{0\le i \le L}\right\rvert}_{(\lambda,b,\boldsymbol{u})=(1,0,\boldsymbol{Q})} &=  (-1)^{L+1}{\left\lvert \left(  \langle \boldsymbol{T}_j, \boldsymbol{H}^{*i}\boldsymbol{\Phi}_M \rangle\right)_{0\le i,j \le L}\right\rvert} \nonumber\\
    &= {\left\lvert \left(  \langle \boldsymbol{\Phi}_M , \boldsymbol{H}^{i} \boldsymbol{T}_j\rangle\right)_{0\le i,j \le L}\right\rvert} \nonumber\\
    &= {\left\lvert \left( (-1)^j\langle \chi_M \boldsymbol{\Lambda} \boldsymbol{Q}, \boldsymbol{\Lambda} \boldsymbol{Q} \rangle \delta_{i,j}\right)_{0\le i,j \le L}\right\rvert} \nonumber\\
    &=(-1)^{\frac{L+1}{2}} \langle \chi_M \boldsymbol{\Lambda} \boldsymbol{Q}, \boldsymbol{\Lambda} \boldsymbol{Q} \rangle^{L+1} \neq 0 . \nonumber  \qedhere
  \end{align}
\end{proof}

%%%%%%%%%%%%%%%%%%%%%%%%%%%%%%%%%%%%%

\subsection{Equation for the error}
Based on the previously established decomposition
\[\boldsymbol{u}=\tilde{\boldsymbol{Q}}_{b(t),{\lambda (t)}}+\boldsymbol{w}=(\tilde{\boldsymbol{Q}}_{b(s)}+\boldsymbol{\varepsilon} (s))_{{\lambda(s)}}, \]
\eqref{eq:WM} turns into the evolution equation of $\boldsymbol{\varepsilon}$:
\begin{align}
  \partial_s\boldsymbol{\varepsilon} -\frac{\lambda_s}{\lambda} \boldsymbol{\Lambda} \boldsymbol{\varepsilon} + \boldsymbol{H}\boldsymbol{\varepsilon} =& - \left( \partial_s \tilde{\boldsymbol{Q}}_b - \frac{\lambda_s}{\lambda} \boldsymbol{\Lambda}\tilde{\boldsymbol{Q}}_b  \right) + \boldsymbol{F}(\tilde{\boldsymbol{Q}}_b + \boldsymbol{\varepsilon} )  + \boldsymbol{H}\boldsymbol{\varepsilon}\nonumber\\
  =&-\left( \partial_s \tilde{\boldsymbol{Q}}_b - \boldsymbol{F}(\tilde{\boldsymbol{Q}}_b  ) + b_1 \boldsymbol{\Lambda}\tilde{\boldsymbol{Q}}_b  \right) + \left(\frac{\lambda_s}{\lambda}+b_1\right)\boldsymbol{\Lambda} \tilde{\boldsymbol{Q}}_b\nonumber\\
  & + \boldsymbol{F}(\tilde{\boldsymbol{Q}}_b + \boldsymbol{\varepsilon} )-\boldsymbol{F}(\tilde{\boldsymbol{Q}}_b  )  + \boldsymbol{H}\boldsymbol{\varepsilon} \nonumber \\
  =& -\widetilde{\mathbf{Mod}}(t)  - \tilde{\boldsymbol{\psi}}_b -\boldsymbol{NL}(\boldsymbol{\varepsilon}) -\boldsymbol{L}(\boldsymbol{\varepsilon}),\label{eq:evolution epsilon}
\end{align}
where
\begin{equation}\label{eq:mod tilde}
  \widetilde{\mathbf{Mod}}(t):=\chi_{B_1}\mathbf{Mod}(t) -  \left(\frac{\lambda_s}{\lambda}+b_1\right)\boldsymbol{\Lambda} \tilde{\boldsymbol{Q}}_b , 
\end{equation}
\begin{equation}
  \boldsymbol{NL}(\boldsymbol{\varepsilon}):= \frac{1}{y^2}\begin{pmatrix}
    0 \\
    f(\tilde{Q}_b+\varepsilon) -f(\tilde{Q}_b) - f'(\tilde{Q}_b)\varepsilon
  \end{pmatrix},\ \ \boldsymbol{L}(\boldsymbol{\varepsilon}):= \frac{1}{y^2}\begin{pmatrix}
    0 \\
    (f'(\tilde{Q}_b) -f'(Q))\varepsilon
  \end{pmatrix}.
\end{equation}
For later analysis, we also employ the evolution equation of $\boldsymbol{w}$:
\begin{equation}\label{eq:evolution W}
  \partial_t \boldsymbol{w}+\boldsymbol{H}_{{\lambda}} \boldsymbol{w}= \frac{1}{\lambda}\boldsymbol{\mathcal{F}}_{\lambda},\quad \boldsymbol{\mathcal{F}}= -\widetilde{\mathbf{Mod}}(t) -\tilde{\boldsymbol{\psi}}_{b} -\boldsymbol{ NL}(\boldsymbol{\varepsilon})-\boldsymbol{L}{(\boldsymbol{\varepsilon})},
\end{equation}
where:
\begin{equation}\label{def:capital H rescaled}
\boldsymbol{H}_{{\lambda}}=\begin{pmatrix}
  0 & -1 \\
  H_{\lambda} & 0
  \end{pmatrix}:=\begin{pmatrix}
0 & -1 \\
-\Delta +r^{-2} f'(Q_{\lambda}) & 0
\end{pmatrix}, 
\end{equation}
We notice that the $\boldsymbol{NL}$ and $\boldsymbol{L}$ terms are situated on the second coordinate:
\begin{equation}\label{eq:localisation NL et L}
  \boldsymbol{NL}(\boldsymbol{\varepsilon})=\begin{pmatrix}0 \\ NL(\varepsilon) \end{pmatrix},  \ \  \boldsymbol{L}(\boldsymbol{\varepsilon})=\begin{pmatrix}0 \\ L(\varepsilon) \end{pmatrix}.
\end{equation}
We also introduce another decomposition
\[\boldsymbol{u}=\hat{\boldsymbol{Q}}_{b(t),{\lambda (t)}}+\hat{\boldsymbol{w}}=(\hat{\boldsymbol{Q}}_{b(s)}+\hat{\boldsymbol{\varepsilon}} (s))_{{\lambda(s)}} \]
which depends on whether $\ell=L$ (Proposition \ref{prop:2nd local approx}) or $\ell=L-1$ (Proposition \ref{prop:3rd local approx}). The evolution equation of $\hat{\boldsymbol{\varepsilon}}$ is given by
\begin{align}
  \partial_s\hat{\boldsymbol{\varepsilon}} -\frac{\lambda_s}{\lambda} \boldsymbol{\Lambda}\hat{\boldsymbol{\varepsilon}} + \boldsymbol{H}\hat{\boldsymbol{\varepsilon}} = -\widehat{\mathbf{Mod}}'(t)  - \hat{\boldsymbol{\psi}}_b -\widehat{\boldsymbol{NL}}(\hat{\boldsymbol{\varepsilon}}) -{\widehat{\boldsymbol{L}}}(\hat{\boldsymbol{\varepsilon}}),\label{eq:evolution hat epsilon}
\end{align}
where
\begin{equation}\label{eq:mod hat}
  \widehat{\mathbf{Mod}}'(t):=\widehat{\mathbf{Mod}}(t) -  \left(\frac{\lambda_s}{\lambda}+b_1\right)\boldsymbol{\Lambda} \hat{\boldsymbol{Q}}_b , 
\end{equation}
\begin{equation}
  \widehat{\boldsymbol{NL}}(\hat{\boldsymbol{\varepsilon}}):= \frac{1}{y^2}\begin{pmatrix}
    0 \\
    f(\hat{Q}_b+\hat{\varepsilon}) -f(\hat{Q}_b) - f'(\hat{Q}_b)\hat{\varepsilon}
  \end{pmatrix},\ \ \widehat{\boldsymbol{L}}(\hat{\boldsymbol{\varepsilon}}):= \frac{1}{y^2}\begin{pmatrix}
    0 \\
    (f'(\hat{Q}_b) -f'(Q))\hat{\varepsilon}
  \end{pmatrix}.
\end{equation}
We also employ the evolution equation of $\hat{\boldsymbol{w}}$:
\begin{equation}\label{eq:evolution hat W}
  \partial_t \hat{\boldsymbol{w}}+\boldsymbol{H}_{{\lambda}} \hat{\boldsymbol{w}}= \frac{1}{\lambda}\widehat{\boldsymbol{\mathcal{F}}}_{\lambda},\quad \widehat{\boldsymbol{\mathcal{F}}}= -\widehat{\mathbf{Mod}}'(t) -\hat{\boldsymbol{\psi}}_{b} -\widehat{\boldsymbol{ NL}}(\hat{\boldsymbol{\varepsilon}})-\widehat{\boldsymbol{L}}{(\hat{\boldsymbol{\varepsilon}})}.
\end{equation}
%%%%%%%%%%%%%%%%%%%%%%%%%%%%%%%%%%%%%

%%%%%%%%%%%%%%%%%%%%%%%%%%%%%%%%%%%%%%%%%%%%

\subsection{Initial data setting for the bootstrap}\label{subsec:initial and bootstrap}
In this subsection, we describe our initial data and the bootstrap assumption. To do this, we recall the fluctuation \eqref{eq:difference b_k} i.e. $U=(U_1,\cdots, U_{\ell})$,
\[ U_k(s)=s^k(\log s)^{\beta}(b_k(s)-b_k^e(s)). \]
We also define the adapted higher-order energies given by
\begin{equation} \label{eq:def E_k}
\mathcal{E}_{k}:=  \langle  \varepsilon_{k},\varepsilon_{k} \rangle + \langle  \dot{\varepsilon}_{k-1},  \dot{\varepsilon}_{k-1} \rangle , \quad 2\le k \le L+1.
\end{equation}
We set our renormalized spacetime variables $(s,y)$ as follows: for a large enough $s_0 \gg 1$, 
\begin{equation*}
  y=\frac{r}{\lambda(t)}, \quad s(t)=s_0+\int_0^t \frac{d \tau}{\lambda(\tau)} .
  \end{equation*}
  For the sake of simplicity, we use a transformed fluctuation $V=(V_1(s),\dots,V_{\ell}(s))$,
  \begin{equation} \label{eq:def V}
    V=P_{\ell} U 
  \end{equation}
   where $P_{\ell}$ yields the diagonalization \eqref{eq:diagonalization}. Then we illustrate the modulation parameters $b$ as a sum of the exact solutions $b^e(s)$ and $V(s)$: for $\ell=L-1$ or $L$,
\begin{align*}
  b(s)=b^e(s)+\left(\frac{(P_{\ell}^{-1}V(s))_1}{s(\log s)^{\beta}},\dots,\frac{(P_{\ell}^{-1}V(s))_{\ell}}{s^{\ell}(\log s)^{\beta}},b_{\ell+1}(s),\dots,b_L(s)\right).
\end{align*}

   Now, we assume some smallness conditions for our initial data $\boldsymbol{u}_0 (s_0) = (u_0,\dot{u}_0)$ as follows: for large constants $M=M(L)$, $K=K(L,M)$, $s_0=s_0(L,M,K)$, we set the initial data $\boldsymbol{u}_0 = \boldsymbol{u}(s_0)$ as
  \begin{equation}\label{eq:initial data form}
    \boldsymbol{u}_0 = (\tilde{\boldsymbol{Q}}_{b(s_0)} + \boldsymbol{\varepsilon}(s_0))_{\lambda(s_0)},
  \end{equation}
where $\boldsymbol{\varepsilon}(s_0)$ satisfies the orthogonality conditions \eqref{eq:orthogonal conditions}, the smallness of higher-order energies
\begin{align}
   \mathcal{E}_k (s_0) &\le b_1^{2L+4}(s_0)\label{eq:initial energies}
 \end{align}
 and $b(s_0)$ satisfies the smallness of the stable modes:
 \begin{equation}\label{eq:initial stable modes}
  |V_1(s_0)|\leq \frac{1}{4}  \ \textnormal{ and } \ |b_L(s_0)|\leq \frac{1}{s_0^{(L-1)c_1 }(\log s_0)^{3/2}} \ \textnormal{ for } \ \ell= L-1 
  \end{equation}
 where $c_1 = \frac{\ell}{\ell-1}$. Furthermore, we may assume
 \begin{equation}\label{eq:lambda renormalized}
  \lambda(s_0)=1
 \end{equation}
 up to rescaling.
\begin{proposition}[The existence of trapped solutions]\label{prop:bootstrap}
 Given $\boldsymbol{u}(s_0)$ of the form \eqref{eq:initial data form} satisfying \eqref{eq:orthogonal conditions}, \eqref{eq:initial energies} and \eqref{eq:initial stable modes}, there exists an initial direction of the unstable modes
 \begin{equation}\label{eq:initial unstable modes}
  (V_2(s_0),...,V_{\ell}(s_0)) \in \mathcal{B}^{{\ell}-1} 
  \end{equation}
  such that the corresponding solution to \eqref{eq:WM} becomes \textit{trapped}, namely, it satisfies the following bounds for all $s\geq s_0$,
\begin{itemize}
  \item\emph{Control of the higher-order energies:} for $2\le k \le \ell-1$,
\begin{align}
  &\mathcal{E}_k(s) \le b_1^{2(k-1)c_1} |\log b_1|^K, \quad\mathcal{E}_{L+1}(s)  \le K \frac{b_1^{2L+2}}{|\log b_1|^2},\label{eq:bootstrap epsilon}\\
  &\mathcal{E}_L(s) \le \begin{cases}
    K \lambda^{2(L-1)} & \textnormal{when }  {\ell}=L,\\
    b_1^{2L}|\log b_1|^K &  \textnormal{when }  {\ell}=L-1,
  \end{cases}  \label{eq:bootstrap epsilon L}\\
  &   \mathcal{E}_{L-1}(s) \le K \lambda^{2(L-2)}\quad \textnormal{when }  {\ell}=L-1. \label{eq:bootstrap epsilon L-1}
\end{align}
\item\emph{Control of the stable modes:} 
\begin{equation}\label{eq:bootstrap modes stables}
  |V_1(s)|\leq 1, \quad |b_L(s)|\leq \frac{1}{s^L (\log s)^{\beta}}, \quad \textnormal{when }  {\ell}=L-1.
  \end{equation}
  
\item\emph{Control of the unstable modes:}  
\begin{equation}\label{eq:bootstrap modes unstables}
  (V_2(s),\dots,V_{\ell}(s))\in \mathcal{B}^{\ell-1}.
  \end{equation}

\end{itemize}
\end{proposition}

Under the initial setting of $(\boldsymbol{\varepsilon}(s_0),V(s_0),b_{\ell+1}(s_0),\dots,b_L(s_0))$ (see \eqref{eq:initial data form}, \eqref{eq:initial energies}, \eqref{eq:initial stable modes} and \eqref{eq:initial unstable modes}), We define an exit time
\begin{align}\label{def:exit time}
  s^*=\sup\{ s\ge s_0 : \eqref{eq:bootstrap epsilon},\ \eqref{eq:bootstrap epsilon L}, \  \eqref{eq:bootstrap epsilon L-1},   \  \eqref{eq:bootstrap modes stables} \ \textnormal{and} \ \eqref{eq:bootstrap modes unstables} \ \textnormal{hold} \ \textnormal{on } [s_0,s] \}.
\end{align}
From \eqref{eq:initial data form}, \eqref{eq:initial energies}, \eqref{eq:initial stable modes} and \eqref{eq:initial unstable modes}, it is clear that \eqref{eq:bootstrap epsilon}, \eqref{eq:bootstrap epsilon L}, \eqref{eq:bootstrap epsilon L-1}, \eqref{eq:bootstrap modes stables} and \eqref{eq:bootstrap modes unstables} hold at $s=s_0$.
We will prove Proposition \ref{prop:bootstrap} in Section 4 by contradiction, assume that 
\begin{align}
    s^* < \infty \quad \textnormal{for all }(V_2(s_0),\dots,V_{\ell}(s_0))\in \mathcal{B}^{{\ell}-1} .\label{eq:finiteness assumption of exit time}
\end{align}
At the exit time $s^*$, we claim that only \eqref{eq:bootstrap modes unstables} fails among the bootstrap bounds in Proposition \ref{prop:bootstrap} through establishing estimates of modulation paramters and some monotonicity formulae of the higher-order energies. Then, the codimension $(\ell-1)$ stability \eqref{eq:diagonalization} leads a contradiction by Brouwer's fixed point theorem.

%%%%%%%%%%%%%%%%%%%%%%%%%%%%%%%%%%%%%%%%%%%%

%%%%%%%%%%%%%%%%%%%%%%%%%%%%%%%%%%%%%
\subsection{Modulation equations} 
Now we provide the evolution of the modulation parameters from the orthogonality conditions \eqref{eq:orthogonal conditions}.
\begin{lemma}[Modulation equations]\label{lem:mod eqn} The modulation parameters $(\lambda,b_1,\dots,b_L)$ satisfy the following bound 
  \begin{align}
    \left\lvert \frac{\lambda_s}{\lambda} + b_1 \right\rvert + \sum_{i=1}^{L-1} \lvert(b_i)_s +(i-1 + c_{b_1,i})b_1b_i - b_{i+1}\rvert &\lesssim C(M) b_1(\sqrt{\mathcal{E}_{L+1}}+b_1^{L+2}),\label{eq:mod bound1}\\
     \lvert (b_L)_s +(L-1 + c_{b_1,L})b_1b_L \rvert & \lesssim \frac{\sqrt{\mathcal{E}_{L+1}}}{\sqrt{\log M}}  + C(M) b_1^{L+3}.\label{eq:mod bound2}
  \end{align}
\end{lemma}
\begin{remark}
  \eqref{eq:mod bound1} and \eqref{eq:bootstrap epsilon} allow us to obtain the a priori assumption \eqref{eq:a priori for local}.
\end{remark}
  \begin{proof}
    \textbf{Step 1:} \emph{Modulation identity}.
    Denote $D(t)=(D_0(t),\dots,D_L(t))$ where $D_i(t)$ is given by
    \begin{equation*}
      D_0(t):=  -\left(\frac{\lambda_s}{\lambda} + b_1\right),\quad D_i(t):=  (b_i)_s +(i-1 + c_{b_1,i})b_1b_i - b_{i+1} ,\quad b_{L+1}=0.
    \end{equation*}
    We take the vector-valued inner product \eqref{def:vector inner product} of \eqref{eq:evolution epsilon} with $\boldsymbol{H}^{*k}\boldsymbol{\Phi}_M$ for $0\le k \le L$, we have the following identity
 \begin{align}\label{eq:modulation identity}
  \langle \widetilde{\mathbf{Mod}}(t), \boldsymbol{H}^{*k}\boldsymbol{\Phi}_M \rangle +\langle \boldsymbol{H}\boldsymbol{\varepsilon}, \boldsymbol{H}^{*k}\boldsymbol{\Phi}_M \rangle &=   \frac{\lambda_s}{\lambda} \langle \boldsymbol{\Lambda} \boldsymbol{\varepsilon},\boldsymbol{H}^{*k}\boldsymbol{\Phi}_M\rangle  - \langle \tilde{\boldsymbol{\psi}}_b,\boldsymbol{H}^{*k}\boldsymbol{\Phi}_M\rangle \nonumber \\
& - \langle \boldsymbol{NL}(\boldsymbol{\varepsilon}) + \boldsymbol{L}(\boldsymbol{\varepsilon}) ,\boldsymbol{H}^{*k}\boldsymbol{\Phi}_M\rangle  .
 \end{align}
 \textbf{Step 2:} \emph{Estimates for each terms in }\eqref{eq:modulation identity}. We claim that the LHS of \eqref{eq:modulation identity} gives the main contribution to prove \eqref{eq:mod bound1} and \eqref{eq:mod bound2}.
 
 (i) $\widetilde{\mathbf{Mod}}(t)$ \emph{terms}.
First, $\chi_{B_1} \boldsymbol{\alpha}_b = \boldsymbol{\alpha}_b$ holds on $|y|\le 2M$ for small enough $b_1$. We also have the pointwise bound
\[|\boldsymbol{\Lambda}\boldsymbol{\alpha}_b|  + \sum_{i=1}^L \sum_{j=i+1}^{L+2}  \left\lvert \frac{\partial \boldsymbol{S}_j}{\partial b_i}  \right\rvert \lesssim b_1 C(M) \quad \textnormal{for } |y| \le 2M\]
from our blow-up profile construction. 
Hence, we estimate the $\widetilde{\mathbf{Mod}}(t)$ term in \eqref{eq:modulation identity} by the transversality \eqref{eq:PhiM orthogonality} and the compact support property of $\boldsymbol{\Phi}_M$
 \begin{align}
  &\langle \widetilde{\mathbf{Mod}}(t), \boldsymbol{H}^{*k}
 \boldsymbol{\Phi}_M \rangle  = D_0(t) \langle \boldsymbol{\Lambda} {\boldsymbol{Q}}_b, \boldsymbol{H}^{*k} \boldsymbol{\Phi}_M \rangle +\sum_{i=1}^L D_i(t) \langle \boldsymbol{T}_i + \sum_{j=i+1}^{L+2}  \frac{\partial \boldsymbol{S}_j}{\partial b_i}, \boldsymbol{H}^{*k}\boldsymbol{\Phi}_M \rangle  \nonumber \\
 &= \sum_{i=0}^L D_i(t)  \langle \boldsymbol{T}_i , \boldsymbol{H}^{*k}\boldsymbol{\Phi}_M \rangle +  \left\langle D_0 (t) \boldsymbol{\Lambda} {\boldsymbol{\alpha}}_b +\sum_{i=1}^L \sum_{j=i+1}^{L+2} D_i(t) \frac{\partial \boldsymbol{S}_j}{\partial b_i} , \boldsymbol{H}^{*k}\boldsymbol{\Phi}_M \right\rangle \nonumber \\
  &=(-1)^k D_k(t) \langle \boldsymbol{\Lambda} {\boldsymbol{Q}}, \boldsymbol{\Phi}_M \rangle + O(C(M) b_1 |D(t)|). \label{eq:mod main}
\end{align}
(ii) \emph{Linear terms}. For $0\le k \le L-1$, we have
\[\langle \boldsymbol{H}\boldsymbol{\varepsilon}, \boldsymbol{H}^{*k}\boldsymbol{\Phi}_M \rangle = \langle \boldsymbol{\varepsilon}, \boldsymbol{H}^{*(k+1)}\boldsymbol{\Phi}_M \rangle  =0\]
from the orthogonal conditions \eqref{eq:orthogonal conditions}. For $k=L$, Cauchy-Schwarz inequality implies
\begin{equation}
  |\langle \boldsymbol{\varepsilon}, \boldsymbol{H}^{*(L+1)}\boldsymbol{\Phi}_M \rangle |=|\langle \boldsymbol{H}^{L+1}\boldsymbol{\varepsilon}, \boldsymbol{\Phi}_M \rangle | \lesssim \sqrt{\log M} \sqrt{\mathcal{E}_{L+1}}.
\end{equation}
(iii) \emph{Scaling terms}. We can estimate the scaling term in \eqref{eq:modulation identity} from the compact support property of $\boldsymbol{\Phi}_M$ and the coercivity bound \eqref{eq:coercivity bound}
\begin{align}
 \left\lvert \frac{\lambda_s}{\lambda} \langle \boldsymbol{\Lambda} \boldsymbol{\varepsilon}, \boldsymbol{H}^{*k}\boldsymbol{\Phi}_M\rangle  \right\rvert &\le \left(b_1 + |D_0(t)| \right)   | \langle \boldsymbol{\Lambda} \boldsymbol{\varepsilon}, \boldsymbol{H}^{*k}\boldsymbol{\Phi}_M\rangle  |  \nonumber\\
 &\lesssim  \left(b_1 + |D_0(t)| \right) C(M) \sqrt{\mathcal{E}_{L+1}}.
\end{align}
(iv) $\tilde{\boldsymbol{\psi}}_b$ \emph{terms}. Here, the improved local bound \eqref{eq:psit local} implies
    \begin{equation}
      |\langle \tilde{\boldsymbol{\psi}}_b,\boldsymbol{H}^{*k}\boldsymbol{\Phi}_M\rangle |  \lesssim C(M) b_1^{L+3}.
    \end{equation}
    (v) $\boldsymbol{NL}(\boldsymbol{\varepsilon})$ and $\boldsymbol{L}(\boldsymbol{\varepsilon})$ \emph{terms}.
 Using the coercivity bound \eqref{eq:coercivity bound} with the crude bound $|NL(\varepsilon)|\lesssim |\varepsilon|^2/y^2$ and $|L(\varepsilon)|\lesssim b_1^2|\varepsilon|/y$,
\begin{equation}\label{eq:mod NL}
  |\langle \boldsymbol{NL}(\boldsymbol{\varepsilon}),\boldsymbol{H}^{*i}\boldsymbol{\Phi}_M\rangle | \lesssim C(M) \mathcal{E}_{L+1},\quad |\langle \boldsymbol{L}(\boldsymbol{\varepsilon}),\boldsymbol{H}^{*i}\boldsymbol{\Phi}_M\rangle| \lesssim C(M) b_1^2 \sqrt{\mathcal{E}_{L+1}}.
\end{equation}
\textbf{Step 3:} \emph{Conclusion}. 
Injecting the estimates from \eqref{eq:mod main} to \eqref{eq:mod NL} into \eqref{eq:modulation identity}, we obtain 
\begin{align}
  (-1)^k D_k(t) \langle \boldsymbol{\Lambda} {\boldsymbol{Q}}, \boldsymbol{\Phi}_M \rangle + O(C(M) b_1 |D(t)|) &= O(\sqrt{\log M} \sqrt{\mathcal{E}_{L+1}})\delta_{kL} \nonumber \\
  &+ O(C(M) b_1 ( \sqrt{\mathcal{E}_{L+1}}+b_1^{L+2} )) \label{eq:final mod identity}
\end{align}
for $0\le k \le L$. We then divide them above equation by $\langle \boldsymbol{\Lambda} {\boldsymbol{Q}}, \boldsymbol{\Phi}_M \rangle$, \eqref{eq:PhiM properties} implies
\begin{align*}
   D_k(t)  + O(C(M) b_1 |D(t)|) &= O\left(\frac{\sqrt{\mathcal{E}_{L+1}}}{\sqrt{\log M}} \right)\delta_{kL} + O(C(M) b_1 ( \sqrt{\mathcal{E}_{L+1}}+b_1^{L+2} )),
\end{align*}
which yields \eqref{eq:mod bound1} and \eqref{eq:mod bound2}.
  \end{proof}
%%%%%%%%%%%%%%%%%%%%%%%%%%%%%%%%%%%%%
%%%%%%%%%%%%%%%%%%%%%%%%%%%%%%%%%%%%%

\subsection{Improved modulation equation of $b_L$}
At first glance, \eqref{eq:mod bound2} seems sufficient to close the modulation equation for $b_L$ because of the presence of $\sqrt{\log M}$. However, our desired blow-up scenario comes from the exact solution $b_L^e$, \eqref{eq:mod bound2} is inadequate to close the bootstrap bounds for stable/unstable modes $V(s)$. Thus, we need to obtain a further logarithm room by adding some correction to $b_L$.
  \begin{lemma}[Improved modulation equation of $b_L$]\label{lem:improved b_L}
    Let $B_{\delta}=B_0^{\delta}$ and 
    \begin{equation}\label{def:tilde b_L}
      \tilde{b}_L=b_L + (-1)^{L}\frac{\langle \boldsymbol{H}^{L}\boldsymbol{\varepsilon}, \chi_{B_{\delta}} \boldsymbol{\Lambda} \boldsymbol{Q} \rangle}{4\delta |\log b_1|}.
    \end{equation}
    for some small enough universal constant $0<\delta \ll 1$. Then $\tilde{b}_L$ satisfies
    \begin{equation}\label{eq:difference tilde b_L}
      |\tilde{b}_L-b_L| \lesssim b_1^{L+1-C\delta}
    \end{equation}
    and 
    \begin{equation}\label{eq:mod eqn tilde b_L}
      |(\tilde{b}_L)_s + (L-1 + c_{b,L})b_1\tilde{b}_L| \lesssim \frac{\sqrt{\mathcal{E}_{L+1}}}{\sqrt{|\log b_1|}}.
    \end{equation}
  \end{lemma}
  \begin{remark}\label{rem:improved b_L}
    We point out that $\tilde{b}_L$ is well-defined at time $s=s_0$, since $\tilde{b}_L-b_L$ only depends on $b_1$ and $\boldsymbol{\varepsilon}$. 
  \end{remark}
  \begin{proof}
   We obtain \eqref{eq:difference tilde b_L} from the coercivity bound \eqref{eq:coercivity bound} and \eqref{eq:mod bound1} 
    \begin{align}\label{eq:delta cutoff bound1}
      |\langle \boldsymbol{H}^{L}\boldsymbol{\varepsilon}, \chi_{B_{\delta}} \boldsymbol{\Lambda} \boldsymbol{Q} \rangle|  &\lesssim \left| \langle H^{\frac{L-1}{2}} \dot{\varepsilon},  \chi_{B_{\delta}} \Lambda Q \rangle \right| \lesssim C(M)\delta b_1^{-C\delta} \sqrt{\mathcal{E}_{L+1}} \lesssim b_1^{L+1-C\delta},
    \end{align}
We also know
    \begin{equation}\label{eq:ds b_L correction}
      \frac{d}{ds} \langle \boldsymbol{H}^{L}\boldsymbol{\varepsilon}, \chi_{B_{\delta}} \boldsymbol{\Lambda} \boldsymbol{Q} \rangle =  \langle \boldsymbol{H}^{L}\boldsymbol{\varepsilon}_s, \chi_{B_{\delta}} \boldsymbol{\Lambda} \boldsymbol{Q} \rangle + \langle \boldsymbol{H}^{L}\boldsymbol{\varepsilon}, (\chi_{B_{\delta}})_s \boldsymbol{\Lambda} \boldsymbol{Q} \rangle .
    \end{equation}
    We compute the last inner product in \eqref{eq:ds b_L correction} similarly to \eqref{eq:delta cutoff bound1}:
    \begin{align}
      \left|\langle \boldsymbol{H}^{L}\boldsymbol{\varepsilon}, (\chi_{B_{\delta}})_s \boldsymbol{\Lambda} \boldsymbol{Q} \rangle  \right| &=   |\delta (b_1)_s b_1^{-1}| \left| \langle H^{\frac{L-1}{2}} \dot{\varepsilon}, (y\partial_y \chi)_{B_{\delta}} \Lambda Q \rangle \right| \lesssim C(M)\delta b_1^{1-\delta} \sqrt{\mathcal{E}_{L+1}}.\label{eq:delta cutoff bound2}
    \end{align}
    Using \eqref{eq:evolution epsilon}, we obtain the following identity similar to \eqref{eq:modulation identity}
    \begin{align*}
      \langle \boldsymbol{H}^{L}\boldsymbol{\varepsilon}_s, \chi_{B_{\delta}} \boldsymbol{\Lambda} \boldsymbol{Q} \rangle &=-\langle \boldsymbol{H}^{L}\widetilde{\mathbf{Mod}}(t), \chi_{B_{\delta}} \boldsymbol{\Lambda} \boldsymbol{Q} \rangle  - \langle \boldsymbol{H}^{L+1}\boldsymbol{\varepsilon}, \chi_{B_{\delta}} \boldsymbol{\Lambda} \boldsymbol{Q} \rangle \\
      &+\frac{\lambda_s}{\lambda}\langle \boldsymbol{H}^{L}\boldsymbol{\Lambda}\boldsymbol{\varepsilon}, \chi_{B_{\delta}} \boldsymbol{\Lambda} \boldsymbol{Q} \rangle 
      -\langle \boldsymbol{H}^{L}\tilde{\boldsymbol{\psi}}_b, \chi_{B_{\delta}} \boldsymbol{\Lambda} \boldsymbol{Q} \rangle \\
      & - \langle \boldsymbol{H}^{L}\boldsymbol{NL}(\boldsymbol{\varepsilon}), \chi_{B_{\delta}} \boldsymbol{\Lambda} \boldsymbol{Q} \rangle  - \langle \boldsymbol{H}^{L}\boldsymbol{L}(\boldsymbol{\varepsilon}), \chi_{B_{\delta}} \boldsymbol{\Lambda} \boldsymbol{Q} \rangle
    \end{align*}
    Considering the support of $\chi_{B_{\delta}} \boldsymbol{\Lambda Q} $, we can borrow all the estimates in \textbf{Step 2} of the proof of Lemma \ref{lem:mod eqn} by replacing the weight $\log M$ and $C(M)$ to $|\log b_1|$ and $b_1^{-C\delta}$, respectively. Hence, Lemma \ref{lem:mod eqn} and \eqref{eq:delta cutoff bound2} give a "$B_{\delta}$ version" of \eqref{eq:final mod identity}
    \begin{align*}
      \frac{d}{ds} \langle \boldsymbol{H}^{L}\boldsymbol{\varepsilon}, \chi_{B_{\delta}} \boldsymbol{\Lambda} \boldsymbol{Q} \rangle &=  
       (-1)^{L+1}D_L(t) \langle \boldsymbol{\Lambda} {\boldsymbol{Q}}, \chi_{B_{\delta}} \boldsymbol{\Lambda Q} \rangle + O( b_1^{1-C\delta} |D(t)|) \\
        &+ O(\sqrt{|\log b_1|} \sqrt{\mathcal{E}_{L+1}}) + O( b_1^{1-C\delta} ( \sqrt{\mathcal{E}_{L+1}}+b_1^{L+2} )) \\
        &=(-1)^{L+1}4\delta |\log b_1| D_L(t) +  O(\sqrt{|\log b_1|} \sqrt{\mathcal{E}_{L+1}}).
    \end{align*}
  Hence, we obtain \eqref{eq:mod eqn tilde b_L} as follows:
  \begin{align*}
    |(\tilde{b}_L)_s + (L-1 + c_{b,L})b_1\tilde{b}_L| &\lesssim |\langle \boldsymbol{H}^{L}\boldsymbol{\varepsilon}, \chi_{B_{\delta}} \boldsymbol{\Lambda} \boldsymbol{Q} \rangle| \left\lvert b_1+\frac{d}{ds} \left\{  \frac{1}{4\delta \log b_1} \right\} \right\rvert + \frac{\sqrt{\mathcal{E}_{L+1}}}{\sqrt{|\log b_1|}}\\
    &\lesssim \frac{\sqrt{\mathcal{E}_{L+1}}}{\sqrt{|\log b_1|}} +b_1^{L+2-C\delta}. \qedhere
  \end{align*}
  \end{proof}

%%%%%%%%%%%%%%%%%%%%%%%%%%%%%%%%%%%%%%

\subsection{Lyapunov monotonicity for $\mathcal{E}_{L+1}$} 
A simple way to control the adapted higher-order energy $\mathcal{E}_{L+1}$ is to estimate its time derivative. However, we cannot obtain enough estimates to close the bootstrap bound \eqref{eq:bootstrap epsilon} with $\mathcal{E}_{L+1}$ by itself, i.e. with $b_1 \sim -\lambda_t $,
\begin{align*}
  \frac{d}{dt}  \left\{  \frac{  \mathcal{E}_{L+1}    }{\lambda^{2L}}\right\} \le C b_1  \frac{  \mathcal{E}_{L+1}    }{\lambda^{2L+1}}, \quad  \frac{  \mathcal{E}_{L+1}(t)    }{\lambda^{2L}(t)} &\le  \frac{  \mathcal{E}_{L+1}(0)    }{\lambda^{2L}(0)} + C\int_0^t b_1(\tau)\frac{  \mathcal{E}_{L+1}(\tau)    }{\lambda^{2L+1}(\tau)}   d\tau \\
  &\le  K\int_0^t \frac{   b_1(\tau)  }{\lambda^{2L+1}(\tau)}  \frac{b_1^{2(L+1)}(\tau)}{|\log b_1(\tau)|^2}  d\tau \\
  &\lesssim  \frac{   K  }{\lambda^{2L}(t)}  \frac{b_1^{2(L+1)}(t)}{|\log b_1(t)|^2}.
\end{align*}
Thus, we use the repulsive property of the conjugated Hamiltonian $\widetilde{H}$ of $H$ observed in \cite{RodnianskiSterbenz2010Ann.Math.} and \cite{RaphaelRodnianski2012IHES} with some additional integration by parts to pull out the accurate corrections.

\begin{proposition}[Lyapunov monotonicity for $\mathcal{E}_{L+1}$]\label{prop:monotonicity}
  We have the following bound:
  \begin{align}
    \frac{d}{dt}  \left\{  \frac{  \mathcal{E}_{L+1}    }{\lambda^{2L}} + O\left(\frac{   b_1 C(M) \mathcal{E}_{L+1}    }{\lambda^{2L}}\right)\right\} \le C \frac{b_1}{\lambda^{2L+1}} \left[ \frac{b_1^{L+1}}{|\log b_1|} \sqrt{\mathcal{E}_{L+1}}  +\frac{\mathcal{E}_{L+1}}{\sqrt{\log M}} \right]\label{eq:monotonicity}
  \end{align}

\end{proposition}
\begin{proof}
  \textbf{Step 1:} \emph{Evolution of adapted derivatives.} We start by introducing the rescaled version of the operators $A$ and $A^*$
  \[ A_{\lambda} :=-\partial_r + \frac{Z_{\lambda}}{r} ,\quad A^*_{\lambda} :=\partial_r + \frac{1+Z_{\lambda}}{r},\quad Z_{\lambda}(r)=Z\left(\frac{r}{\lambda}\right)=\frac{1-(r/\lambda)^2}{1+(r/\lambda)^2} .\]
  We also recall $H_{\lambda}$ in \eqref{def:capital H rescaled} and define its conjugate operator $\widetilde{H}_{\lambda}$ as the rescaled version of the linearized operator $H$ and its conjugate $\widetilde{H}$:
  \begin{align*}
    H_{\lambda} &:= A^*_{\lambda}A_{\lambda}=-\Delta + \frac{V_{\lambda}}{r^2},\quad  V(y)=\frac{y^4-6y^2+1}{(y^2+1)^2}, \\
     \widetilde{H}_{\lambda} &:= A_{\lambda}A^*_{\lambda}=-\Delta + \frac{\widetilde{V}_{\lambda}}{r^2},\quad \widetilde{V}(y) =\frac{4}{y^2+1}.
  \end{align*}
  In the same manner as \eqref{def:A^k}, we denote the rescaled version of the adapted derivative operator
  \begin{equation}\label{def:rescaled A^k}
    \mathcal{A}_{\lambda}:=A_{\lambda},\quad \mathcal{A}_{\lambda}^2:=A_{\lambda}^* A_{\lambda}  ,\quad \mathcal{A}_{\lambda}^3:= A_{\lambda}A_{\lambda}^*A_{\lambda},\;\cdots, \; \mathcal{A}_{\lambda}^k:=\underbrace{\cdots A_{\lambda}^*A_{\lambda} A_{\lambda}^* A_{\lambda}}_{k \textnormal{ times}},
  \end{equation}
  so the higher-order derivatives of $\boldsymbol{w}=(w,\dot{w})^t$ adapted to the Hamiltonian $H_{\lambda}$ are given by 
  \[w_k:=\mathcal{A}_{\lambda}^k w,\quad \dot{w}_k:=\mathcal{A}_{\lambda}^k \dot{w}.\]
  %\begin{align*}
   % \frac{\mathcal{E}_{k}}{\lambda^{2(k-1)}} &= \left\langle  \frac{\varepsilon_{k}}{\lambda^{k-1}},\frac{\varepsilon_{k}}{\lambda^{k-1}} \right\rangle + \left\langle \frac{\dot{\varepsilon}_{k-1}}{\lambda^{k-1}}, \frac{\dot{\varepsilon}_{k-1}}{\lambda^{k-1}} \right\rangle \\
   % &=\left\langle  \frac{(\varepsilon_{k})_{\lambda}}{\lambda^{k}},\frac{(\varepsilon_{k})_{\lambda}}{\lambda^{k}} \right\rangle + \left\langle \frac{(\dot{\varepsilon}_{k-1})_{\lambda}}{\lambda^{k}}, \frac{(\dot{\varepsilon}_{k-1})_{\lambda}}{\lambda^{k}} \right\rangle.
  %\end{align*}
One can easily check that $w_k=\frac{(\varepsilon_k)_{\lambda}}{\lambda^k}$ and  $\dot{w}_k=\frac{(\dot{\varepsilon_k})_{\lambda}}{\lambda^{k+1}}$, our target energy can be written as
\begin{equation}\label{eq:energy w form}
 \frac{\mathcal{E}_{L+1}}{\lambda^{2L}}=\langle  w_{L+1},w_{L+1} \rangle + \langle \dot{w}_L, \dot{w}_L \rangle = \langle \widetilde{H}_{\lambda} w_L,w_L \rangle + \langle \dot{w}_L, \dot{w}_L \rangle.
\end{equation}
To describe the evolution of $w_k$ and $\dot{w}_{k}$, we first rewrite the flow \eqref{eq:evolution W} of $\boldsymbol{w}=(w,\dot{w})$ component-wisely:
  \begin{equation}\label{eq:evolution w}
    \begin{cases}
      w_t -\dot{w} = \mathcal{F}_1 \\
      \dot{w}_t + H_{\lambda} w = \mathcal{F}_2 
    \end{cases},\quad \begin{pmatrix}
      \mathcal{F}_1  \\
      \mathcal{F}_2
    \end{pmatrix} := \frac{1}{\lambda} \boldsymbol{\mathcal{F}}_{\lambda}=\frac{1}{\lambda}\begin{pmatrix}
      \mathcal{F} \\
      \dot{\mathcal{F}}
    \end{pmatrix}_{\lambda} .
  \end{equation}
  Taking $\mathcal{A}_{\lambda}^k$ given by \eqref{def:rescaled A^k} into \eqref{eq:evolution w}, we obtain the evolution equation of $w_k$:
\begin{equation}\label{eq:evolution w_k}
  \begin{cases}
    \partial_t w_{k} -\dot{w}_{k} = [\partial_t,\mathcal{A}_{\lambda}^{k}]w +\mathcal{A}_{\lambda}^{k}\mathcal{F}_1 \\
    \partial_t\dot{w}_{k} + w_{k+2} = [\partial_t,\mathcal{A}_{\lambda}^{k}]\dot{w} +\mathcal{A}_{\lambda}^{k}\mathcal{F}_2 
  \end{cases} .
\end{equation}
Lastly, we employ the following notation: for any time-dependent operator $P$,
\begin{equation*}
  \partial_t(P):=[\partial_t, P], 
\end{equation*}
which yields the Leibniz rule between the operator and function:
\[\partial_t(Pf)=\partial_t(P)f + P f_t.\]
\textbf{Step 2:} \emph{First energy identity}. 
Recall \eqref{eq:energy w form}, we compute the energy identity:
\begin{align}
  \partial_t \left(\frac{\mathcal{E}_{L+1}}{2\lambda^{2L}} \right)&= \frac{1}{2} \langle \partial_t(\widetilde{H}_{\lambda}) w_L,w_L \rangle + \langle \widetilde{H}_{\lambda} w_L,\partial_t w_L \rangle +    \langle  \dot{w}_L,\partial_t \dot{w}_L \rangle \nonumber \\
  &= \frac{1}{2} \langle \partial_t(\widetilde{H}_{\lambda}) w_L,w_L \rangle \label{eq:initial repulsive} \\
  & +\langle  \widetilde{H}_{\lambda} w_L, \partial_t(\mathcal{A}_{\lambda}^L)w \rangle +   \langle \dot{w}_L,\partial_t(\mathcal{A}_{\lambda}^L)\dot{w} \rangle \label{eq:initial main}\\
  &+ \langle  \widetilde{H}_{\lambda} w_L, \mathcal{A}_{\lambda}^L \mathcal{F}_1 \rangle +   \langle \dot{w}_L,\mathcal{A}_{\lambda}^L\mathcal{F}_2 \rangle. \label{eq:initial error}
\end{align}
We will check that \eqref{eq:initial error} satisfies the desired bound \eqref{eq:monotonicity} later. Unlike \eqref{eq:initial error}, when \eqref{eq:initial repulsive} and \eqref{eq:initial main} are estimated using coercivity \eqref{eq:coercivity bound} directly, we obtain the following insufficient bound
\[\frac{b_1}{\lambda^{2L+1}} C(M) \mathcal{E}_{L+1}.\]
One can employ repulsive property \eqref{eq:repulsive property} for \eqref{eq:initial repulsive} with the modulation equation \eqref{eq:mod bound1}:
\begin{equation}\label{eq:explicit repulsive}
  \partial_t(\widetilde{H}_{\lambda})=-\frac{\lambda_t}{\lambda} \frac{(\Lambda \widetilde{V})_{\lambda}}{r^2} = -\frac{b_1+O(b_1^{L+2})}{\lambda^3} \frac{8}{(1+y^2)^2} \Rightarrow \langle \partial_t(\widetilde{H}_{\lambda}) w_L,w_L \rangle <0.
\end{equation}
We claim that \eqref{eq:initial main} is eventually negative like \eqref{eq:explicit repulsive} by adding some corrections. For this,  we start by employing \eqref{eq:evolution w_k} to exchange $\widetilde{H}_{\lambda}w_L$ for $-\partial_{t}\dot{w}_L$,
\begin{align}
  \langle  \widetilde{H}_{\lambda} w_L, \partial_t(\mathcal{A}_{\lambda}^L)w \rangle  =& -\langle  \partial_t \dot{w}_L, \partial_t(\mathcal{A}_{\lambda}^L)w \rangle \label{eq:H to dt}\\
  &+ \langle  \partial_t(\mathcal{A}_{\lambda}^L) \dot{w}, \partial_t(\mathcal{A}_{\lambda}^L)w \rangle + \langle  \mathcal{A}_{\lambda}^L \mathcal{F}_2, \partial_t(\mathcal{A}_{\lambda}^L)w \rangle \label{eq:H to dt error},
\end{align}
we can treat \eqref{eq:H to dt} via integration by parts in time with \eqref{eq:evolution w},
\begin{align}
  -\langle  \partial_t \dot{w}_L, \partial_t(\mathcal{A}_{\lambda}^L)w \rangle +\partial_t \langle  \dot{w}_L, \partial_t(\mathcal{A}_{\lambda}^L)w \rangle &=   \langle   \dot{w}_L, \partial_{tt}(\mathcal{A}_{\lambda}^L)w \rangle + \langle   \dot{w}_L, \partial_t(\mathcal{A}_{\lambda}^L)w_t \rangle \label{eq:first int by parts correction}\\
  &=  \langle   \dot{w}_L, \partial_t(\mathcal{A}_{\lambda}^L)\dot{w} \rangle \nonumber \\
  & +   \langle   \dot{w}_L, \partial_{tt}(\mathcal{A}_{\lambda}^L)w \rangle + \langle   \dot{w}_L, \partial_t(\mathcal{A}_{\lambda}^L) \mathcal{F}_1 \rangle. \label{eq:first int by parts error}
\end{align}
In short, we add a correction to the energy identity to transform the first inner product in \eqref{eq:initial main} to the second one in \eqref{eq:initial main} up to some errors \eqref{eq:H to dt error}, \eqref{eq:first int by parts error}:
\begin{align}
  \langle  \widetilde{H}_{\lambda} w_L, \partial_t(\mathcal{A}_{\lambda}^L)w \rangle  + \partial_t D_{0,1,1} &=  \langle   \dot{w}_L, \partial_t(\mathcal{A}_{\lambda}^L)\dot{w} \rangle \label{eq:first int by parts}\\
  &+ E_{0,1,1} + E_{0,1,2} + F_{0,1,1} + F_{0,1,2} \nonumber
\end{align}
where
\begin{align*}
  D_{0,1,1} =  \langle  \dot{w}_L, \partial_t(\mathcal{A}_{\lambda}^L)w \rangle ,\quad E_{0,1,1}&=  \langle   \dot{w}_L, \partial_{tt}(\mathcal{A}_{\lambda}^L)w \rangle ,\quad E_{0,1,2}= \langle  \partial_t(\mathcal{A}_{\lambda}^L) \dot{w}, \partial_t(\mathcal{A}_{\lambda}^L)w \rangle, \\
  F_{0,1,1}&= \langle   \dot{w}_L, \partial_t(\mathcal{A}_{\lambda}^L) \mathcal{F}_1 \rangle ,\quad F_{0,1,2}=\langle  \mathcal{A}_{\lambda}^L \mathcal{F}_2, \partial_t(\mathcal{A}_{\lambda}^L)w \rangle .
\end{align*}
However, the second inner product in \eqref{eq:initial main} is also not small enough to close our bootstrap by itself. Thus, we use \eqref{eq:evolution w_k} again to exchange $\dot{w}_L$ for $\partial_t w_L$,
\begin{align*}
  \langle   \dot{w}_L, \partial_t(\mathcal{A}_{\lambda}^L)\dot{w} \rangle  &= \langle   \partial_t{w}_L, \partial_t(\mathcal{A}_{\lambda}^L)\dot{w} \rangle \\
  &- \langle   \partial_t(\mathcal{A}_{\lambda}^L)w, \partial_t(\mathcal{A}_{\lambda}^L)\dot{w} \rangle - \langle  \mathcal{A}_{\lambda}^L \mathcal{F}_1, \partial_t(\mathcal{A}_{\lambda}^L)\dot{w} \rangle.
\end{align*}
Integrating by parts in time once more,  
\begin{align*}
  \langle   \partial_t{w}_L, \partial_t(\mathcal{A}_{\lambda}^L)\dot{w} \rangle - \partial_t\langle  {w}_L, \partial_t(\mathcal{A}_{\lambda}^L)\dot{w} \rangle &=  -  \langle  {w}_L, \partial_{tt}(\mathcal{A}_{\lambda}^L)\dot{w} \rangle-\langle  {w}_L, \partial_t(\mathcal{A}_{\lambda}^L)\dot{w}_t \rangle  \\
  &=  \langle  {w}_L, \partial_t(\mathcal{A}_{\lambda}^L){w}_2 \rangle \\
  &- \langle  {w}_L, \partial_{tt}(\mathcal{A}_{\lambda}^L)\dot{w} \rangle -\langle  {w}_L, \partial_t(\mathcal{A}_{\lambda}^L)\mathcal{F}_2\rangle. 
\end{align*}
To sum it up, we obtain a relation similar to \eqref{eq:first int by parts}:
\begin{align}
  \langle   \dot{w}_L, \partial_t(\mathcal{A}_{\lambda}^L)\dot{w} \rangle + \partial_t D_{0,2,1} &= \langle  {w}_L, \partial_t(\mathcal{A}_{\lambda}^L){w}_2 \rangle \label{eq:second int by parts}\\
  & + E_{0,2,1} + E_{0,2,2} + F_{0,2,1} + F_{0,2,2} \nonumber
\end{align}
where
\begin{align*}
  D_{0,2,1} &=  -\langle  {w}_L, \partial_t(\mathcal{A}_{\lambda}^L)\dot{w} \rangle ,\\
   E_{0,2,1}&= - \langle  {w}_L, \partial_{tt}(\mathcal{A}_{\lambda}^L)\dot{w} \rangle ,\quad E_{0,2,2}= - \langle   \partial_t(\mathcal{A}_{\lambda}^L)w, \partial_t(\mathcal{A}_{\lambda}^L)\dot{w} \rangle, \\
  F_{0,2,1}&= - \langle  \mathcal{A}_{\lambda}^L \mathcal{F}_1, \partial_t(\mathcal{A}_{\lambda}^L)\dot{w} \rangle ,\quad F_{0,2,2}=-\langle  {w}_L, \partial_t(\mathcal{A}_{\lambda}^L)\mathcal{F}_2\rangle.
\end{align*}
In \cite{RaphaelRodnianski2012IHES} (the case $L=1$), the authors directly checked that $\langle w_1,  \partial_{t}(\mathcal{A}_{\lambda}^L)w_{2} \rangle < 0$. In contrast, when $L\ge 3$, we cannot obtain similar information from $ \langle w_L, \partial_{t}(\mathcal{A}_{\lambda}^L)w_{2} \rangle $ by itself. We pull out the repulsive terms using the Leibniz rule,
\begin{align}\label{eq:pullout repulsive}
  \langle   w_L , \partial_{t}(\mathcal{A}_{\lambda}^L)w_{2}\rangle &= \langle w_L, \partial_{t}(\widetilde{H}_{\lambda})w_{L} \rangle + \langle  w_L, \widetilde{H}_{\lambda}\partial_{t}(\mathcal{A}_{\lambda}^{L-2})w_{2} \rangle \nonumber\\
  &=\langle w_L, \partial_{t}(\widetilde{H}_{\lambda})w_{L} \rangle + \langle \widetilde{H}_{\lambda} w_L, \partial_{t}(\mathcal{A}_{\lambda}^{L-2})w_{2} \rangle .
\end{align}
We observe that the second inner product in \eqref{eq:pullout repulsive} has the same form as the first inner product in \eqref{eq:first int by parts}, we can iterate integration by parts, which leads to the following recurrence equations. For $0\le k \le \frac{L-1}{2}$,
\begin{align}
  \langle  \widetilde{H}_{\lambda} w_L, \partial_t(\mathcal{A}_{\lambda}^{L-2k})w_{2k} \rangle  + \partial_t D_{k,1,1} &=  \langle   \dot{w}_L, \partial_t(\mathcal{A}_{\lambda}^{L-2k})\dot{w}_{2k} \rangle  \label{eq:first int by parts k} \\
  & + E_{k,1,1} + E_{k,1,2} + F_{k,1,1} + F_{k,1,2}\nonumber
\end{align}
where
\begin{align*}
  D_{k,1,1} &= \langle \dot{w}_L, \partial_t(\mathcal{A}_{\lambda}^{L-2k}) w_{2k} \rangle, \quad E_{k,1,1}  = \langle   \dot{w}_L, \partial_{tt}(\mathcal{A}_{\lambda}^{L-2k})w_{2k} \rangle,\\
  E_{k,1,2} &= \langle  \partial_t(\mathcal{A}_{\lambda}^L) \dot{w}, \partial_t(\mathcal{A}_{\lambda}^{L-2k})w_{2k} \rangle +  \langle   \dot{w}_L, \partial_t(\mathcal{A}_{\lambda}^{L-2k}) \partial_t(H_{\lambda}^{k})w \rangle,\\
  F_{k,1,1} & = \langle   \dot{w}_L, \partial_t(\mathcal{A}_{\lambda}^{L-2k}) H_{\lambda}^{k}\mathcal{F}_1  \rangle, \quad F_{k,1,2}  = \langle  \mathcal{A}_{\lambda}^L \mathcal{F}_2, \partial_t(\mathcal{A}_{\lambda}^{L-2k})w_{2k} \rangle 
\end{align*}
and
\begin{align}
  \langle   \dot{w}_L, \partial_t(\mathcal{A}_{\lambda}^{L-2k})\dot{w}_{2k} \rangle + \partial_t D_{k,2,1} &= \langle  {w}_L, \partial_t(\mathcal{A}_{\lambda}^{L-2k}){w}_{2k+2} \rangle \label{eq:second int by parts k} \\
  & + E_{k,2,1} + E_{k,2,2} + F_{k,2,1} + F_{k,2,2} \nonumber
\end{align}
where
\begin{align*}
  D_{k,2,1} &=-\langle w_L, \partial_t(\mathcal{A}_{\lambda}^{L-2k}) \dot{w}_{2k} \rangle, \quad E_{k,2,1}  =- \langle  {w}_L, \partial_{tt}(\mathcal{A}_{\lambda}^{L-2k})\dot{w}_{2k} \rangle ,\\
  E_{k,2,2} &= - \langle   \partial_t(\mathcal{A}_{\lambda}^L)w, \partial_t(\mathcal{A}_{\lambda}^{L-2k})\dot{w}_{2k} \rangle - \langle  w_L, \partial_t(\mathcal{A}_{\lambda}^{L-2k})\partial_t(H_{\lambda}^k)\dot{w}\rangle,\\
  F_{k,2,1} & =  - \langle  \mathcal{A}_{\lambda}^L \mathcal{F}_1, \partial_t(\mathcal{A}_{\lambda}^{L-2k})\dot{w}_{2k} \rangle , \quad F_{k,2,2}  = -\langle  {w}_L, \partial_t(\mathcal{A}_{\lambda}^{L-2k})\mathcal{F}_2\rangle.
\end{align*}
We can also pull out the repulsive term like \eqref{eq:pullout repulsive} from \eqref{eq:second int by parts k}: for $0\le k \le \frac{L-3}{2}$, 
\begin{align}\label{eq:pullout repulsive k}
  \langle   w_L , \partial_{t}(\mathcal{A}_{\lambda}^{L-2k})w_{2k+2}\rangle &= \langle w_L, \partial_{t}(\widetilde{H}_{\lambda})w_{L} \rangle + \langle  \widetilde{H}_{\lambda} w_L,\partial_{t}(\mathcal{A}_{\lambda}^{L-2k-2})w_{2k+2} \rangle.
\end{align}
The displays \eqref{eq:first int by parts k}, \eqref{eq:second int by parts k} and \eqref{eq:pullout repulsive k} allow us to iterate our recurrence relations. For $k=\frac{L-1}{2}$, we can verify that \eqref{eq:second int by parts k} is negative from the fact $\partial_t(A_{\lambda})=\partial_t(A_{\lambda}^*)=\frac{-\lambda_t}{\lambda} \frac{(\Lambda {Z})_{\lambda}}{r}$,
\begin{align*}
  \langle \partial_t(\widetilde{H}_{\lambda}) w_L,w_L \rangle &= \langle \partial_t(A_{\lambda}A_{\lambda}^*) w_L,w_L \rangle \\
  &= \langle \partial_t(A_{\lambda})A_{\lambda}^* w_L,w_L \rangle + \langle A_{\lambda}\partial_t(A_{\lambda}^* )w_L,w_L \rangle = 2\langle \partial_{t}({A}_{\lambda}) w_{L+1},  w_L \rangle .
\end{align*}
Hence, we decompose the first term of \eqref{eq:initial main} as follows: 
\begin{align*}
  \langle  \widetilde{H}_{\lambda} w_L, \partial_t(\mathcal{A}_{\lambda}^L)w \rangle + \sum_{k=0}^{\frac{L-1}{2}} \sum_{i=1}^2 \partial_t D_{k,i,1} &= \frac{L}{2}  \langle \partial_t(\widetilde{H}_{\lambda}) w_L,w_L \rangle + \sum_{k=0}^{\frac{L-1}{2}} \sum_{i,j=1}^2  (E_{k,i,j}+F_{k,i,j} ).
\end{align*}
Similarly, we decompose the second term of \eqref{eq:initial main} as follows: 
\begin{align*}
  &\langle \dot{w}_L,\partial_t(\mathcal{A}_{\lambda}^L)\dot{w} \rangle + \sum_{k=0}^{\frac{L-1}{2}} \sum_{i=1}^2 \partial_t (1-\delta_{k,0}\delta_{i,1})D_{k,i,1} \\
   &= \frac{L}{2}  \langle \partial_t(\widetilde{H}_{\lambda}) w_L,w_L \rangle + \sum_{k=0}^{\frac{L-1}{2}} \sum_{i,j=1}^2  (1-\delta_{k,0}\delta_{i,1})(E_{k,i,j}+F_{k,i,j} ).
\end{align*}
Together with \eqref{eq:initial repulsive} and \eqref{eq:initial error}, we obtain the following initial identity of $\mathcal{E}_{L+1}$:
\begin{align}\label{eq:repulsive form}
  &\partial_t \left\{\frac{\mathcal{E}_{L+1}}{2\lambda^{2L}}  +  \sum_{k=0}^{\frac{L-1}{2}} \sum_{i=1}^2 (2-\delta_{k,0}\delta_{i,1}) D_{k,i,1}  \right\} =\frac{2L+1}{2} \langle \partial_t(\widetilde{H}_{\lambda}) w_L,w_L \rangle \\
  & \qquad \qquad \quad   +  \langle  \widetilde{H}_{\lambda} w_L, \mathcal{A}_{\lambda}^L \mathcal{F}_1 \rangle +   \langle \dot{w}_L,\mathcal{A}_{\lambda}^L\mathcal{F}_2 \rangle    +  \sum_{k=0}^{\frac{L-1}{2}} \sum_{i,j=1}^2 (2-\delta_{k,0}\delta_{i,1}) (E_{k,i,j}+F_{k,i,j} ). \nonumber
\end{align}
\textbf{Step 3:} \emph{Second energy identity}. We find out another corrections from $E_{k,i,1}$, which contains $\partial_{tt}(\mathcal{A}_{\lambda}^{L-2k})$. More precisely from Lemma \ref{lem:Leibniz rule},
\begin{align*}
  E_{k,1,1}  &= \langle   \dot{w}_L, \partial_{tt}(\mathcal{A}_{\lambda}^{L-2k})w_{2k} \rangle \\
  &=  \sum_{m=2k}^{L-1} \frac{ \lambda_{tt}}{\lambda^{L+1-m}} \langle (\Phi_{m,L,k}^{(1)})_{\lambda} w_{m}, \dot{w}_L \rangle  + \sum_{m=2k}^{L-1} \frac{ O(b_1^2)}{\lambda^{L+2-m}} \langle (\Phi_{m,L,k}^{(2)})_{\lambda} w_{m}, \dot{w}_L \rangle
\end{align*}
and
\begin{align*}
  E_{k,2,1}  &= -\langle   {w}_L, \partial_{tt}(\mathcal{A}_{\lambda}^{L-2k})\dot{w}_{2k} \rangle \\
  &=  -\sum_{m=2k}^{L-1} \frac{ \lambda_{tt}}{\lambda^{L+1-m}} \langle (\Phi_{m,L,k}^{(1)})_{\lambda} \dot{w}_{m}, {w}_L \rangle  - \sum_{m=2k}^{L-1} \frac{ O(b_1^2)}{\lambda^{L+2-m}} \langle (\Phi_{m,L,k}^{(2)})_{\lambda} \dot{w}_{m}, {w}_L \rangle
\end{align*}
where $\Phi_{m,L,k}^{(j_1)}(y):= \Phi_{m-2k,L-2k}^{(j_1)}(y)$ with $j_1=1,2$, so that 
\[ |\Phi_{m,L,k}^{(j_1)}(y)| \lesssim \frac{1}{1+y^{L+2-m}}.\]
 Here, we cannot treat $\lambda_{tt}$ directly because we do not have estimates on second derivatives of the modulation parameters (and we did not set $\lambda_t=-b_1$). Thus, we add $(b_1)_t$ to $\lambda_{tt}$ and use \eqref{eq:mod bound1},
\begin{align}
  \frac{ \lambda_{tt}}{\lambda^{L+1-m}} \langle (\Phi_{m,L,k}^{(1)})_{\lambda} w_{m}, \dot{w}_L \rangle &= \frac{ (\lambda_{t}+b_1)_t}{\lambda^{L+1-m}} \langle (\Phi_{m,L,k}^{(1)})_{\lambda} w_{m}, \dot{w}_L \rangle \label{eq:further correction} \\
  &+ \frac{ O(b_1^2)}{\lambda^{L+2-m}} \langle (\Phi_{m,L,k}^{(1)})_{\lambda} w_{m}, \dot{w}_L \rangle. \nonumber
\end{align}
We then correct \eqref{eq:further correction} via integration by parts in time with \eqref{eq:evolution w_k}:
\begin{align*}
  &\frac{(\lambda_{t} + b_1)_t}{\lambda^{L+1-m}} \langle (\Phi_{m,L,k}^{(1)})_{\lambda} w_{m} , \dot{w}_L \rangle - \partial_t \left(  \frac{\lambda_{t} + b_1}{\lambda^{L+1-m}} \langle (\Phi_{m,L,k}^{(1)})_{\lambda} w_{m} , \dot{w}_L \rangle  \right) \\
  %&= (\lambda_{t} + b_1) \left\langle  \partial_t\left(\frac{(\Phi_{m,L,k}^{(1)})_{\lambda}}{\lambda^{L+1-m}}\right) w_{m} , \dot{w}_L \right\rangle  \\
  &= ({\lambda_{t} + b_1})\left\langle  \partial_t\left(\frac{1}{\lambda^{L+1-m}}{(\Phi_{m,L,k}^{(1)})_{\lambda}}\right) w_{m} , \dot{w}_L \right\rangle   \\
  & + \frac{\lambda_{t} + b_1}{\lambda^{L+1-m}} \left[  \left\langle (\Phi_{m,L,k}^{(1)})_{\lambda} \partial_tw_{m} , \dot{w}_L \right\rangle + \left\langle (\Phi_{m,L,k}^{(1)})_{\lambda} w_{m} , \partial_t\dot{w}_L \right\rangle  \right] \\
  &  = -\frac{\lambda_t(\lambda_{t} + b_1)}{\lambda^{L+2-m}}  \langle ( \Lambda_{m-L}\Phi_{m,L,k}^{(1)})_{\lambda} w_{m} , \dot{w}_L \rangle \\
  &- \frac{\lambda_{t} + b_1}{\lambda^{L+1-m}}  \langle (\Phi_{m,L,k}^{(1)})_{\lambda} (\dot{w}_m + \partial_t(\mathcal{A}_{\lambda}^m)w + \mathcal{A}_{\lambda}^m \mathcal{F}_1) , \dot{w}_L \rangle \\
  & +\frac{\lambda_{t} + b_1}{\lambda^{L+1-m}}  \langle (\Phi_{m,L,k}^{(1)})_{\lambda} w_{m} , w_{L+2} - \partial_t(\mathcal{A}_{\lambda}^L)\dot{w} - \mathcal{A}_{\lambda}^L \mathcal{F}_2  \rangle.
\end{align*}
We can also obtain the same correction for $E_{k,2,1}$:
\begin{align*}
  &\frac{(\lambda_{t} + b_1)_t}{\lambda^{L+1-m}} \langle (\Phi_{m,L,k}^{(1)})_{\lambda} \dot{w}_{m} , {w}_L \rangle - \partial_t \left(  \frac{\lambda_{t} + b_1}{\lambda^{L+1-m}} \langle (\Phi_{m,L,k}^{(1)})_{\lambda} \dot{w}_{m} , w_L \rangle  \right) \\
  &  = -\frac{\lambda_t(\lambda_{t} + b_1)}{\lambda^{L+2-m}}  \langle (\Lambda_{m-L}\Phi_{m,L,k}^{(1)})_{\lambda} \dot{w}_{m} , {w}_L \rangle  \\
  &- \frac{\lambda_{t} + b_1}{\lambda^{L+1-m}}  \langle (\Phi_{m,L,k}^{(1)})_{\lambda} ( w_{m+2} - \partial_t(\mathcal{A}_{\lambda}^m)\dot{w} - \mathcal{A}_{\lambda}^m \mathcal{F}_2  ) , {w}_L \rangle \\
  &+ \frac{\lambda_{t} + b_1}{\lambda^{L+1-m}}  \langle (\Phi_{m,L,k}^{(1)})_{\lambda} \dot{w}_{m} , \dot{w}_L + \partial_t(\mathcal{A}_{\lambda}^L)w + \mathcal{A}_{\lambda}^L \mathcal{F}_1\rangle.
\end{align*}
Rearranging the existing errors $E_{k,i,j}$, $F_{k,i,j}$ with introducing a new correction notation $D_{k,i,2}$ and new error notation $E^*_{k,i,j}$, $F^*_{k,i,j}$ for $0\le k \le \frac{L-1}{2}$ and $i=1,2$:
\begin{equation}
  E_{k,i,1} - \partial_t D_{k,i,2}  + E_{k,i,2} + F_{k,i,1} + F_{k,i,2} = E^*_{k,i,1} + E^*_{k,i,2} + F^*_{k,i,1} + F^*_{k,i,2}
\end{equation}
where
\begin{align*}
  D_{k,1,2}  =&  \sum_{m=2k}^{L-1}  \frac{\lambda_{t} + b_1}{\lambda^{L+1-m}} \langle (\Phi_{m,L,k}^{(1)})_{\lambda} w_{m} , \dot{w}_L \rangle , \\
  E_{k,1,1}^* =&-\sum_{m=2k}^{L-1}\frac{\lambda_t(\lambda_{t} + b_1)}{\lambda^{L+2-m}}  \langle (\Lambda_{m-L}\Phi_{m,L,k}^{(1)})_{\lambda} w_{m} , \dot{w}_L \rangle  \\
     &-\sum_{m=2k}^{L-1}  \frac{\lambda_{t} + b_1}{\lambda^{L+1-m}}  \langle (\Phi_{m,L,k}^{(1)})_{\lambda} (\dot{w}_m + \partial_t(\mathcal{A}_{\lambda}^m)w) , \dot{w}_L \rangle \\
   &+ \sum_{m=2k}^{L-1} \frac{\lambda_{t} + b_1}{\lambda^{L+1-m}}  \langle (\Phi_{m,L,k}^{(1)})_{\lambda} w_{m} , w_{L+2} - \partial_t(\mathcal{A}_{\lambda}^L)\dot{w}   \rangle,\\
  E_{k,1,2}^*  =&  E_{k,1,2}+  \sum_{m=2k}^{L-1}  \frac{O(b_1^2)}{\lambda^{L+2-m}}  \langle (\Phi_{m,L,k}^{(2)})_{\lambda} w_{m} , \dot{w}_L \rangle,\\
  F_{k,1,1}^*=& F_{k,1,1}  - \sum_{m=2k}^{L-1}  \frac{\lambda_{t} + b_1}{\lambda^{L+1-m}}  \langle (\Phi_{m,L,k}^{(1)})_{\lambda} \mathcal{A}_{\lambda}^m \mathcal{F}_1 , \dot{w}_L \rangle\\
  F_{k,1,2}^*=& F_{k,1,2} - \sum_{m=2k}^{L-1} \frac{\lambda_{t} + b_1}{\lambda^{L+1-m}}  \langle (\Phi_{m,L,k}^{(1)})_{\lambda} w_{m} , \mathcal{A}_{\lambda}^L \mathcal{F}_2   \rangle
\end{align*}
and
\begin{align*}
  \quad D_{k,2,2} = & -\sum_{m=2k}^{L-1}    \frac{\lambda_{t} + b_1}{\lambda^{L+1-m}} \langle (\Phi_{m,L,k}^{(1)})_{\lambda} \dot{w}_{m} , w_L \rangle , \\
  E_{k,2,1}^*  =&  \sum_{m=2k}^{L-1} \frac{\lambda_t(\lambda_{t} + b_1)}{\lambda^{L+2-m}}  \langle (\Lambda_{m-L}\Phi_{m,L,k}^{(1)})_{\lambda} \dot{w}_{m} , {w}_L \rangle  \\
   +&\sum_{k=2m}^{L-1}\frac{\lambda_{t} + b_1}{\lambda^{L+1-m}}  \langle (\Phi_{m,L,k}^{(1)})_{\lambda} ( w_{m+2} - \partial_t(\mathcal{A}_{\lambda}^m)\dot{w}  ) , {w}_L \rangle\\
   -& \sum_{m=2k}^{L-1} \frac{\lambda_{t} + b_1}{\lambda^{L+1-m}}  \langle (\Phi_{m,L,k}^{(1)})_{\lambda} \dot{w}_{m} , \dot{w}_L + \partial_t(\mathcal{A}_{\lambda}^L)w \rangle,\\
  E_{k,2,2}^*  =&  E_{k,2,2} -  \sum_{m=2k}^{L-1}  \frac{O(b_1^2)}{\lambda^{L+2-m}} \langle  (\Phi_{m,L,k}^{(2)})_{\lambda} \dot{w}_{m} , {w}_L \rangle,\\
  F_{k,2,1}^*=& F_{k,2,1}  - \sum_{m=2k}^{L-1} \frac{\lambda_{t} + b_1}{\lambda^{L+1-m}}  \langle (\Phi_{m,L,k}^{(1)})_{\lambda} \dot{w}_{m} , \mathcal{A}_{\lambda}^L \mathcal{F}_1\rangle\\
  F_{k,2,2}^*=& F_{k,2,2} - \sum_{m=2k}^{L-1} \frac{\lambda_{t} + b_1}{\lambda^{L+1-m}}  \langle (\Phi_{m,L,k}^{(1)})_{\lambda}  \mathcal{A}_{\lambda}^m \mathcal{F}_2  , {w}_L \rangle,
\end{align*}
we obtain the following modified energy identity:
\begin{align}\label{eq:final form}
  &\partial_t \left\{\frac{\mathcal{E}_{L+1}}{2\lambda^{2L}}  +  \sum_{k=0}^{\frac{L-1}{2}} \sum_{i,j=1}^2 (2-\delta_{k,0}\delta_{i,1}) D_{k,i,j}  \right\} =\frac{2L+1}{2} \langle \partial_t(\widetilde{H}_{\lambda}) w_L,w_L \rangle \\
  & \qquad \qquad \quad   +  \langle  \widetilde{H}_{\lambda} w_L, \mathcal{A}_{\lambda}^L \mathcal{F}_1 \rangle +   \langle \dot{w}_L,\mathcal{A}_{\lambda}^L\mathcal{F}_2 \rangle    +  \sum_{k=0}^{\frac{L-1}{2}} \sum_{i,j=1}^2 (2-\delta_{k,0}\delta_{i,1}) (E^*_{k,i,j}+F^*_{k,i,j} ). \nonumber
\end{align}
\textbf{Step 4:} \emph{Error estimation}.
All we need is to estimate all inner products except the repulsive one $\langle \partial_t(\widetilde{H}_{\lambda}) w_L, w_L \rangle$. We can classify such inner products into two main categories: quadratic terms with respect to $w$ (i.e. $D_{k,i,j}$ and $E^*_{k,i,j}$), those involving $\mathcal{F}_i$, $i=1,2$ (i.e. $F^*_{k,i,j}$ and \eqref{eq:initial error}). 
 
(i) $D_{k,i,j}$ \emph{terms}. From \eqref{eq:1-t} and Lemma \ref{lem:Leibniz rule}, all inner products of $D_{k,i,j}$ can be written as sums of terms of the form: $0\le m \le L-1$,
\begin{align*}
  \frac{O(b_1)}{\lambda^{2L}} \langle \Phi_{m,L} \varepsilon_m  ,  \dot{\varepsilon}_L \rangle,\quad  \frac{O(b_1)}{\lambda^{2L}} \langle \Phi_{m,L} \dot{\varepsilon}_m   ,  {\varepsilon}_L \rangle,\quad |\Phi_{m,L}(y)| \lesssim  \frac{1}{1+y^{L+2-m}}.
\end{align*}
Indeed, the $\Phi_{m,L}$'s included in each of the above inner products are different functions (ex. $\Phi_{m-2k,L-2k}^{(j_1)}$, $\Phi_{m,L,k}^{(j_2)}$, $\Lambda_{m-L}\Phi_{m,L,k}^{(1)}$\dots), but we abuse the notation because they are all rational functions with the same asymptotics. From the coercive property \eqref{eq:coercivity bound}, we obtain the desired bound for the correction in \eqref{eq:monotonicity}:
\begin{align*}
   | \langle \Phi_{m,L} \varepsilon_m  ,  \dot{\varepsilon}_L \rangle | &\lesssim  {\left\lVert \frac{\varepsilon_m}{1+y^{L+2-m}}  \right\rVert}_{L^2} \sqrt{\mathcal{E}_{L+1}}\lesssim C(M) \mathcal{E}_{L+1}, \\
  | \langle \Phi_{m,L} \dot{\varepsilon}_m   ,  {\varepsilon}_L \rangle | & \lesssim {\left\lVert \frac{1+|\log y|}{1+y^{L+1-m}} \dot{\varepsilon}_m \right\rVert}_{L^2} \sqrt{\mathcal{E}_{L+1}} 
  \lesssim C(M) \mathcal{E}_{L+1}. 
\end{align*}
(ii) $E^*_{k,i,j}$ \emph{terms}. Similarly, all inner products of $E^*_{k,i,j}$ can be written as sums of terms of the form: for $0\le m,n \le L-1$, 
\begin{align*}
  \frac{O(b_1^2)}{\lambda^{2L+1}}   \langle \Phi_{m,L} \varepsilon_m  ,  \dot{\varepsilon}_L \rangle ,\quad \frac{O(b_1^2)}{\lambda^{2L+1}} \langle \Phi_{m,L} \dot{\varepsilon}_m   ,  {\varepsilon}_L \rangle ,\quad  \frac{O(b_1^2)}{\lambda^{2L+1}}\langle \Phi_{m,L} \dot{\varepsilon}_m   , \Phi_{n,L} {\varepsilon}_n \rangle \\
  \frac{O(b_1^{2})}{\lambda^{2L+1}}   \langle \Phi_{m,L} \dot{\varepsilon}_m  ,  \dot{\varepsilon}_L \rangle ,\quad \frac{O(b_1^{2})}{\lambda^{2L+1}} \langle \Phi_{m,L} {\varepsilon}_m   ,  {\varepsilon}_{L+2} \rangle ,\quad  \frac{O(b_1^{2})}{\lambda^{2L+1}}\langle \Phi_{m,L} {\varepsilon}_{m+2}   ,  {\varepsilon}_L \rangle,
\end{align*}
which are bounded by
\[\frac{b_1^2}{\lambda^{2L+1}} C(M) \mathcal{E}_{L+1}.\]
(iii) $F^*_{k,i,j}$ \emph{and} \eqref{eq:initial error}. Recall $\mathcal{F}_1 = \lambda^{-1}\mathcal{F}_{\lambda}$ and $\mathcal{F}_2 = \lambda^{-2}\dot{\mathcal{F}}_{\lambda}$, all inner products of $F^*_{k,i,j}$ can be written as sums of terms of the form: for $0\le m \le L-1$
\begin{align}
  &\frac{O(b_1)}{\lambda^{2L+1}}   \langle \Phi_{m,L} \mathcal{A}^m \mathcal{F} ,  \dot{\varepsilon}_L \rangle ,\quad  \frac{O(b_1)}{\lambda^{2L+1}}\langle \Phi_{m,L} \dot{\varepsilon}_m   , \mathcal{A}^L \mathcal{F}  \rangle , \quad \frac{O(b_1)}{\lambda^{2L+1}}\langle \Phi_{m,L} {\varepsilon}_m   , \mathcal{A}^L \dot{\mathcal{F}}  \rangle, \label{eq:rough F} \\
  &\frac{O(b_1)}{\lambda^{2L+1}} \langle \Phi_{m,L} \mathcal{A}^m \dot{\mathcal{F}}   ,  {\varepsilon}_L \rangle ,\quad  \frac{1}{\lambda^{2L+1}}\langle  \varepsilon_{L+1}, \mathcal{A}^{L+1} \mathcal{F} \rangle ,\quad    \frac{1}{\lambda^{2L+1}}\langle \dot{\varepsilon}_L,\mathcal{A}^L\dot{\mathcal{F}} \rangle \label{eq:sharp F}. 
\end{align}
We claim that $\mathcal{F}$ and $\dot{\mathcal{F}}$ satisfy the following estimates: for $0\le k \le L-1$, 
\begin{align}
  {\left\lVert \mathcal{A}^{L+1} \mathcal{F} \right\rVert}_{L^2} +  {\left\lVert \mathcal{A}^L \dot{\mathcal{F}} \right\rVert}_{L^2} & \lesssim b_1 \left[ \frac{b_1^{L+1}}{|\log b_1|} + \sqrt{\frac{\mathcal{E}_{L+1}}{\log M}} \right], \label{eq:main FF} \\
  {\left\lVert \frac{1+|\log y|}{1+y^{L+1-k}} \mathcal{A}^k \mathcal{F} \right\rVert}_{L^2}& \lesssim b_1^{L+2}|\log b_1|^C, \label{eq:weighted F}\\
  {\left\lVert \frac{1+|\log y|}{1+y^{L+1-k}} \mathcal{A}^k \dot{\mathcal{F}} \right\rVert}_{L^2}&\lesssim  \frac{b_1^{L+1}}{|\log b_1|} + \sqrt{\frac{\mathcal{E}_{L+1}}{\log M}} \label{eq:weighted dot F}.
\end{align}
Assuming these claims \eqref{eq:main FF}, \eqref{eq:weighted F} and \eqref{eq:weighted dot F} with the coercivity \eqref{eq:coercivity bound}, we can estimate $F^*_{k,i,j}$ terms as follows: the three inner products in \eqref{eq:rough F} are bounded by
\[\frac{b_1}{\lambda^{2L+1}} C(M) b_1^{L+2}|\log b_1|^C  \sqrt{\mathcal{E}_{L+1}} .\]
For the three inner products in \eqref{eq:sharp F}, we obtain the sharp bound
\[\frac{b_1}{\lambda^{2L+1}} \left(  \frac{b_1^{L+1}}{|\log b_1|} + \sqrt{\frac{\mathcal{E}_{L+1}}{\log M}} \right)\sqrt{\mathcal{E}_{L+1}}\]
 from \eqref{eq:main FF}, \eqref{eq:weighted dot F} and the sharp coercivity bound
 \[ {\left\lVert \frac{\varepsilon_L}{y(1+|\log y|)} \right\rVert}_{L^2}^2 \le C\langle \widetilde{H} \varepsilon_L, \varepsilon_L \rangle \le C{\mathcal{E}_{L+1}}.\]
Hence, it remains to prove \eqref{eq:main FF}, \eqref{eq:weighted F} and \eqref{eq:weighted dot F}.

\textbf{Step 5:} Proof of \eqref{eq:main FF}, \eqref{eq:weighted F} and \eqref{eq:weighted dot F}. Recall \eqref{eq:evolution W}, we have $\boldsymbol{\mathcal{F}}=(\mathcal{F},\dot{\mathcal{F}})^t$ and
\[\begin{pmatrix}
  \mathcal{F} \\
  \dot{\mathcal{F}}
\end{pmatrix}= -\widetilde{\mathbf{Mod}}(t) -\tilde{\boldsymbol{\psi}}_{b} -\boldsymbol{ NL}(\boldsymbol{\varepsilon})-\boldsymbol{L}{(\boldsymbol{\varepsilon})}, \quad \boldsymbol{ NL}(\boldsymbol{\varepsilon}) = \begin{pmatrix}
  0 \\
  NL(\varepsilon)
\end{pmatrix},\quad   \boldsymbol{L}(\boldsymbol{\varepsilon}) = \begin{pmatrix}
  0 \\
  L(\varepsilon)
\end{pmatrix}.\]
Thus, we will estimate each of the above four errors.

(i) $\tilde{\boldsymbol{\psi}}_b$ \emph{term}. It directly follows from the global and logarithmic weighted bounds of Proposition \ref{prop:local approx}.

(ii) ${\widetilde{\mathbf{Mod}}}(t)$ \emph{term}. Recall \eqref{eq:mod tilde}, we have
\begin{align}
  {\widetilde{\mathbf{Mod}}}(t)=&-\left(\frac{\lambda_s}{\lambda}+b_1 \right)\left( \boldsymbol{\Lambda }\boldsymbol{Q} + \sum_{i=1}^L b_i\boldsymbol{\Lambda }( \chi_{B_1} \boldsymbol{T}_i) +  \sum_{i=2}^{L+2} \boldsymbol{\Lambda }(\chi_{B_1} \boldsymbol{S}_i) \right) \nonumber \\
  &+ \sum_{i=1}^L \left( (b_{i})_s+(i-1 + c_{b,i})b_1b_i -b_{i+1}  \right) \chi_{B_1}\left( \boldsymbol{T}_i   +  \sum_{j=i+1}^{L+2}  \frac{\partial \boldsymbol{S}_j}{\partial b_i} \right) \label{eq:mod in energy}. 
\end{align}
Due to Lemma \ref{lem:mod eqn}, the logarithmic weighted bounds \eqref{eq:weighted F} and \eqref{eq:weighted dot F} are derived from the finiteness of the following integrals
\begin{align*}
  \int  \left\lvert \frac{1+|\log y|}{1+y^{L+1-k}}\mathcal{A}^k \left[ \Lambda Q + \sum_{i=1}^L b_i\Lambda_{1-\overline{i}}(\chi_{B_1}T_i)+  \sum_{i=2}^{L+2}\Lambda_{1-\overline{i}}(\chi_{B_1}S_i) \right] \right\rvert^2 \lesssim 1\\
  \sum_{i=1}^L\int  \left\lvert \frac{1+|\log y|}{1+y^{L+1-k}}\mathcal{A}^k \left[ \chi_{B_1}T_i+ \chi_{B_1}\sum_{j=i+1}^{L+2} \frac{\partial {S}_j}{\partial b_i} \right] \right\rvert^2 \lesssim 1,
\end{align*}
which comes from the admissibility of $\boldsymbol{T}_i$ and Lemma \ref{lem:admissible bounds}. For the global bounds \eqref{eq:main FF}, we need to gain one extra $b_1$ as follows: since $A \Lambda Q=0$, the admissibility of $\boldsymbol{T}_i$ and Lemma \ref{lem:admissible bounds} imply  
\begin{align*}
  &\int  \left\lvert \mathcal{A}^{L+1} \Lambda Q + \sum_{i=1}^L b_i\mathcal{A}^{L+1-\overline{i}}\left[\Lambda_{1-\overline{i}}(\chi_{B_1}T_i)\right]+  \sum_{i=2}^{L+2} \mathcal{A}^{L+1-\overline{i}}\left[\Lambda_{1-\overline{i}}(\chi_{B_1}S_i)  \right]\right\rvert^2 \\
  &\lesssim  \sum_{i=1}^L  \int_{y\le 2B_1}  b_1^i \left\lvert \frac{(1+|\log y|)y^{i-2}}{1+y^L}  \right\rvert^2 + \sum_{i=2}^{L+1} b_1^{2i} + \frac{b_1^{2(L+1)}}{|\log b_1|^2} \lesssim b_1^2.
\end{align*}
For \eqref{eq:mod in energy}, we additionally use the cancellation $\mathcal{A}^L T_i = 0$ for $1\le i \le L$ to estimate
\begin{align*}
  \sum_{i=1}^L \int |\mathcal{A}^{L+1-\overline{i}} (\chi_{B_1} T_i) |^2 &\lesssim \sum_{i=1}^L  \int_{y\sim B_1} \left\lvert \frac{y^{i-2} \log y}{y^L} \right\rvert^2 \lesssim  b_1^2.\\
  \sum_{j=i+1}^{L+2} \int  \left\lvert \mathcal{A}^{L+1-\overline{i}} \left[\chi_{B_1}  \frac{\partial {S}_j}{\partial b_i} \right] \right\rvert^2 &\lesssim \sum_{j=i+1}^{L+2} b_1^{2(j-i)} + \frac{b_1^{2(L+1-i)}}{|\log b_1|^2} \lesssim b_1^2. 
\end{align*}
Hence, \eqref{eq:main FF} comes from Lemma \ref{lem:mod eqn}:
\[{\left\lVert \mathcal{A}^{L+1} {\widetilde{\mathrm{Mod}}}(t) \right\rVert}_{L^2}+{\left\lVert \mathcal{A}^L \dot{\widetilde{\mathrm{Mod}}}(t) \right\rVert}_{L^2}\lesssim b_1\left[ \frac{b_1^{L+1}}{|\log b_1|} + \sqrt{\frac{\mathcal{E}_{L+1}}{\log M}} \right].\]
For the remaining two terms, $\boldsymbol{NL}(\boldsymbol{\varepsilon})$ and $\boldsymbol{L}(\boldsymbol{\varepsilon})$, we follow the approach developed in \cite{RaphaelSchweyer2014Anal.PDE}. We deal with the case $y\le 1$ and $y\ge 1$ separately.

(iii) $\boldsymbol{NL}(\boldsymbol{\varepsilon})$ \emph{term}: 
(a) $y\le 1$.
From a Taylor Lagrange formula in Lemma \ref{lem:interpolation bound}, $NL(\varepsilon)$ also satisfies a Taylor Lagrange formula
\begin{equation}
  NL(\varepsilon)= \sum_{i=0}^{\frac{L-1}{2}}{c}_i y^{2i+1} + {r}_{\varepsilon},
\end{equation}
where
\begin{align}\label{eq:NL origin remainder}
  |{c}_i|\lesssim C(M) \mathcal{E}_{L+1},\quad |\mathcal{A}^k {r}_{\varepsilon}| \lesssim y^{L-k}|\log y|C(M) \mathcal{E}_{L+1}, \quad 0\le k \le L.
\end{align}
Since the expansion part of $NL(\varepsilon)$ is an odd function, that of $\mathcal{A}^k NL(\varepsilon)$ also has a single parity from the cancellation $A(y)=O(y^2)$. Using \eqref{eq:NL origin remainder}, we obtain
\begin{equation}\label{eq:NL bound near zero}
  |\mathcal{A}^k NL(\varepsilon)(y)| \lesssim C(M)|\log y| \mathcal{E}_{L+1},\quad 0\le k \le L,  
\end{equation}
and thus we conclude
\[{\left\lVert \mathcal{A}^L NL(\varepsilon) \right\rVert}_{L^2(y\le 1)}+{\left\lVert \frac{1+|\log y|^C}{1+y^{L+1-k}} \mathcal{A}^k NL(\varepsilon) \right\rVert}_{L^2(y\le 1)} \lesssim C(M) \mathcal{E}_{L+1}\lesssim b_1^{2L+1}.\]
(b) $y\ge 1$. Let
\begin{equation}\label{eq:decompose NL}
  NL(\varepsilon)=\zeta^2 N_1(\varepsilon), \quad \zeta=\frac{\varepsilon}{y},\quad N_1(\varepsilon) = \int_0^1 (1-\tau)f''(\tilde{Q}_b+\tau \varepsilon)d\tau .
\end{equation}
We have the following bounds for $i\ge 0$, $j\ge 1$ and $1\le i+j \le L$,
\begin{align}
  {\left\lVert \frac{\partial_y^i \zeta}{y^{j-1}} \right\rVert}_{L^{\infty}(y\ge 1)} + {\left\lVert \frac{\partial_y^i \zeta}{y^{j}} \right\rVert}_{L^{2}(y\ge 1)} &\lesssim |\log b_1|^{C(K)} b_1^{m_{i+j+1}}, \quad {\left\lVert \zeta \right\rVert}_{L^{2}(y\ge 1)} \lesssim 1 \label{eq:zeta bound}\\
  |N_1(\varepsilon)| \lesssim 1 ,\quad |\partial_y^k N_1(\varepsilon)| &\lesssim |\log b_1|^{C(K)} \left[ \frac{1}{y^{k+1}} + b_1^{m_{k+1}} \right],\; 1\le k \le L \label{eq:N_1 bound}
\end{align}
where 
\begin{equation}
  m_{k+1} = \begin{cases}
    kc_1 & \textnormal{if}\quad 1\le k \le L-2, \\
    L & \textnormal{if} \quad k=L-1, \\
    L+1 & \textnormal{if}\quad  k =L.
  \end{cases}
\end{equation}
The estimates \eqref{eq:zeta bound} are consequences of Lemma \ref{lem:interpolation bound} and the orbital stability \eqref{eq:orbital stability}.
We can prove the estimates \eqref{eq:N_1 bound} by borrowing \textit{Proof of} (3-77) in \cite{RaphaelSchweyer2014Anal.PDE} (p. 1768 line 1 of \cite{RaphaelSchweyer2014Anal.PDE}), since we can obtain the crude bound
\[|\partial_y^k \tilde{Q}_b | \lesssim |\log b_1|^C \left[ \frac{1}{y^{k+1}} + \sum_{i=1}^{\frac{L+1}{2}} b_1^{2i} y^{2i-1-k} 1_{y\le 2B_1} \right] \lesssim \frac{|\log b_1|^C}{y^{k+1}}.\]
Returning to the estimates for $\boldsymbol{NL}(\boldsymbol{\varepsilon})$, we have the trivial bound 
\[ \textnormal{for } 0\le k \le L, \quad\left\lvert\frac{1+|\log y|^C}{y^{L+1-k}}\mathcal{A}^k NL(\varepsilon) \right\rvert \lesssim \left\lvert\frac{\mathcal{A}^k NL(\varepsilon)}{y^{L-k}} \right\rvert,  \]
 \eqref{eq:decompose NL} and \eqref{eq:N_1 bound} imply
\begin{align*}
   \left\lvert\frac{\mathcal{A}^k NL(\varepsilon)}{y^{L-k}} \right\rvert 
   & \lesssim \sum_{k=0}^L \frac{|\partial_y^k NL(\varepsilon)|}{y^{L-k}}  \lesssim \sum_{k=0}^L \frac{1}{y^{L-k}} \sum_{i=0}^k  |\partial_y^i \zeta^2| |\partial_y^{k-i}N_1(\varepsilon)| \\
   &\lesssim \sum_{k=0}^L \frac{|\log b_1|^{C(K)}}{y^{L-k}}  \left[ {|\partial_y^k \zeta^2|} + \sum_{i=0}^{k-1} b_1^{m_{k-i+1}} {|\partial_y^i \zeta^2|} \right] \\
   &\lesssim  \sum_{k=0}^L \frac{|\log b_1|^{C(K)}}{y^{L-k}}   \left[ \sum_{i=0}^k {|\partial_y^i \zeta||\partial_y^{k-i}\zeta|} + \sum_{i=0}^{k-1} \sum_{j=0}^i  b_1^{m_{k-i+1}} {|\partial_y^j \zeta||\partial_y^{i-j}\zeta|} \right].
\end{align*}
Denote $I_1=k-i$, $I_2=i$, there exists $J_2 \in \mathbb{N}$ such that
\[\max(0,1-i) \le J_2 \le \min(L+1-k,L-i),\quad J_1=L+1-k-J_2,\]
we have
\[1\le I_1+J_1 \le L,\; 1\le I_2+J_2 \le L,\;I_1+I_2+J_1+J_2=L+1.\]
Thus
\begin{align*}
  {\left\lVert \frac{\partial_y^i\zeta \cdot \partial_y^{k-i}\zeta}{y^{L-k}} \right\rVert}_{L^{2}(y\ge 1)} &\le {\left\lVert \frac{\partial_y^{I_1}\zeta}{y^{J_1-1}} \right\rVert}_{L^{\infty}(y\ge 1)}{\left\lVert \frac{\partial_y^{I_2}\zeta}{y^{J_2}} \right\rVert}_{L^{2}(y\ge 1)} \\
  &\lesssim |\log b_1|^{C(K)}  b_1^{m_{I_1+J_1+1}}b_1^{m_{I_2+J_2+1}} \lesssim b_1^{\delta(L)} b_1^{L+2}
\end{align*}
since
\begin{align*}
  m_{I_1+J_1+1} + m_{I_2+J_2+1} &= \begin{cases}
    (L+1)c_1 & \textnormal{if $I_1+J_1<L-1$ and $I_2+J_2<L-1$},\\
    L + 2c_1 & \textnormal{if $I_1+J_1=L-1$ or $I_2+J_2=L-1$},\\
    L + 1+ c_1 & \textnormal{if $I_1+J_1=L$ or $I_2+J_2=L$}
  \end{cases}\\
  & > L+2.
\end{align*}
We calculate the latter term similarly except for the case $k=L$ and $0\le i=j \le k-1$. Here, we use the energy bound ${\left\lVert \zeta \right\rVert}_{L^{2}(y\ge 1)} \lesssim 1$, 
\begin{align*}
  |\log b_1|^{C(K)} b_1^{m_{L-i+1}}  {\left\lVert \partial_y^i \zeta \cdot \zeta \right\rVert}_{L^{2}(y\ge 1)}&\lesssim |\log b_1|^{C(K)} b_1^{m_{L-i+1}}{\left\lVert \partial_y^i \zeta \right\rVert}_{L^{\infty}(y\ge 1)}\\
  &\lesssim \begin{cases}
    |\log b_1|^{C(K)} b_1^{(L+1)c_1} & \textnormal{if } 0<i<L-1 \\
    |\log b_1|^{C(K)} b_1^{L+2c_1} & \textnormal{if } i=1,L-2 \\
    |\log b_1|^{C(K)} b_1^{L+1 + c_1} & \textnormal{if } i=0,L-1 
  \end{cases}\\ &\lesssim b_1^{\delta(L)} b_1^{L+2}.
\end{align*}
The remaining case can be estimated by the following inequalities: since $k-i\ge 1$, $I_1+J_1\ge 1$, $I_2+J_2\ge 1$ and $I_1+I_2+J_1+J_2=L+1-(k-i)$, 
\begin{align*}
  |\log b_1|^{C(K)} b_1^{m_{k-i+1} + m_{I_1+J_1+1} + m_{I_2+J_2+1}} &\lesssim \begin{cases}
    |\log b_1|^{C(K)} b_1^{(L+1)c_1} & \textnormal{if } k-i <L-1 \\
    |\log b_1|^{C(K)} b_1^{L+2c_1} & \textnormal{if } k-i=L-1 
  \end{cases} \\
  &\lesssim b_1^{\delta(L)} b_1^{L+2}.
\end{align*}
(iv) $\boldsymbol{L}(\boldsymbol{\varepsilon})$ \emph{term} : 
(a) $y\le 1$. Similar to the case $\boldsymbol{NL}(\boldsymbol{\varepsilon})$, we obtain a Taylor Lagrange formula for $\boldsymbol{L}(\boldsymbol{\varepsilon})$:
\begin{equation}
  L(\varepsilon)=  b_1^2\left[\sum_{i=0}^{\frac{L-1}{2}}{\tilde{c}}_i y^{2i+1} + {\tilde{r}}_{\varepsilon}\right],
\end{equation}
where
\begin{align}\label{eq:L origin remainder}
  |\tilde{c}_i|\lesssim C(M) \sqrt{\mathcal{E}_{L+1}},\quad |\mathcal{A}^k {\tilde{r}}_{\varepsilon}| \lesssim y^{L-k}|\log y|C(M) \sqrt{\mathcal{E}_{L+1}}, \quad 0\le k \le L.
\end{align}
Using the cancellation $A(y)=O(y^2)$ and \eqref{eq:L origin remainder}, we obtain
\begin{equation}\label{eq:L bound near zero}
  |\mathcal{A}^k L(\varepsilon)(y)| \lesssim C(M)b_1^2|\log y| \sqrt{\mathcal{E}_{L+1}},\quad 0\le k \le L,  
\end{equation}
and thus we conclude
\[{\left\lVert \mathcal{A}^L L(\varepsilon) \right\rVert}_{L^2(y\le 1)}+{\left\lVert \frac{1+|\log y|^C}{1+y^{L+1-k}} \mathcal{A}^k L(\varepsilon) \right\rVert}_{L^2(y\le 1)} \lesssim C(M) b_1^2 \sqrt{\mathcal{E}_{L+1}}.\]
(b) $y\ge 1$. Let 
\[L(\varepsilon)=\varepsilon N_2({\alpha}_{b}),\quad N_2({\alpha}_{b})=\frac{f'(\tilde{Q}_{b})-f'(Q)}{y^2}=\frac{\chi_{B_1}\alpha_b}{y^2}\int_0^1 f''(Q+\tau \chi_{B_1}\alpha_{b})d\tau.\]
Similar to \eqref{eq:N_1 bound}, we have the bound
\begin{equation}
  |\partial_y^k N_2|\lesssim \frac{b_1^2 |\log b_1|^C}{y^{k+1}},\quad 0\le k \le L,
\end{equation}
this yields the desired result since $L(\varepsilon)$ satisfies the pointwise bound
\begin{equation}\label{eq:L bound near infinity}
  \left\lvert\frac{\mathcal{A}^k L(\varepsilon)}{y^{L-k}} \right\rvert \lesssim \sum_{i=0}^k \frac{|\partial_y^i \varepsilon||\partial_y^{k-i} N_2|}{y^{L-k}} \lesssim b_1^2 |\log b_1|^C \sum_{i=0}^k \frac{|\partial_y^i\varepsilon|}{y^{L+1-i}}. \qedhere
\end{equation}
\end{proof}
%%%%%%%%%%%%%%%%%%%%%%%%%%%%%%%%%%%%%%
%%%%%%%%%%%%%%%%%%%%%%%%%%%%%%%%%%%%%%

\section{Proof of the main theorem}
\subsection{Proof of Proposition \ref{prop:bootstrap}} 
\begin{proof}
  \textbf{Step 1:} \emph{Control of the scaling law}. We have the bound
  \[-\frac{\lambda_s}{\lambda}=\frac{c_1}{s} + \frac{d_1}{s\log s} + O\left( \frac{1}{s(\log s)^{\beta}} \right).\]
  We rewrite as
\[\left\lvert \frac{d}{ds}\left(\log \left( s^{c_1}(\log s)^{d_1} \lambda(s) \right)\right)\right\rvert \lesssim \frac{1}{s (\log s)^{\beta}},\]
integration and \eqref{eq:lambda renormalized} 
give
\begin{equation}\label{eq:scaling law}
  \lambda(s)=\frac{s_0^{c_1}(\log s_0)^{d_1}}{s^{c_1}(\log s)^{d_1}} \left(1+O\left( \frac{1}{(\log s_0)^{\beta-1}} \right) \right).
\end{equation}
Note that
  \begin{equation}\label{eq:derivative scaling law}
    \frac{d}{ds}\left( \frac{ b_1^{2n} (\log b_1)^{2m} }{\lambda^{2k-2}}\right) = 2\frac{ b_1^{2n-1} (\log b_1)^{2m} }{\lambda^{2k-2}} \left[ (k-1)b_1^2 + b_{1s} \left( n + \frac{m}{\log b_1} \right) +O(b_1^{L+2}) \right] . 
  \end{equation}
From Lemma \ref{lem:mod eqn} with \eqref{eq:difference b_k}, \eqref{eq:exact b system} and \eqref{eq:bootstrap modes stables},
\begin{align*}
  (k-1)b_1^2 + b_{1s} \left( n + \frac{m}{\log b_1} \right) & =  (k-1)b_1^2 + \left( b_2 - c_{b_1,1} b_1^2 \right) \left( n + \frac{m}{\log b_1} \right) + O(b_1^{L+2} )\\
  &= (k-1)b_1^2 + n b_2 + \frac{2mb_2-nb_1^2}{2\log b_1} + O\left( \frac{b_1^2}{(\log b_1)^2} \right) \\
  &= \frac{(k-1)c_1^2 + n c_2}{s^2} + \frac{ 2(k-1)c_1d_1 - nd_2 - mc_2 + \frac{n}{2}c_1^2 }{s^2 \log s}\\
  &+ O\left( \frac{1}{s^2 (\log s)^{\beta}}\right). 
\end{align*}
The recurrence relations \eqref{eq:c recurrence} and \eqref{eq:d recurrence} imply
\begin{equation*}
  (k-1)c_1^2 + nc_2 = c_1\left( (k-1)\frac{\ell}{\ell-1} -n \right) 
\end{equation*}
and
\begin{equation*}
  2(k-1)c_1d_1 -nd_2+ \frac{n}{2}c_1^2 = d_1\left(2(k-1)c_1 +n \right) <0.
\end{equation*}
Hence, if we set $n=L+1$ and $m=-1$ for $k=L+1$, $c_1 \ge \frac{L}{L-1}$ implies
\begin{equation*}
  (k-1)b_1^2 + b_{1s} \left( n + \frac{m}{\log b_1} \right) \ge \frac{1}{s^2} \left(\frac{c_1}{L-1} + O\left( \frac{1}{\log s} \right) \right) > 0 
\end{equation*} 
and if we set $n=(k-1)c_1$ and large enough $m=m(k,L)$ for $k\le L$,
\begin{equation*}
  (k-1)b_1^2 + b_{1s} \left( n + \frac{m}{\log b_1} \right) \ge \frac{c_1}{s^2\log s} \left( \frac{m}{2} + O\left(\frac{1}{(\log s)^{\beta-1}} \right)\right) > 0
\end{equation*}
for all $s\in [s_0,s^*)$ with sufficiently large $s_0$. Thus, 
\begin{equation}\label{eq:initial scaling law control1}
  \frac{b_1^{2(L+1)}(0)}{(\log b_1(0))^2 \lambda^{2L}(0)} \le  \frac{b_1^{2(L+1)}(t)}{(\log b_1(t))^2 \lambda^{2L}(t)}
\end{equation}
and
\begin{equation}
  \frac{b_1^{2(k-1)c_1}(0)|\log b_1(0)|^{m}}{ \lambda^{2(k-1)}(0)} \le  \frac{b_1^{2(k-1)c_1}(t)|\log b_1(t)|^{m}}{ \lambda^{2(k-1)}(t)}.
\end{equation}
\textbf{Step 2:} \emph{Improved bound on} $\mathcal{E}_{L+1}$. 
  We integrate the Lyapunov monotonicity \eqref{eq:monotonicity} and inject the bootstrap bounds \eqref{eq:initial energies} and \eqref{eq:bootstrap epsilon},
\begin{align}
  \mathcal{E}_{L+1}(t) & \lesssim  \frac{\lambda^{2L}(t)}{\lambda^{2L}(0)}  (1+b_1C(M))\mathcal{E}_{L+1}(0) + b_1C(M) \mathcal{E}_{L+1}(t) \nonumber \\
  &  + \left[ \frac{K}{\sqrt{\log M}} +\sqrt{K} \right] \lambda^{2L}(t) \int_0^t \frac{b_1}{\lambda^{2L+1}} \frac{b_1^{2(L+1)}}{|\log b_1|^2} d\tau \nonumber \\
  &\lesssim  \frac{b_1^{2(L+1)}(t)}{|\log b_1(t)|^2} + \left[  \frac{K}{\sqrt{\log M}} +\sqrt{K} \right] \lambda^{2L}(t) \int_0^t \frac{b_1}{\lambda^{2L+1}} \frac{b_1^{2(L+1)}}{|\log b_1|^2}. \label{eq:energy integral}  
\end{align}
To deal with the integral in \eqref{eq:energy integral}, one can directly replace $\lambda$ and $b_1$ with functions of $s$ using \eqref{eq:scaling law} and \eqref{eq:difference b_k}. However, the fact that $s_0$ in \eqref{eq:scaling law} depends on the bootstrap constant $K$ requires (more) care in direct substitution. On behalf of this approach, we integrate by parts using \eqref{eq:derivative scaling law}, \eqref{eq:initial scaling law control1} and the fact $c_1 \ge L/(L-1)$,
\begin{align}
  \int_0^t \frac{b_1}{\lambda^{2L+1}} \frac{b_1^{2(L+1)}}{|\log b_1|^2}  & = - \int_0^t \frac{\lambda_t}{\lambda^{2L+1}} \frac{b_1^{2(L+1)}}{|\log b_1|^2} + \int_0^t O\left( b_1^{L+2}\right)\frac{b_1^{2(L+1)}}{\lambda^{2L+1}|\log b_1|^2}  \nonumber \\
  & = \frac{1}{2L} \left[\frac{b_1^{2(L+1)}(t)}{\lambda^{2L}(t)|\log b_1(t)|^2}-\frac{b_1^{2(L+1)}(0)}{\lambda^{2L}(0)|\log b_1(0)|^2} \right] \nonumber \\
  &-\frac{1}{2L}\int_0 ^t \frac{1}{\lambda^{2L}} \left(  \frac{b_1^{2(L+1)}}{|\log b_1|^2} \right)_t  +\int_0^t O\left( b_1^{L+2}\right)\frac{b_1^{2(L+1)}}{\lambda^{2L+1}|\log b_1|^2} \nonumber \\
  &\le \frac{b_1^{2(L+1)}(t)}{\lambda^{2L}(t)|\log b_1(t)|^2} + \int_0^t \frac{b_1}{\lambda^{2L+1}} \left( \frac{L^2-1}{L^2} + \frac{C}{|\log b_1|} \right) \frac{b_1^{2(L+1)}}{|\log b_1|^2},\nonumber
\end{align}
we obtain the bound
\begin{equation*}
  \int_0^t \frac{b_1}{\lambda^{2L+1}} \frac{b_1^{2(L+1)}}{|\log b_1|^2}  \lesssim  \frac{b_1^{2(L+1)}(t)}{\lambda^{2L}(t)|\log b_1(t)|^2}
\end{equation*}
and therefore,
\begin{equation}
  \mathcal{E}_{L+1}(t) \lesssim  \left[ 1+ \frac{K}{\sqrt{\log M}} +\sqrt{K} \right] \frac{b_1^{2(L+1)}(t)}{|\log b_1(t)|^2} \le \frac{K}{2} \frac{b_1^{2(L+1)}(t)}{|\log b_1(t)|^2} .
\end{equation}
  \textbf{Step 3:} \emph{Improved bound on} $\mathcal{E}_{k}$. We now claim the improved bound on the intermediate energies: for $2\le k \le L$,
  \begin{equation}\label{eq:improved intermediate energy}
    \mathcal{E}_{k} \le  b_1^{2(k-1)c_1} |\log b_1|^{C+K/2}.
  \end{equation}
  This follows from the monotonicity formula for $2\le k \le L$,
  \begin{align}\label{eq:intermediate energy mono}
    \frac{d}{dt}  \left\{  \frac{\mathcal{E}_{k}  }{\lambda^{2k-2}}   \right\} \le C \frac{b_1|\log b_1|^C}{\lambda^{2k-1}}  (\sqrt{\mathcal{E}_{k+1}}+b_1^{k}+b_1^{\delta(k)+(k-1)c_1})\sqrt{\mathcal{E}_k}
  \end{align}
  for some universal constants $C$, $\delta>0$ independent of the bootstrap constant $K$. \eqref{eq:improved intermediate energy} will be proved in Appendix \ref{sec:intermediate monotonicity}. We integrate the above monotonicity formula ($K/2$ comes from $\sqrt{\mathcal{E}_{k}}$),
  \begin{equation}\label{eq:int intermediate monotonicity}
    \mathcal{E}_k \lesssim b_1^{2(k-1)c_1}|\log b_1|^{C+K/2} + \lambda^{2k-2}(t) \int_0^t \frac{b_1^{1+ 2(k-1)c_1}}{\lambda^{2k-1}} |\log b_1|^{C+K/2}
  \end{equation}
 In this case, we directly substitute $\lambda$ and $b_1$ with functions of $s$ since the possible large coefficient can be absorbed by $|\log b_1|^C$. From \eqref{eq:scaling law}, \eqref{eq:exact b formula} and \eqref{eq:difference b_k},
\begin{align}
  \lambda^{2k-2}(t) \int_0^t \frac{b_1^{1+ 2(k-1)c_1}}{\lambda^{2k-1}} |\log b_1|^{C+K/2} d\tau & = \lambda^{2k-2}(s) \int_{s_0}^s  \frac{b_1^{1+ 2(k-1)c_1}}{\lambda^{2k-2}} |\log b_1|^{C+K/2} d \sigma \nonumber \\
  &\lesssim \frac{(\log s)^{C+K/2}}{s^{2(k-1)c_1}} \int_{s_0}^s \frac{1}{\sigma} d\sigma \nonumber \\
  &\lesssim b_1^{2(k-1)c_1}|\log b_1|^{C + K/2}.\label{eq:int rough}
\end{align}
However, these improved bounds \eqref{eq:improved intermediate energy} are inadequate to close the bootstrap bounds when $\ell=L$ \eqref{eq:bootstrap epsilon L} and when $\ell=L-1$  \eqref{eq:bootstrap epsilon L-1} due to the logarithm factor. In these cases, we employ alternative energies defined by
\begin{equation}
  \widehat{\mathcal{E}}_{\ell}:=\langle \hat{\varepsilon}_{\ell}, \hat{\varepsilon}_{\ell} \rangle + \langle \dot{\hat{\varepsilon}}_{\ell-1}, \dot{\hat{\varepsilon}}_{\ell-1} \rangle.
\end{equation}
We can easily check that
\[ \widehat{\mathcal{E}}_{\ell} =  {\mathcal{E}}_{\ell} + O(b_1^{2\ell} |\log b_1|^2)\]
Then we have the following monotonicity formulae 
\begin{align}
  \frac{d}{dt}  \left\{  \frac{\widehat{\mathcal{E}}_{\ell}}{\lambda^{2\ell-2}} + O\left(\frac{b_1^{2\ell}|\log b_1|^2}{\lambda^{2\ell-2}}\right)  \right\}  &\le  \frac{b_1^{\ell+1}|\log b_1|^{\delta}}{\lambda^{2\ell-1}} ( b_1^{\ell}|\log b_1|+\sqrt{\mathcal{E}_{\ell}}  ). \label{eq:l energy mono}
\end{align}
Integrating \eqref{eq:l energy mono}, the initial bounds \eqref{eq:initial energies} and the bootstrap bounds \eqref{eq:bootstrap epsilon L}, \eqref{eq:bootstrap epsilon L-1} imply
\begin{align*}
  \frac{\widehat{\mathcal{E}}_{\ell}(t)}{\lambda^{2(\ell-1)}(t)} &\lesssim \frac{b_1^{2\ell}|\log b_1|^2(t)}{\lambda^{2\ell-2}(t)} + \frac{\widehat{\mathcal{E}}_{\ell}(0) + b_1^{2\ell}(0)|\log b_1(0)|^2}{\lambda^{2(\ell-1)}(0)} \\
  & + \int_0^t  \frac{b_1^{\ell+1}|\log b_1|^{\delta}}{\lambda^{2\ell-1}} ( b_1^{\ell}|\log b_1|+\sqrt{\mathcal{E}_{\ell}}  ) d\tau  \\
  &\lesssim 1 +   \int_0^t  \frac{b_1^{\ell+1}|\log b_1|^{\delta'}}{\lambda^{\ell}}  d\tau \lesssim 1 + \int_{s_0}^s  \frac{1}{\sigma (\log \sigma)^{\frac{\ell}{\ell-1}-\delta'}}  d\sigma \lesssim \frac{K}{2}.
\end{align*}
The monotonicity formulae \eqref{eq:intermediate energy mono}, \eqref{eq:l energy mono} are proved in Appendix \ref{sec:intermediate monotonicity}.
\begin{remark}
  We remark that the exponent $1+2(k-1)c_1$ of $b_1$ in \eqref{eq:int intermediate monotonicity} can be replaced by $1+\delta+2(k-1)c_1$ for some small $\delta>0$ when $2\le k \le \ell-1$, so we can improve the bound \eqref{eq:int rough} to $b_1^{2(k-1)c_1 + \delta}|\log b_1|^C$. Hence for $2\le k \le \ell$, we get the uniform bounds
\begin{equation}\label{eq:scale invariance bound}
  \mathcal{E}_{k} \lesssim \lambda^{2k-2}.
\end{equation}
\end{remark}
  \textbf{Step 4:} \emph{Control of stable/unstable parameters}. We use the modified modulation parameters $\tilde{b}=({b}_1,\dots,{b}_{L-1},\tilde{b}_L)$ with $\tilde{b}_L$ given by \eqref{def:tilde b_L} and the corresponding fluctuation $\widetilde{V}=P_{\ell} \widetilde{U}$ where $\widetilde{U}=(\widetilde{U}_1,\dots,\widetilde{U}_{\ell})$ is defined by
  \[ \frac{\widetilde{U}_k}{s^k (\log s)^{\beta}} =  \tilde{b}_k-b_k^e ,\quad  1\le k \le \ell.\]
   We note that the existence of $V(s_0)$ in Proposition \ref{prop:bootstrap} is equivalent to the existence of $\widetilde{V}(s_0)$ from remark \ref{rem:improved b_L} and \eqref{eq:difference tilde b_L} in view of  
   \begin{equation}\label{eq:difference tilde V}
     |V-\widetilde{V}| \lesssim s^L|\log s|^{\beta}|b_L-\tilde{b}_L|\lesssim s^L|\log s|^{\beta} b_1^{L+1-C\delta} \lesssim \frac{1}{s^{1/2}}.
   \end{equation}
   Hence, we can replace $\widetilde{V}$ for all $V$ of the initial assumptions \eqref{eq:initial stable modes}, \eqref{eq:initial unstable modes} and bootstrap bounds \eqref{eq:bootstrap modes stables}, \eqref{eq:bootstrap modes unstables} in subsection \ref{subsec:initial and bootstrap}. In particular, we replace the assumption \eqref{eq:finiteness assumption of exit time} as
   \begin{equation}
    \tilde{s}^* <\infty \quad \textnormal{for all }(V_2(s_0),\dots,V_{\ell}(s_0))\in \mathcal{B}^{{\ell}-1} .\label{eq:modified finiteness assumption of exit time}
   \end{equation}
   where $\tilde{s}^*$ denotes the modified exit time to indicate that $V$ has been changed to $\widetilde{V}$. 
   
   We start by closing the bootstrap bounds for the stable parameters $b_L$ (for the case $\ell=L-1$) and $\widetilde{V}_1$, then we rule out the assumption of the unstable parameters $(\widetilde{V}_2(s),\dots,\widetilde{V}_{\ell}(s))$ via showing a contradiction by Brouwer's fixed point theorem. 

  (i) \emph{Stable parameter }$b_L$\emph{ when} $\ell=L-1$: Recall Lemma \ref{lem:improved b_L}, we have 
  \begin{equation}\label{eq:improve b_L control}
    |(\tilde{b}_L)_s + (L-1 + c_{b,L})b_1\tilde{b}_L| \lesssim \frac{\sqrt{\mathcal{E}_{L+1}}}{\sqrt{|\log b_1|}}.
  \end{equation}
Note that $c_1=(L-1)/(L-2)$ and $b_1\sim c_1/s+d_1/(s\log s)$. Then from \eqref{eq:bootstrap modes stables} and \eqref{eq:improve b_L control},
\begin{align*}
  \frac{d}{ds} \left( s^{(L-1)c_1}(\log s)^{\frac{3}{2}} \tilde{b}_L \right)  &= s^{(L-1)c_1-1}(\log s)^{\frac{3}{2}}\left((L-1)c_1 + \frac{3/2}{\log s} \right)\tilde{b}_L  \\
  &-s^{(L-1)c_1}(\log s)^{\frac{3}{2}} \left( (L-1 + c_{b,L})b_1 \tilde{b}_L + O\left( \frac{\sqrt{\mathcal{E}_{L+1}}}{\sqrt{|\log b_1|}} \right) \right) \\
  &=s^{(L-1)c_1-1}(\log s)^{\frac{3}{2}} O\left( \frac{1}{s^L (\log s)^{1+\beta}} + \frac{1}{s^L (\log s)^{3/2}} \right) \\
  &=O\left( s^{(L-1)c_1-L-1} \right).
\end{align*}
We integrate the above equation and estimate using the initial condition \eqref{eq:initial stable modes}
\[|b_L(s)| \lesssim b_1^{L+1-C\delta} + \frac{s_0^{(L-1)c_1 }(\log s_0)^{3/2}|\tilde{b}_L(s_0)|}{s^{(L-1)c_1 }(\log s)^{3/2}} + \frac{1+(s_0/s)^{(L-1)c_1-L}}{s^L (\log s)^{3/2}} \le \frac{1/2}{s^L(\log s)^{\beta}}\]
with the fact $(L-1)c_1 > L$. Here, we choose $\beta=5/4$.

To control the modes $\widetilde{V}$, we rewrite \eqref{eq:linearization of b} for our $\tilde{b}$ as follows:
\begin{align}\label{eq:derivative of tilde U}
  s(\widetilde{U})_s- A_{\ell}\widetilde{U} = O \left( \frac{1}{(\log s)^{3/2-\beta}}  \right) 
\end{align}
using \eqref{eq:mod linearization of b}, Lemma \ref{lem:mod eqn} and Lemma \ref{lem:improved b_L}. Here, the reduced exponent $3/2$ comes from \eqref{eq:improve b_L control}.
%\begin{align}
 % &(\tilde{b}_k)_s+\left(k-1+\frac{1/(1+\delta_{1k})}{\log s}\right)\tilde{b}_1\tilde{b}_k-\tilde{b}_{k+1} \nonumber \\
 % &=\frac{1}{s^{k+1}(\log s)^{\beta}} \left[ s(\widetilde{U}_k)_s- (A_{\ell}\widetilde{U})_k + O \left( \frac{1}{(\log s)^{2-\beta}}  \right) \right],
%\end{align}
%and thus from ,
%\begin{align*}
 % &| s(\widetilde{U}_k)_s- (A_{\ell}\widetilde{U})_k|\\
 % &\lesssim \frac{1}{(\log s)^{2-\beta}} + s^{k+1}(\log s)^{\beta} \left\lvert (\tilde{b}_k)_s+\left(k-1+\frac{1/(1+\delta_{1k})}{\log s}\right)\tilde{b}_1\tilde{b}_k-\tilde{b}_{k+1} \right\rvert \\
 % & \lesssim \frac{1}{(\log s)^{2-\beta}} + s^{k+1}(\log s)^{\beta} \left\lvert \frac{1}{s^{k+1}(\log s)^2} + \frac{b_1^{L+1}}{|\log b_1|^{3/2}} \right\rvert \lesssim  \frac{1}{(\log s)^{3/2-\beta}}.
%\end{align*}
By the definition of $\widetilde{V}$, \eqref{eq:derivative of tilde U} is equivalent to
\begin{equation}\label{eq:derivative of tilde V}
  s(\widetilde{V})_s-D_{\ell} \widetilde{V} = O\left(  \frac{1}{(\log s)^{3/2-\beta}} \right)
\end{equation}
where $D_{\ell}$ is given by \eqref{eq:diagonalization}.

 (ii) \emph{Stable mode }$\widetilde{V}_1$: the first coordinate of \eqref{eq:derivative of tilde V} can be written as
\[s(\widetilde{V}_1)_s + \widetilde{V}_1=(s\widetilde{V}_1)_s =  O \left(\frac{1}{(\log s)^{3/2-\beta}}\right).\]
Hence, we improve the bound for $\widetilde{V}_1(s)$ from the initial assumption \eqref{eq:initial stable modes}:
\begin{equation*}
  |\widetilde{V}_1(s)| \lesssim \frac{s_0}{s}|\widetilde{V}_1(s_0)| + \frac{C}{s}\int_{s_0}^s \frac{d\tau}{(\log \tau)^{3/2-\beta}} \le \frac{1}{2}.
\end{equation*}
(iii) \emph{Unstable mode }$\widetilde{V}_k$, $2\le k \le \ell$: Our goal is to construct a continuous map $f:\mathcal{B}^{\ell-1} \to \mathcal{S}^{\ell-1}$ as 
\begin{align*}
  f(\widetilde{V}_2(s_0),\dots,\widetilde{V}_{\ell}(s_0)) = (\widetilde{V}_2(\tilde{s}^*),\dots,\widetilde{V}_{\ell}(\tilde{s}^*)).
\end{align*}
The assumption \eqref{eq:modified finiteness assumption of exit time} yields that $f$ can be well-defined on $\mathcal{B}^{\ell-1}$ and the improved bootstrap bounds give the exit condition $(\widetilde{V}_2(\tilde{s}^*),\dots,\widetilde{V}_{\ell}(\tilde{s}^*)) \in \mathcal{S}^{\ell-1}$. 

We obtain the outgoing behavior of the flow map $s\mapsto (\widetilde{V}_2,\dots,\widetilde{V}_{\ell})$ from \eqref{eq:derivative of tilde V}: for all time $s \in [s_0,\tilde{s}^*]$ such that $\sum_{i=2}^{\ell} \widetilde{V}_i^2 \ge 1/2$,
\begin{align}\label{eq:outgoing}
  \frac{d}{ds}\left(\sum_{i=2}^{\ell} \widetilde{V}_i^2 \right) &= 2 \sum_{i=2}^{\ell} (\widetilde{V}_i)_s  \widetilde{V}_i  = \frac{2}{s} \sum_{i=2}^{\ell} \left[ \frac{i}{\ell-1} \widetilde{V}_i^2 + O \left(  \frac{1}{(\log s)^{3/2-\beta}} \right) \right]   >0.
\end{align}
We note that \eqref{eq:outgoing} implies two key results. First, \eqref{eq:outgoing} allows us to prove the continuity of $f$ by showing the continuity of the map  $ (\widetilde{V}_2(s_0),\dots,\widetilde{V}_{\ell}(s_0)) \mapsto \tilde{s}^*$ with some standard arguments (see Lemma 6 in \cite{CoteMartelMerle2011RMI}).

Second, if we choose $s=s_0$ and $(\widetilde{V}_2(s_0),\dots,\widetilde{V}_{\ell}(s_0)) \in \mathcal{S}^{\ell-1}$, $\sum_{i=2}^{\ell} \widetilde{V}_i^2(s) >1$ for any $s>s_0$, so $\tilde{s}^*=s_0$. Hence, $f$ is an identity map on $\mathcal{S}^{\ell-1}$ itself, which contradicts to Brouwer's fixed point theorem. 
\end{proof}
\subsection{Proof of Theorem 1.1} Recall that there exists $c(u_0,\dot{u}_0)>0$ such that
\[\lambda(s)= \frac{c(u_0,\dot{u}_0)}{s^{c_1}(\log s)^{d_1}}\left[ 1+O\left( \frac{1}{(\log s_0)^{\beta-1}} \right)\right].\]
Using $T-t= \int_s^{\infty} \lambda(s) ds <\infty $, we have $T<\infty$ and
\begin{align*}
  (T-t)^{\ell-1}&=c'(u_0,\dot{u}_0) s^{-1} (\log s)^{\frac{\ell}{(\ell-1)}}\left[ 1+o_{t\to T}\left( 1\right)\right] \\
  &=c''(u_0,\dot{u}_0)\lambda(s)^{\frac{\ell-1}{\ell}}( \log s) \left[ 1+o_{t\to T}\left( 1\right)\right].
\end{align*}
Therefore, we obtain
\[\lambda(t)=c'''(u_0,\dot{u}_0)\frac{(T-t)^{\ell}}{|\log (T-t)|^{\ell/(\ell-1)}}[1+o_{t\to T}(1)].\]
The strong convergence \eqref{eq:strong convergence} follows as in \cite{RaphaelRodnianski2012IHES}. 
\begin{appendix}
    \section{Coercive properties}
    We recall that $\boldsymbol{\Phi}_M=(\Phi_M,0)^t$, the orthogonality conditions \eqref{eq:orthogonal conditions} are equivalent to
    \begin{equation}
      \langle \varepsilon, H^{i} \Phi_M \rangle =  \langle \dot{\varepsilon}, H^{i} \Phi_M \rangle = 0, \quad 0\le i \le \frac{L-1}{2}.
    \end{equation}
    In this section, we claim that the above equivalent orthogonality conditions yield the coercive property of the higher-order energy $\mathcal{E}_{k+1}$
    \begin{equation}
      \mathcal{E}_{k+1} = \langle \varepsilon_{k+1}, \varepsilon_{k+1} \rangle + \langle \dot{\varepsilon}_{k}, \dot{\varepsilon}_{k} \rangle,\quad 1\le k \le L.
    \end{equation}
    Our desired result is deduced from the coercivity of $ \{{\lVert v_m \rVert}_{L^2}^2 \}_{m=1}^{L+1}$
    under the following orthogonality conditions
    \begin{equation}\label{eq:orthogonality v}
      \langle v, H^i \Phi_M\rangle =0 ,\quad 0\le i \le \left\lfloor \frac{m-1}{2} \right\rfloor .
    \end{equation}
    First, we restate Lemma B.5 of \cite{RaphaelSchweyer2014Anal.PDE}, which established the coercivity of ${\lVert v_m \rVert}_{L^2}^2$ when $m$ is even.
    \begin{lemma}[coercivity of ${\lVert v_{2k+2} \rVert}_{L^2}^2$]  Let $0\le k \le \frac{L-1}{2}$ and $M=M(L)>0$ be a large constant. Then there exists $C(M)>0$ such that the following holds true. For all radially symmetric $v$ with (denote $v_{-1}=0$)
      \begin{align}
        &\int |v_{2k+2}|^2 + \int \frac{|v_{2k+1}|^2}{y^2(1+y^2)} \nonumber \\
        &+ \sum_{i=0}^{k} \int  \frac{|v_{2i-1}|^2}{y^6(1+|\log y|^2)(1+y^{4(k-i)})} + \frac{|v_{2i}|^2}{y^4(1+|\log y|^2)(1+y^{4(k-i)})}  < \infty \label{smooth assume 1}
      \end{align} 
      and \eqref{eq:orthogonality v} for $m=2k+2$, we have 
      \begin{align}
        &\int |v_{2k+2}|^2 \ge C(M) \bigg\{ \int \frac{|v_{2k+1}|^2}{y^2(1+|\log y|^2)} \nonumber \\
        & + \sum_{i=0}^{k} \int \left[  \frac{|v_{2i-1}|^2}{y^6(1+|\log y|^2)(1+y^{4(k-i)})} + \frac{|v_{2i}|^2}{y^4(1+|\log y|^2)(1+y^{4(k-i)})} \right] \bigg\}.
      \end{align}
    \end{lemma}
    We additionally prove the coercivity of ${\lVert v_m \rVert}_{L^2}^2$ when $m$ is odd, which is an unnecessary step in \cite{RaphaelSchweyer2014Anal.PDE}.
  \begin{lemma}[coercivity of ${\lVert v_{2k+1} \rVert}_{L^2}^2$]\label{lem:odd coercivity} Let $1\le k \le \frac{L-1}{2}$ and $M=M(L)>0$ be a large constant. Then there exists $C(M)>0$ such that the following holds true. For all radially symmetric $v$ with (denote $v_{-1}=0$)
    \begin{align}
      &\int |v_{2k+1}|^2 + \int \frac{|v_{2k}|^2}{y^2} + \int   \frac{|v_{2k-1}|^2}{y^4(1+|\log y|^2)} \nonumber\\
      &+ \sum_{i=0}^{k-1} \int  \frac{|v_{2i-1}|^2}{y^6(1+|\log y|^2)(1+y^{4(k-i)-2})} + \frac{|v_{2i}|^2}{y^4(1+|\log y|^2)(1+y^{4(k-i)-2})}  < \infty \label{smooth assume 2}
    \end{align} 
    and \eqref{eq:orthogonality v} for $m=2k+1$, we have
    \begin{align}
      &\int |v_{2k+1}|^2 \ge C(M) \bigg\{ \int \frac{|v_{2k}|^2}{y^2} +  \frac{|v_{2k-1}|^2}{y^4(1+|\log y|^2)} \nonumber \\
      & + \sum_{i=0}^{k-1} \int \left[  \frac{|v_{2i-1}|^2}{y^6(1+|\log y|^2)(1+y^{4(k-i)-2})} + \frac{|v_{2i}|^2}{y^4(1+|\log y|^2)(1+y^{4(k-i)-2})} \right] \bigg\}.
    \end{align}
  \end{lemma}
  \begin{remark}
    The case $k=0$ is nothing but the coercivity of $H$, described in Lemma B.1 of \cite{RaphaelSchweyer2014Anal.PDE}.
  \end{remark}
  Based on the induction on $k$ introduced in the proof of Lemma B.5 of \cite{RaphaelSchweyer2014Anal.PDE}, Lemma \ref{lem:odd coercivity} can be deduced from the following two lemmas, corresponding to the cases $k=1$ and $k \to k+1$. 
  \begin{lemma}[coercivity of ${\lVert v_{3} \rVert}_{L^2}^2$]\label{lem:3 coercivity} Let $M=M(L)>0$ be a large constant. Then there exists $C(M)>0$ such that the following holds true. For all radially symmetric $v$ with (denote $v_{-1}=0$)
    \begin{align}
      &\int |v_{3}|^2 + \int \frac{|v_{2}|^2}{y^2} + \int   \frac{|v_{1}|^2}{y^4(1+|\log y|^2)} +  \int   \frac{|v|^2}{y^4(1+|\log y|^2)(1+y^{2})}  < \infty  \nonumber
    \end{align} 
    and \eqref{eq:orthogonality v} for $m=3$, we have
    \begin{align}\label{eq:3 coercivity}
      \int |v_{3}|^2 \ge C(M) \bigg\{ \int \frac{|v_{2}|^2}{y^2} +  \frac{|v_{1}|^2}{y^4(1+|\log y|^2)}  +  \int  \frac{|v|^2}{y^4(1+|\log y|^2)(1+y^{2})} \bigg\}.
    \end{align}
  \end{lemma}

\begin{proof}
  From the coercivity of $H$, we have
  \begin{align}
    \int |v_{3}|^2 = \langle H v_2, v_2 \rangle \ge C(M) \int 
    \frac{|v_2|^2}{y^2}.
  \end{align}
  To prove the rest part of \eqref{eq:3 coercivity}, we claim the following weighted coercive bound
  \begin{equation}\label{eq:weighted coercive bound 0}
    \int \frac{|Hv|^2}{y^2(1+|\log y|^2)} 
      \ge C(M) \left\{ \int \frac{|v|^2}{y^4(1+|\log y|^2)(1+y^{2})} + \frac{|Av|^2}{y^4(1+|\log y|^2)} \right\}.
  \end{equation}
  By proving Lemma B.4 in \cite{RaphaelSchweyer2014Anal.PDE}, it is sufficient for \eqref{eq:weighted coercive bound 0} to prove only the following subcoercivity estimate:
  \begin{align}
    \int \frac{|Hv|^2}{y^2(1+|\log y|^2)} &\gtrsim \int \frac{|\partial_y^2 v|^2}{y^2(1+|\log y|^2)} + \int \frac{|\partial_y v|^2}{y^2(1+|\log y|^2)(1+y^2)} \nonumber  \\
    & +  \int \frac{|v|^2}{y^4(1+|\log y|^2)(1+y^{2})} - C \left[ \int \frac{|\partial_y v|^2}{1+y^{6}} +  \int \frac{|v|^2}{1+y^{8}}\right]. \label{eq:subcoercivity estimate 0}
  \end{align}
  Unlike the region $y \le 1$, which can be directly proved by borrowing the proof of Lemma B.4 in \cite{RaphaelSchweyer2014Anal.PDE}, we remark that \eqref{eq:subcoercivity estimate 0} required some cautious estimates in the region $y\ge 1$: we have
  \begin{align}
    \int_{y\ge 1} \frac{|Hv|^2}{y^2(1+|\log y|^2)}  & \ge \int_{y\ge 1} \frac{|\partial_y (y \partial_y v)|^2}{y^4(1+|\log y|^2)} - \int_{y\ge 1} |v|^2 \Delta \left( \frac{V}{y^4 (1+|\log y|^2)} \right)\nonumber \\
    &+ \int_{y\ge 1} \frac{V^2 |v|^2}{ y^{6} (1+|\log y|^2)} - C \int_{1\le y \le 2} [|\partial_y v|^2 + |v|^2]\label{eq:additional V}
  \end{align}
  where $V(y)=1-8y^2/(1+y^2)^2$ is the potential part of $H$. Using the sharp logarithmic Hardy inequality, employed in the proof of Lemma B.4 of \cite{RaphaelSchweyer2014Anal.PDE}, we obtain
  \begin{align*}
    \int_{y\ge 1} \frac{|\partial_y (y \partial_y v)|^2}{y^{4}(1+|\log y|^2)} - \int_{y\ge 1} |v|^2 \Delta \left( \frac{1}{y^{4} (1+|\log y|^2)} \right) \ge - C \int_{1\le y \le 2} [|\partial_y v|^2 + |v|^2].
  \end{align*}
  Now we employ the additional positive term in \eqref{eq:additional V} with the asymptotics of the potential $V(y)=1+O(y^{-2})$ for $y\ge 1$,
  \[\int_{y\ge 1} \frac{V^2 |v|^2}{y^{6}(1+|\log y|^2)} \ge 1^- \int_{y\ge 1} \frac{|v|^2}{y^{6}(1+|\log y|^2)} - C \int \frac{|v|^2}{1+y^{8}}. \qedhere\] 
\end{proof}

\begin{lemma}[weighted coercivity bound]
  For $k\ge 1$ and radially symmetric $v$ with 
  \begin{equation}
    \int \frac{|v|^2}{y^4(1+|\log y|^2)(1+y^{4k+2})} + \frac{|Av|^2}{y^6(1+|\log y|^2)(1+y^{4k-2})}  < \infty \label{smooth assume 3}
  \end{equation}
  and
  \[\langle v, \Phi_M \rangle=0,\]
  we have
  \begin{align}\label{eq:weighted coercive bound}
    &\int \frac{|Hv|^2}{y^4(1+|\log y|^2)(1+y^{4k-2})} \nonumber \\
    &\ge C(M) \left\{ \int \frac{|v|^2}{y^4(1+|\log y|^2)(1+y^{4k+2})} + \frac{|Av|^2}{y^6(1+|\log y|^2)(1+y^{4k-2})} \right\}.
  \end{align}

\end{lemma}
\begin{proof}
  We can prove \eqref{eq:weighted coercive bound} easily by replacing all $4k$ in the proof of Lemma B.4 of \cite{RaphaelSchweyer2014Anal.PDE} to $4k-2$, since the range of our $k$ is $k \ge 1$. 
\end{proof}
From the previous lemmas, we obtain the coercivity of $\mathcal{E}_{k+1}$.
\begin{lemma}[Coercivity of $\mathcal{E}_{k+1}$]
  Let $1\le k \le L$ and $M=M(L)>0$ be a large constant. Then there exists $C(M)>0$ such that 
  \begin{align}
    \mathcal{E}_{k+1}=&\langle \varepsilon_{k+1} , \varepsilon_{k+1}\rangle + \langle \dot{\varepsilon}_{k} , \dot{\varepsilon}_{k}\rangle \nonumber  \\
    \ge &{C(M)} \Bigg[  \sum_{i=0}^{k} \int \frac{|\varepsilon_i|^2}{y^2(1+y^{2(k-i)})(1+|\log y|^2)} \nonumber \\
     &\qquad \; + \sum_{i=0}^{k-1} \int \frac{|\dot{\varepsilon}_i|^2}{y^2(1+y^{2(k-1-i)})(1+|\log y|^2)} \Bigg].\label{eq:coercivity bound}
  \end{align}
\end{lemma}
\begin{remark}
  The finiteness assumptions \eqref{smooth assume 1}, \eqref{smooth assume 2} and \eqref{smooth assume 3} for \eqref{eq:coercivity bound} are satisfied from the well-localized smoothness of $1$-corotational map $(\Phi, \partial_t \Phi)$ (see Lemma A.1 in \cite{RaphaelSchweyer2014Anal.PDE}). 
\end{remark}
\section{Interpolation estimates}
 In this section, we provide some interpolation estimates for $\varepsilon$, i.e. the first coordinate part of $\boldsymbol{\varepsilon}$. We will employ these bounds to deal with $\boldsymbol{NL}(\boldsymbol{\varepsilon})$ and $\boldsymbol{L}(\boldsymbol{\varepsilon})$ terms in the evolution equation of $\boldsymbol{\varepsilon}$ \eqref{eq:evolution epsilon}.
\begin{lemma}[interpolation estimates]\label{lem:interpolation bound}
  (ii) For $y\le 1$, $\varepsilon$ has a Taylor-Lagrange expansion
\begin{equation}
  \varepsilon=\sum_{i=1}^{\frac{L+1}{2}} c_i T_{L+1-2i} + r_{\varepsilon} 
\end{equation}
where $T_{2i}$ is the first coordinate part of $\boldsymbol{T}_{2i}$ and  
\begin{align}
  |c_i| \lesssim C(M) \sqrt{\mathcal{E}_{L+1}},\quad |\partial_y^k r_{\varepsilon}| \lesssim C(M) y^{L-k}|\log y| \sqrt{\mathcal{E}_{L+1}}, \quad 0\le k \le L.
\end{align}
(iii) For $y\le 1$, $\varepsilon$ satisfies the following pointwise bounds
\begin{align}
  |\varepsilon_{k}| &\lesssim C(M) y^{1+\overline{k}}|\log y| \sqrt{\mathcal{E}_{L+1}},\quad 0\le k \le {L-1}, \\
  |\varepsilon_{L}| &\lesssim C(M) \sqrt{\mathcal{E}_{L+1}},\\
  |\partial_y^{k}\varepsilon| &\lesssim C(M) y^{\overline{k+1}}|\log y| \sqrt{\mathcal{E}_{L+1}},\quad 0\le k \le L. 
\end{align}
(iv) For $1\le k \le L$ and $0\le i \le k$,
\begin{align}
   \int \frac{1+|\log y|^C}{1+y^{2(k-i+1)}}(|\varepsilon_i|^2 + |\partial_y^i \varepsilon |^2) +  {\left\lVert \frac{\partial_y^i \varepsilon}{y^{k-i}} \right\rVert}^2_{L^{\infty}(y\ge 1)} &\lesssim |\log b_1|^C b_1^{2m_{k+1}}
\end{align}
where 
\begin{equation*}
  m_{k+1} = \begin{cases}
    kc_1 & \textnormal{if}\quad 1\le k \le L-2, \\
    L & \textnormal{if} \quad k=L-1, \\
    L+1 & \textnormal{if}\quad  k =L.
  \end{cases}
\end{equation*}
\end{lemma}
\begin{proof}
  It is provided from the proof of Lemma C.1 in \cite{RaphaelSchweyer2014Anal.PDE}.
\end{proof}

\section{Leibniz rule for $\mathcal{A}^k$}
Unlike \cite{RaphaelSchweyer2014Anal.PDE}, we encounter some terms in which $\partial_t$ is taken more than once to $\mathcal{A}_{\lambda}^k$, such as $\partial_{tt}(\mathcal{A}_{\lambda}^k)$, $\partial_t(\mathcal{A}_{\lambda}^{i})\partial_t(H_{\lambda}^j)$, etc. To control those terms, we recall the following asymptotics
\begin{equation}\label{eq:1-t}
  \partial_t(\mathcal{A}_{\lambda}^k)f_{\lambda}(r)= \frac{\lambda_t}{\lambda^{k+1}}  \sum_{i=0}^{k-1}  \Phi_{i,k}^{(1)}(y) f_i (y),\quad |\Phi_{i,k}^{(1)}(y)| \lesssim \frac{1}{1+y^{k+2-i}},
\end{equation}
which was introduced in Appendices D and E of \cite{RaphaelSchweyer2014Anal.PDE}. We note that near the origin, $\Phi_{i,k}^{(1)}$ satisfies 
\begin{equation}\label{eq:coeff near origin}
  \Phi_{i,k}^{(1)}(y) = \begin{cases}
    \sum_{p=0}^{N} c_{i,k,p} y^{2p} + O(y^{2N+2}) & k-i \textnormal{ is even} \\
    \sum_{p=0}^{N} c_{i,k,p} y^{2p+1} + O(y^{2N+3}) & k-i \textnormal{ is odd}.
  \end{cases}
\end{equation} 
Based on the above facts, we can obtain the following lemma.
\begin{lemma}\label{lem:Leibniz rule}
  Let $1\le k\le (L-1)/2$. Then
  \begin{align}
    \partial_{tt}(\mathcal{A}_{\lambda}^k)f_{\lambda}(r)&=\frac{\lambda_{tt}}{\lambda^{k+1}}  \sum_{i=0}^{k-1}  \Phi_{i,k}^{(1)}(y) f_i (y)+ \frac{O(b_1^2)}{\lambda^{k+2}}  \sum_{i=0}^{k-1}  \Phi_{i,k}^{(2)}(y) f_i (y), \label{eq:2-t}\\
    \partial_t (\mathcal{A}_{\lambda}^{L-2k})\partial_t ({H}_{\lambda}^{k})f_{\lambda}(r) &= \frac{O(b_1^2)}{\lambda^{L+2}} \sum_{i=0}^{L-1} \Phi_{i,L}^{(3)}(y) f_i(y) \label{eq:1,1-t}
  \end{align}
  where
  \[ |\Phi_{i,k}^{(2)}(y)| \lesssim \frac{1}{1+y^{k+2-i}},\quad  |\Phi_{i,L}^{(3)}(y)| \lesssim \frac{1}{1+y^{L+3-i}}. \]
\end{lemma}
\begin{proof}
  Recall $\partial_{tt}(\mathcal{A}_{\lambda}^k)f_{\lambda} = [\partial_t, \partial_t(\mathcal{A}_{\lambda}^k)]f_{\lambda}$ and
  \[\frac{\lambda_t}{\lambda^{k+1}} \Phi_{i,k}^{(1)}(y) f_i(y) = \frac{\lambda_t}{\lambda^{k+1-i}} (\Phi_{i,k}^{(1)})_{\lambda}(r) \mathcal{A}_{\lambda}^i f_{\lambda}(r),\quad  \partial_t \Phi_{\lambda} = -\frac{\lambda_t}{\lambda} (\Lambda \Phi)_{\lambda}, \]
  we get \eqref{eq:2-t} since
  \begin{align*}
    [\partial_t, \frac{\lambda_t}{\lambda^{k+1-i}} (\Phi_{i,k}^{(1)})_{\lambda} \mathcal{A}_{\lambda}^i ]f_{\lambda} &= \frac{\lambda_{tt}}{\lambda^{k+1-i}} (\Phi_{i,k}^{(1)})_{\lambda} \mathcal{A}_{\lambda}^i f_{\lambda}  \\
    &-\frac{(\lambda_t)^2}{\lambda^{k+2-i}} (\Lambda_{i-k} \Phi_{i,k}^{(1)})_{\lambda} \mathcal{A}_{\lambda}^i f_{\lambda} + \frac{\lambda_t}{\lambda^{k+1-i}} (\Phi_{i,k}^{(1)})_{\lambda} \partial_t(\mathcal{A}_{\lambda}^i )f_{\lambda} \\
    &= \frac{\lambda_{tt}}{\lambda^{k+1}}   \Phi_{i,k}^{(1)}(y) f_i (y) +\frac{O(b_1^2)}{\lambda^{k+2}} \sum_{j=0}^{i} \Phi_{i,j,k}(y) f_j(y)
  \end{align*}
  where
  \[|\Phi_{i,j,k}(y)| \lesssim \frac{1}{1+y^{k+2-j}}.\]
Moreover, we can easily check that $\Phi_{i,k}^{(2)}$ satisfies \eqref{eq:coeff near origin} because the scaling generator $\Lambda$ preserves the asymptotics near origin as well as infinity.   

To prove \eqref{eq:1,1-t}, we need to justify the terms of the form $\mathcal{A}^i \circ \Phi \mathcal{A}^j$. When $j$ is an even number, we can use the Leibniz rule from the Appendix D of \cite{RaphaelSchweyer2014Anal.PDE}. However, when $j$ is odd, terms such as $A\circ \Phi A$ appear, making the problem a bit more tricky. 

Fortunately, our $\Phi$ from the terms of the form $\mathcal{A}^i \circ \Phi\mathcal{A}^{2j+1}$ have an expansion
\[\Phi(y)=\sum_{p=0}^N c_{p} y^{2p+1} + O(y^{2N+3})\]
near the origin since each $\Phi \mathcal{A}^{2j+1}$ comes from $\partial_t(H_{\lambda}^k)$ or $\partial_{tt}(H_{\lambda}^k)$, satisfies \eqref{eq:coeff near origin}. Hence
\begin{align*}
  (A \circ \Phi \mathcal{A}^{2j+1})f &=  (A\Phi) f_{2j+1} - \Phi \partial_y f_{2j+1} \\
  &=  \left(-\partial_y + \frac{1+2Z}{y}\right)\Phi \cdot f_{2j+1} - \Phi  f_{2j+2} =: \Phi_1 f_{2j+1} - \Phi f_{2j+2}
\end{align*}  
where $\Phi_1$ satisfies 
\[\Phi_1(y)=\sum_{p=0}^N c_{p} y^{2p} + O(y^{2N+2})\]
near the origin. If we take $A^*$ here,
\begin{align*}
  (H\circ \Phi \mathcal{A}^{2j+1} )f & = A^* (\Phi_1 f_{2j+1} - \Phi f_{2j+2}) \\
  &= (\partial_y \Phi_1) f_{2j+1} + (\Phi_1 -A^* \Phi)f_{2j+2} - \Phi \partial_y f_{2j+2} \\
  &=(\partial_y \Phi_1) f_{2j+1} + \left(\Phi_1 -\partial_y \Phi - \frac{1+2Z}{y}\Phi\right) f_{2j+2} + \Phi f_{2j+3},
\end{align*}
we can justify $\mathcal{A}^i \circ \Phi \mathcal{A}^{2j+1}$ by iterating above calculation.
\end{proof}
\section{Monotonicity for the intermediate energy}\label{sec:intermediate monotonicity}
\begin{proposition}[Lyapunov monotonicity for $\mathcal{E}_k$]\label{prop:intermediate monotonicity} 
  Let $2\le k \le L$. We have
  \begin{align}\label{eq:intermediate monotonicity}
    \frac{d}{dt}  \left\{  \frac{\mathcal{E}_{k}}{\lambda^{2k-2}}  \right\}  \le  \frac{b_1|\log b_1|^{C(k)}}{\lambda^{2k-1}}      (\sqrt{\mathcal{E}_{k+1}}+b_1^{k}+b_1^{\delta(k)+(k-1)c_1})\sqrt{\mathcal{E}_k} 
  \end{align}
  where $C(k),\delta(k)>0$ are constants that depend only on $k,L$.
\end{proposition}
\begin{proof}
We compute the energy identity:
\begin{align}
  \partial_t \left(\frac{\mathcal{E}_{k}}{2\lambda^{2(k-1)}} \right)&=  \langle  \partial_tw_{k}, w_{k} \rangle + \langle  \partial_t\dot{w}_{k-1},\dot{w}_{k-1} \rangle \nonumber \\
  &=   \langle \partial_t (\mathcal{A}^{k}_{\lambda}) w, w_k \rangle + \langle \partial_t (\mathcal{A}^{k-1}_{\lambda}) \dot{w} ,\dot{w}_{k-1}\rangle \label{eq:intermediate main}\\
  &+  \langle  \mathcal{A}^{k}_{\lambda} \mathcal{F}_1, w_k \rangle + \langle \mathcal{A}^{k-1}_{\lambda} \mathcal{F}_2,\dot{w}_{k-1}\rangle. \label{eq:intermediate F}
\end{align}
We can directly estimate \eqref{eq:intermediate main} by Lemma \ref{lem:Leibniz rule}
\begin{align}
  |\langle  \partial_{t}(\mathcal{A}_{\lambda}^{k})w , w_{k} \rangle| & \lesssim  \frac{b_1}{\lambda^{2k-1}}\sum_{m=0}^{k-1} |\langle \Phi_{m,k}^{(1)} \varepsilon_m, \varepsilon_k \rangle | \nonumber \\
  &\lesssim \frac{b_1}{\lambda^{2k-1}}\sum_{m=0}^{k-1}  {\left\lVert \frac{\varepsilon_m}{1+y^{k+2-m}}  \right\rVert}_{L^2} \sqrt{\mathcal{E}_{k}}\lesssim \frac{b_1 C(M)}{\lambda^{2k-1}} \sqrt{\mathcal{E}_{k+1}\mathcal{E}_{k}}, \label{eq:intermediate ww}\\
  |\langle  \partial_{t}(\mathcal{A}_{\lambda}^{k-1})\dot{w} , \dot{w}_{k-1} \rangle| & \lesssim  \frac{b_1}{\lambda^{2k-1}}\sum_{m=0}^{k-2} |\langle \Phi_{m,k-1}^{(1)} \dot{\varepsilon}_m, \dot{\varepsilon}_{k-1} \rangle | \lesssim \frac{b_1 C(M)}{\lambda^{2k-1}} \sqrt{\mathcal{E}_{k+1}\mathcal{E}_{k}}.\label{eq:intermediate dot ww}
\end{align}
Then we conclude \eqref{eq:intermediate monotonicity} from the following bounds:
\begin{align}\label{eq:intermediate F bound}
  {\left\lVert \mathcal{A}^{k} \mathcal{F} \right\rVert}_{L^2} +  {\left\lVert \mathcal{A}^{k-1} \dot{\mathcal{F}} \right\rVert}_{L^2} & \lesssim b_1 |\log b_1|^C \left[ b_1^k + b_1^{\delta(k) + (k-1)c_1} \right],
\end{align}
\eqref{eq:intermediate F} is bounded by
\[\frac{b_1|\log b_1|^C}{\lambda^{2k-1}}(b_1^{k}+b_1^{\delta(k)+(k-1)c_1})\sqrt{\mathcal{E}_k}.\]
Now, it remains to prove \eqref{eq:intermediate F bound} and we address it by separating $\boldsymbol{\mathcal{F}}=(\mathcal{F},\dot{\mathcal{F}})^t$ into four types, as we did for \textbf{Step 5} in the proof of Proposition \ref{prop:monotonicity}. 

(i) $\tilde{\boldsymbol{\psi}}_b$ \emph{terms}. The contribution of $\tilde{\boldsymbol{\psi}}_b$ terms to the above inequalities is estimated from the global weighted bounds of Proposition \ref{prop:local approx}.

(ii) ${\widetilde{\mathbf{Mod}}}(t)$ \emph{terms}. Similar to (ii) of \textbf{Step 5} in the proof of Proposition \ref{prop:monotonicity} with the cancellation $\mathcal{A}^{k}T_i=0$ for $1\le i \le k$ and Lemma \ref{lem:admissible bounds}, we obtain 
\begin{align*}
  \int  \left\lvert  \sum_{i=1}^L b_i\mathcal{A}^{k-\overline{i}}\left[ \Lambda_{1-\overline{i}}(\chi_{B_1}T_i)\right] +  \sum_{i=2}^{L+2}\mathcal{A}^{k-\overline{i}}\left[ \Lambda_{1-\overline{i}}(\chi_{B_1}S_i) \right] \right\rvert^2 \lesssim b_1^2\\
  \sum_{i=1}^L\int  \left\lvert \mathcal{A}^{k-\overline{i}} \left[ \chi_{B_1}T_i+ \chi_{B_1}\sum_{j=i+1}^{L+2} \frac{\partial {S}_j}{\partial b_i} \right] \right\rvert^2 \lesssim b_1^{2(k-L)}|\log b_1|^{2\gamma(L-k)+2}
\end{align*}
Hence, Lemma \ref{lem:mod eqn} and the bootstrap bound \eqref{eq:bootstrap epsilon} implies:
\begin{align*}
  {\left\lVert \mathcal{A}^k{\widetilde{\mathrm{Mod}}}(t) \right\rVert}_{L^2}+{\left\lVert \mathcal{A}^{k-1} \dot{\widetilde{\mathrm{Mod}}}(t) \right\rVert}_{L^2}&\lesssim  b_1^{k-L}|\log b_1|^{\gamma(L-k)+1} \frac{b_1^{L+1}}{|\log b_1|} \\
  &\lesssim  b_1^{k+1}|\log b_1|^{\gamma(L-k)}.
\end{align*}

(iii) $\boldsymbol{NL}(\boldsymbol{\varepsilon})$ \emph{term}: We can utilize the bound \eqref{eq:NL bound near zero} near origin. For $y\ge 1$, we recall the calculation and estimates from (iii) of \textbf{Step 5} in the proof of Proposition \ref{prop:monotonicity}, ${\left\lVert \mathcal{A}^{k-1}NL(\varepsilon) \right\rVert}_{L^2(y\ge 1)}$ is bounded by
\[|\log b_1|^C  b_1^{m_{I+1}}b_1^{m_{J+1}} + |\log b_1|^C b_1^{m_{X+1}}b_1^{m_{Y+1}}b_1^{m_{J+1}} \]
where $I,J,X,Y,Z \ge 1$, $I+J=k$ and $X+Y+Z=k$. From the bootstrap bounds \eqref{eq:bootstrap epsilon}, \eqref{eq:bootstrap epsilon L-1} and the fact that $c_1 >1$, we obtain
\[{\left\lVert \mathcal{A}^{k-1}NL(\varepsilon) \right\rVert}_{L^2(y\ge 1)} \lesssim |\log b_1|^{C(K)} b_1^{k c_1} \lesssim b_1^{1+\delta(k)+(k-1)c_1}.\]

(iv) $\boldsymbol{L}(\boldsymbol{\varepsilon})$ \emph{term}: With some modifications (replace $L$ to $k-1$, for instance), it is proved by \eqref{eq:L bound near zero} and \eqref{eq:L bound near infinity}.
\end{proof}
\begin{remark}
  In step (iii) when $k=L$, we can avoid the case that either $I=L-1$ or $J=L-1$ by estimating ${\left\lVert \partial_y^{L-1}N_1(\varepsilon) \right\rVert}_{L^2(y\ge 1)}$ instead of ${\left\lVert \partial_y^{L-1}N_1(\varepsilon) \right\rVert}_{L^{\infty}(y\ge 1)}$.
\end{remark}
Recall the modified higher order energies
\begin{equation*}
  \widehat{\mathcal{E}}_{\ell}:=\langle \hat{\varepsilon}_{\ell}, \hat{\varepsilon}_{\ell} \rangle + \langle \dot{\hat{\varepsilon}}_{\ell-1}, \dot{\hat{\varepsilon}}_{\ell-1} \rangle.
\end{equation*}
We rewrite the flow \eqref{eq:evolution hat W} component-wisely: for $1\le k \le \ell$,
\begin{equation}\label{eq:evolution hat w_k}
  \begin{cases}
    \partial_t \hat{w}_{k} -\dot{\hat{w}}_{k} = \partial_t(\mathcal{A}_{\lambda}^{k})\hat{w} +\mathcal{A}_{\lambda}^{k}\widehat{\mathcal{F}}_1 \\
    \partial_t\dot{\hat{w}}_{k} + \hat{w}_{k+2} = \partial_t(\mathcal{A}_{\lambda}^{k})\dot{\hat{w}} +\mathcal{A}_{\lambda}^{k}\widehat{\mathcal{F}}_2 
  \end{cases} , \quad \begin{pmatrix}
    \widehat{\mathcal{F}}_1  \\
    \widehat{\mathcal{F}}_2
  \end{pmatrix} := \frac{1}{\lambda} \widehat{\boldsymbol{\mathcal{F}}}_{\lambda}=\frac{1}{\lambda}\begin{pmatrix}
    \widehat{\mathcal{F}} \\
    \dot{\widehat{\mathcal{F}}}
  \end{pmatrix}_{\lambda} .
\end{equation}
\begin{proposition}[Lyapunov monotonicity for $\mathcal{E}_L$]\label{prop:L monotonicity} 
  Let $\ell=L$. Then we have
  \begin{align}\label{eq:L monotonicity}
    \frac{d}{dt}  \left\{  \frac{\widehat{\mathcal{E}}_{L}}{\lambda^{2L-2}} + O\left(\frac{b_1^{2L}|\log b_1|^2}{\lambda^{2L-2}}\right)  \right\}  \le  \frac{b_1^{L+1}|\log b_1|^{\delta}}{\lambda^{2L-1}} ( b_1^L|\log b_1|+\sqrt{\mathcal{E}_L}  ) 
  \end{align}
  where $0<\delta \ll 1$ is a sufficient small constant that depend only on $L$.
\end{proposition}
\begin{proof}
We compute the energy identity:
\begin{align}
  \partial_t \left(\frac{\widehat{\mathcal{E}_{L}}}{2\lambda^{2(L-1)}} \right)&=   \langle \partial_t (\mathcal{A}^{L}_{\lambda}) \hat{w}, \hat{w}_L \rangle + \langle \partial_t (\mathcal{A}^{L-1}_{\lambda}) \dot{\hat{w}} ,\dot{\hat{w}}_{L-1}\rangle \label{eq:L main}\\
  &+  \langle  \mathcal{A}^{L}_{\lambda} \widehat{\mathcal{F}}_1, \hat{w}_L \rangle + \langle \mathcal{A}^{L-1}_{\lambda}  \widehat{\mathcal{F}}_2,\dot{\hat{w}}_{L-1}\rangle. \label{eq:L F}
\end{align}
We can directly estimate \eqref{eq:L main} from the bounds \eqref{eq:intermediate ww}, \eqref{eq:intermediate dot ww} and the fact $\boldsymbol{\varepsilon}-\hat{\boldsymbol{\varepsilon}}=\boldsymbol{\zeta}_b$, we obtain the bound
\begin{align*}
 | \eqref{eq:L main} | \lesssim \frac{b_1 C(M)}{\lambda^{2L-1}} \sqrt{\mathcal{E}_{L+1}\mathcal{E}_{L}} +  \frac{b_1^{L+3}|\log b_1|^C}{\lambda^{2L-1}} \sqrt{\mathcal{E}_{L}} + \frac{b_1^{2L+3}|\log b_1|^C}{\lambda^{2L-1}}. 
\end{align*}
We can borrow step (ii), (iii) and (iv) in the proof of Proposition \ref{prop:intermediate monotonicity} to estimate \eqref{eq:L F} except $\hat{\boldsymbol{\psi}}_b$ terms. Also by Proposition \ref{prop:2nd local approx}, all the inner products we have to deal with are:
\begin{equation}
  b_L\langle  \mathcal{A}^{L} (\chi_{B_1}-\chi_{B_0})T_{L-1} , \hat{\varepsilon}_L \rangle ,\quad b_L\langle  \mathcal{A}^{L-1} (\partial_s\chi_{B_0}+b_1(y\chi')_{B_0})T_{L} , \dot{\hat{\varepsilon}}_{L-1} \rangle .
\end{equation}
From the fact $\hat{\varepsilon} = \varepsilon$ and $\mathcal{A}^{L-1} T_{L-1}= (-1)^{\frac{L-1}{2}}\Lambda Q$, we obtain
\begin{align*}
  \mathcal{A}^{L-1} (\chi_{B_1}-\chi_{B_0})T_{L-1} = (-1)^{\frac{L-1}{2}}(\chi_{B_1}-\chi_{B_0})\Lambda Q + (\mathbf{1}_{y\sim B_1}+ \mathbf{1}_{y\sim B_0}) O(y^{-1}|\log y|).
\end{align*}
Hence, the bootstrap bound \eqref{eq:bootstrap epsilon} yields
\begin{align*}
  |\langle  \mathcal{A}^{L} (\chi_{B_1}-\chi_{B_0})T_{L-1} , \hat{\varepsilon}_L \rangle| & = |\langle  \mathcal{A}^{L-1} (\chi_{B_1}-\chi_{B_0})T_{L-1} , \hat{\varepsilon}_{L+1} \rangle| \\
  &\le |\langle  y^{-1}\mathbf{1}_{B_0 \le y \le 2B_1} + (\mathbf{1}_{y\sim B_1}+ \mathbf{1}_{y\sim B_0}) y^{-1}|\log y| ,\varepsilon_{L+1}\rangle |\\
  &\le (|\log b_1|^{1/2}+|\log b_1|)\sqrt{\mathcal{E}_{L+1}} \le b_1^{L+1} |\log b_1|^{\delta}.
\end{align*}
Note that $\dot{\hat{\varepsilon}}= \dot{\varepsilon} + b_L (\chi_{B_1}-\chi_{B_0})T_L$. The asymptotics \eqref{eq:derivative chi} implies
\begin{align*}
  |\langle  \mathcal{A}^{L-1} (\partial_s\chi_{B_0}+b_1(y\chi')_{B_0})T_{L} , \dot{{\varepsilon}}_{L-1} \rangle| & \le  b_1|\langle  \mathcal{A}^{L-2} (\mathbf{1}_{y\sim B_0}y^{L-2}|\log y| ), \dot{{\varepsilon}}_{L} \rangle|\\
  &\le |\log b_1|\sqrt{\mathcal{E}_{L+1}}\le b_1^{L+1} |\log b_1|^{\delta}.
\end{align*}
To estimate the last inner product, we employ the sharp asymptotics
\[b_1 (y\chi')_{B_0}= -c_1 \partial_s\chi_{B_0} + O\left( \frac{b_1 \mathbf{1}_{y\sim B_0}}{|\log b_1|}\right)\]
from the fact $(b_1)_s =b_2 + O(b_1^2/|\log b_1|)$. Using the cancellation $\mathcal{A}^{L} T_L =0$ and $\chi_{B_1}=1$ on $y \sim B_0$, the remaining inner product can be written as
\begin{equation}\label{eq:L F final} 
  \frac{1}{L-1} b_L^2 \langle  \mathcal{A}^{L-1} \partial_s(\chi_{B_0}T_{L}) ,  \mathcal{A}^{L-1}(\chi_{B_0}T_L) \rangle + O \left( \frac{b_1^{2L+1}}{|\log b_1|} {\left\lVert\mathcal{A}^{L-1} (\mathbf{1}_{y\sim B_0}T_{L}) \right\rVert}_{L^2}^2  \right). 
\end{equation}
We can easily check that the second term in \eqref{eq:L F final} is bounded by $b_1^{2L+1}|\log b_1|$. For the first term in \eqref{eq:L F final}, we use integration by parts in time to find out the correction for $\widehat{\mathcal{E}}_{L}$:
\begin{align*}
  \frac{b_L^2}{\lambda^{2L-1}} \langle  \mathcal{A}^{L-1} \partial_s(\chi_{B_0}T_{L}) ,  \mathcal{A}^{L-1}(\chi_{B_0}T_L) \rangle & = \frac{b_L^2}{2\lambda^{2L-1}} \partial_s \langle  \mathcal{A}^{L-1} (\chi_{B_0}T_{L}) ,  \mathcal{A}^{L-1}(\chi_{B_0}T_L) \rangle \\
  &=\frac{b_L^2}{2\lambda^{2L-2}} \partial_t{\left\lVert  \mathcal{A}^{L-1} (\chi_{B_0}T_{L})  \right\rVert}_{L^2}^2,
\end{align*}
by Lemma \eqref{lem:mod eqn}, we conclude \eqref{eq:L monotonicity}:
\begin{align*}
 & \frac{b_L^2}{2\lambda^{2L-2}} \partial_t{\left\lVert  \mathcal{A}^{L-1} (\chi_{B_0}T_{L})  \right\rVert}_{L^2}^2-\partial_t \left( \frac{b_L^2}{2\lambda^{2L-2}}{\left\lVert  \mathcal{A}^{L-1} (\chi_{B_0}T_{L})  \right\rVert}_{L^2}^2 \right)\\
  &= - \partial_t \left( \frac{b_L^2}{2\lambda^{2L-2}} \right){\left\lVert  \mathcal{A}^{L-1} (\chi_{B_0}T_{L})  \right\rVert}_{L^2}^2 \\
  &= \left( \frac{(L-1)b_L^2 \lambda_t}{\lambda^{2L-1}} -\frac{b_L (b_L)_t}{\lambda^{2L-2}} \right){\left\lVert  \mathcal{A}^{L-1} (\chi_{B_0}T_{L})  \right\rVert}_{L^2}^2 \\
  &=-\frac{b_L}{\lambda^{2L-1}} \left((b_L)_s + (L-1 )b_1b_L \right)  O\left(|\log b_1|^2\right) = O\left(\frac{b_1^{2L+1}}{\lambda^{2L-1}}|\log b_1|\right). \qedhere
\end{align*}
\end{proof}

\begin{proposition}[Lyapunov monotonicity for $\mathcal{E}_{L-1}$] 
  Let $\ell=L-1$. Then we have
  \begin{align}\label{eq:L-1 monotonicity}
    \frac{d}{dt}  \left\{  \frac{\widehat{\mathcal{E}}_{L-1}}{\lambda^{2L-4}} + O\left(\frac{b_1^{2L-2}|\log b_1|^2}{\lambda^{2L-4}}\right)  \right\}  \le  \frac{b_1^{L}|\log b_1|^{\delta}}{\lambda^{2L-3}} ( b_1^{L-1}|\log b_1|+\sqrt{\mathcal{E}_{L-1}}  ) 
  \end{align}
  where $0<\delta \ll 1$ is a sufficient small constant that depend only on $L$.
\end{proposition}
\begin{proof}
  Based on the proof of Proposition \ref{prop:L monotonicity} with Proposition \ref{prop:3rd local approx}, all the inner products we have to deal with are:
  \begin{align}
    &b_L\langle  \mathcal{A}^{L-1} (\chi_{B_1}-\chi_{B_0})T_{L-1} , \hat{\varepsilon}_{L-1} \rangle ,\quad b_{L-1}\langle  \mathcal{A}^{L-1}(\partial_s \chi_{B_0} + b_1(y\chi')_{B_0}) T_{L-1} , \hat{\varepsilon}_{L-1} \rangle\nonumber \\
    &b_{L-1}\langle  \mathcal{A}^{L-2}H (\chi_{B_1}-\chi_{B_0})  T_{L} , \dot{\hat{\varepsilon}}_{L-2} .\rangle,\quad b_L\langle  \mathcal{A}^{L-2} (\partial_s\chi_{B_0}+b_1(y\chi')_{B_0})T_{L} , \dot{\hat{\varepsilon}}_{L-2} \rangle. \nonumber
  \end{align}
  By additionally considering ${\hat{\varepsilon}}= {\varepsilon} + b_{L-1} (\chi_{B_1}-\chi_{B_0})T_{L-1}$, we can estimate the above inner products similarly to \eqref{eq:L F final} due to the derivative gain $\mathcal{A}^{L-2}H = \mathcal{A}^L$ and the logarithmic gain $|\log b_1|^{-\beta}$ from the bootstrap bound \eqref{eq:bootstrap modes stables} for $b_L$ when $\ell=L-1$. The exact correction term is given by
  \[-\partial_t \left( \frac{b_{L-1}^2}{2(L-2)\lambda^{2L-4}}{\left\lVert  \mathcal{A}^{L-1} (\chi_{B_0}T_{L-1})  \right\rVert}_{L^2}^2 \right).\qedhere\]
\end{proof}
\end{appendix}

\bibliographystyle{abbrv}
%\bibliography{my}

\end{document}